\documentclass{amsart}
\usepackage[utf8]{inputenc}
\usepackage[english]{babel}
\usepackage[pdftex]{graphicx}
\usepackage{amsmath,amsfonts}
\usepackage{color}
\usepackage[left=2.5cm,top=4cm,right=2.5cm,bottom=4cm]{geometry}
\usepackage[hidelinks]{hyperref}
\usepackage{fancyhdr}
\usepackage{subcaption}
\usepackage{mathtools}
\usepackage{caption}
\usepackage{relsize}
\usepackage{amssymb}
\usepackage{amsthm}
\usepackage{mathrsfs}
\usepackage[parfill]{parskip}
\usepackage{enumitem}
\usepackage{fancyhdr}
\usepackage{csquotes}
\usepackage{nicefrac}
\usepackage{bbm}
\usepackage[nameinlink,capitalize]{cleveref}

\usepackage{comment}

\usepackage{pgfplots}
\pgfplotsset{compat=1.18}
 \usepgfplotslibrary{groupplots}
\usepgfplotslibrary{fillbetween}

\usepackage{tikz}
\usetikzlibrary{
    arrows.meta,
    calc,
    positioning,
    decorations.pathreplacing,
    bending
}

\numberwithin{equation}{section}
\newtheorem{theorem}{Theorem}[section]
\newtheorem{corollary}[theorem]{Corollary}
\newtheorem{lemma}[theorem]{Lemma}
\newtheorem{proposition}[theorem]{Proposition}
\newtheorem{remark}[theorem]{Remark}
\theoremstyle{definition}
\newtheorem{definition}[theorem]{Definition}
\newtheorem{lemdef}[theorem]{Lemma/Definition}
\newtheorem{example}[theorem]{Example}

\theoremstyle{remark}

\usepackage[backend=biber,backref=true,style=numeric,giveninits,maxbibnames=20,url=false,isbn=false,doi=false]{biblatex}

\def \l {\left}
\def \r {\right}
\def \R {\mathbb{R}}
\def \C {\mathbb{C}}
\def \P {\mathbb{P}}
\def \F {\mathcal{F}}
\def \E {\mathbb{E}}
\def \N {\mathbb{N}}

\def \ind {\mathbbm{1}}
\newcommand{\calN}{\mathcal{N}}

\DeclareMathOperator{\tr}{tr}
\DeclareMathOperator{\id}{Id}

\def \law {\textup{Law}}

\def \Cov {\textup{Cov}}

\newcommand{\br}[1]{{\l(#1\r)}}
\newcommand{\sqbr}[1]{{\l[#1\r]}}

\newcommand{\dTV}[1]{{d_{\textup{TV}}\br{#1}}}

\author{Nawaf Bou-Rabee \and Sonja Cox \and Roy Schieven
}
\title{On couplings for kinetic Langevin diffusions}
\date{\today}

\begin{document}

\begin{abstract}
For the kinetic Langevin diffusion and its splitting discretizations, the hypoelliptic noise structure makes the relationship between couplings and total variation (TV) bounds more subtle than in the elliptic case. We establish that, for the kinetic Langevin equation with quadratic
potential, no Markovian coupling (continuous or discrete) captures the
asymptotic decay rate of the TV distance between two solutions with
different initial values; the canonical iterated one-shot (or sticky) coupling, for
which we derive an exact contraction formula, saturates this lower bound. On the constructive side, we show that the recent sharp TV bounds of \citeauthor{ChakMonmarche:2025} admit a natural interpretation through an explicit non-Markovian coupling, built from an optimal coalescence trajectory characterized by a classical minimum-energy control problem. For the OBABO splitting scheme, this approach additionally eliminates the Hessian-Lipschitz, step-size, and final-time assumptions of \cite{ChakMonmarche:2025}.
\end{abstract}

\maketitle

\section{Introduction}

The \emph{kinetic Langevin diffusion} is the $2d$-dimensional Markov process
$(X_t, V_t)_{t \geq 0}$ satisfying the stochastic differential equation
\begin{equation}
\label{eq: general kinetic Langevin}
    \begin{split}
        dX_t &= V_t\,dt,\\
        dV_t &= -\nabla U(X_t)\,dt-\gamma V_t\,dt+\sqrt{2\gamma}\,dW_t,
    \end{split}
\end{equation}
on the position--velocity space $\R^{2d}$, with continuously differentiable
potential $U\colon \R^d\to\R$, friction parameter $\gamma>0$, and standard
$d$-dimensional Brownian motion $(W_t)_{t\geq0}$.  Originating in
statistical mechanics as a model for a particle in the potential $U$
subject to frictional dissipation and thermal noise, it is a canonical
example of a hypoelliptic diffusion: noise enters directly only on the
velocity and propagates to the position through the kinematic coupling
$dX_t = V_t\,dt$.  When $e^{-U}$ is integrable,
the process $(X_t,V_t)_{t\geq 0}$ leaves the \emph{Boltzmann--Gibbs measure}
\begin{equation*}
    \mu = \mu_{\textup{target}}\otimes\mathcal{N}(0,\id_d),
    \qquad \mu_{\textup{target}}(dx)\propto e^{-U(x)}\,dx,
\end{equation*}
invariant, with $\mu_{\textup{target}}$ as its position marginal.  Under standard 
regularity conditions,  the law of $(X_t,V_t)_{t\geq 0}$ converges to $\mu$ in total variation (TV)
distance~\cite{Talay:2002,MattinglyEtAl:2002}; in particular, the law of
$X_t$ converges to $\mu_{\textup{target}}$.

Compared to its overdamped counterpart, the kinetic Langevin diffusion
exhibits a quantitative \emph{diffusive-to-ballistic} acceleration of
equilibration: under a Poincar\'e inequality with constant $m$ together
with mild non-convexity of $U$, the $L^2$ relaxation time of
critically-tuned kinetic Langevin dynamics is of order $m^{-1/2}$ rather than
$m^{-1}$~\cite{CaoLuWang:2023,EberleLoerler:2024}. This acceleration has
motivated using kinetic Langevin to construct
Markov chain Monte Carlo (MCMC) algorithms for sampling from
$\mu_{\textup{target}}$. To implement such an algorithm,
however, the dynamics must be discretized, and the sample
complexity of the resulting chain is governed by its TV mixing time. The present work concerns this TV mixing time, with particular attention to its dependence on the time discretization step size.

The TV distance admits the coupling characterization
$\dTV{\nu,\tilde\nu}=\inf_{(X,\tilde X)}\P(X\neq\tilde X)$, where the
infimum runs over couplings $(X,\tilde X)$ of $\nu$ and $\tilde\nu$.
Bounding TV distance therefore reduces to exhibiting a coupling under
which the two marginals coincide with high probability. For
\emph{elliptic} diffusions, in particular the overdamped Langevin
equation
\begin{equation}
\label{eq: general overdamped Langevin}
    dX_t = -\nabla U(X_t)\,dt + \sqrt{2}\,dW_t,
\end{equation}
which has $\mu_{\textup{target}}$ as invariant measure, reflection
couplings combined with concave distances yield Wasserstein and TV
contraction rates even when $U$ is non-convex~\cite{LindvallRogers:1986,Eberle:2016},
with corresponding concentration estimates for empirical
averages~\cite{JoulinOllivier:2010,Villani:2009}. These ideas extend to
infinite dimensions for stochastic partial differential equations driven
by non-degenerate space-time white noise, again without convexity
assumptions~\cite{DaPratoEtAl:2005}. In the discrete-time MCMC setting,
the same Markovian coupling arguments yield sharp TV mixing bounds for
discretizations of~\eqref{eq: general overdamped Langevin}~\cite{DurmusMoulines:2019}.

For diffusions driven by \emph{degenerate} (hypoelliptic) noise, such
as the kinetic Langevin diffusion~\eqref{eq: general kinetic Langevin},
in which noise enters only through the velocity, none of these
elliptic constructions transfer directly: the interplay between the
noise degeneracy and a possibly non-convex drift complicates coupling
design. \citeauthor{BenArousEtAl:1995}~\cite{BenArousEtAl:1995}
initiated the hypoelliptic coupling program with a hybrid
synchronous/reflection coupling for the stochastic oscillator with
generator $\tfrac{1}{2}\partial_v^2 + v\,\partial_x$, and
\citeauthor{EberleEtAl:2019}~\cite{EberleEtAl:2019} subsequently
established Wasserstein contraction
for~\eqref{eq: general kinetic Langevin} under gradient-Lipschitz
potentials by combining synchronous and reflection couplings in distinct
regions of phase space. Related ideas appear
in~\cite{Bou-RabeeEberle:2022} for Andersen dynamics --- a
piecewise-deterministic relative of~\eqref{eq: general kinetic Langevin}
in which diffusive forcing is replaced by periodic velocity
randomizations --- and in~\cite{BanerjeeKendall:2016} for the
Kolmogorov diffusion $(W_t,\int_0^t W_s\,ds)$, where the structural
limitations of Markovian couplings under degenerate noise were first
quantified.

Within this landscape, the present paper makes three contributions:
\begin{enumerate}[label=C.\arabic*]
    \item \label{cont: a coupling for CM} In
    Section~\ref{sec: CM coupling} we construct an explicit non-Markovian
    coupling whose meeting probability realizes the recent TV bound of
    \citeauthor{ChakMonmarche:2025}~\cite{ChakMonmarche:2025} for
    discretizations of~\eqref{eq: general kinetic Langevin}. The coupling
    is built from a \emph{coalescence map} associated with a classical
    minimum-energy control problem; for quadratic $U$, its meeting
    probability equals the TV distance exactly
    (Remark~\ref{rem: coalescence gives optimal TV for linear case}).
    
    \item \label{cont: no Markovian coupling} The coupling
    of~\ref{cont: a coupling for CM} is fundamentally non-Markovian. In
    Section~\ref{sec: non-optimality Markovian couplings} we show that
    this is unavoidable: for the kinetic Langevin diffusion with a
    quadratic potential $U(x)=\alpha|x|^2$, in both continuous and
    discrete time, the meeting probability of any Markovian coupling
    either fails to match the asymptotic decay rate of the TV distance
    between two initial conditions, or does so only at the cost of a
    multiplicative constant scaling like $1/h$. This extends to the
    kinetic Langevin setting the impossibility framework introduced
    by~\citeauthor{BanerjeeKendall:2016}~\cite{BanerjeeKendall:2016}
    for the Kolmogorov diffusion.
    
    \item \label{cont: iterated one-shot} For the canonical Markovian
    coupling --- the iterated one-shot (or sticky) coupling  ---  which yields
    sharp TV bounds for discretizations of the overdamped Langevin
    equation~\cite{DurmusMoulines:2019}, we derive in
    Section~\ref{sec: iterated one-shot coupling} a closed-form expression
    for the meeting probability in the linear Markov chain setting
    (Theorem~\ref{thm: iterated one-shot coupling probability}). Applied
    to the kinetic Langevin diffusion, this expression saturates the
    worse of the two regimes
    identified in~\ref{cont: no Markovian coupling}: the meeting
    probability captures the correct asymptotic rate in $T$, but only
    with a $1/h$ prefactor.
\end{enumerate}

The impossibility result of~\ref{cont: no Markovian coupling} explains why the non-Markovian construction of~\ref{cont: a coupling for CM} is unavoidable, while~\ref{cont: iterated one-shot} shows that the iterated one-shot coupling saturates this lower bound. Related work appears at the end of this introduction.

\subsection*{\ref{cont: a coupling for CM}: Interpretation of TV bound in \cite{ChakMonmarche:2025} as a non-Markovian coupling}

A key innovation of~\citeauthor{ChakMonmarche:2025}~\cite{ChakMonmarche:2025}
is a TV bound between two gHMC chains with different initial values over
a time horizon spanning multiple time steps; see~\cite[Theorem 2.2]{ChakMonmarche:2025}.
 The gHMC framework of~\cite{ChakMonmarche:2025} encompasses a broad class
of \emph{unadjusted} sampling methods: it includes unadjusted Hamiltonian
Monte Carlo and the OBABO splitting scheme for the kinetic Langevin equation \eqref{eq: general kinetic Langevin}, up to a velocity refreshment at the start and the end of the chain. The proof
of~\cite[Theorem 2.2]{ChakMonmarche:2025} treats OBABO before extending
to gHMC; specialized to OBABO, it gives:
\begin{theorem}[OBABO version of Theorem 2.1 of the supplement to \cite{ChakMonmarche:2025}]\label{thm: CM Wasserstein-to-TV regularization}
    Let $U\in C^2(\R^d)$ be twice continuously differentiable such that the Hessian of $U$ is globally Lipschitz continuous and bounded. Let $\pi_n(\delta_z)$ be the distribution after $n$ steps of the OBABO discretization of \eqref{eq: general kinetic Langevin} when starting at $z\in \R^{2d}$.  
    Then there exist $C,h_0>0$ such that for all $z,\tilde z\in \R^{2d}$, $h\in(0,h_0]$ and $n\in\N$ such that $hn\leq1$ we have
    \begin{equation}
    \label{eq: CMs TV bound OBABO}
        \dTV{\pi_n(\delta_z),\pi_n(\delta_{\tilde z})}\leq C\frac{h^{1/2}}{(1-e^{-\gamma h/2})^{1/2}}\br{\frac{e^{\gamma h/2}}{(hn)^{3/2}}+\frac{1-e^{-\gamma h/2}}{h}\frac{e^{\gamma h/2}}{(hn)^{1/2}}+(hn)^{1/2}}|\tilde z -z|.
    \end{equation}
\end{theorem}
Three aspects of Theorem~\ref{thm: CM Wasserstein-to-TV regularization} are worth noting.
First, the upper bound in~\eqref{eq: CMs TV bound OBABO} scales like
$T^{-3/2}$ for small terminal time $T\coloneqq hn$, which is optimal: the
position has variance of order $T^3$ on small time scales. Second,
combined with a Wasserstein bound, this yields mixing time bounds;
see~\cite[Theorems 2.1 and 2.2]{ChakMonmarche:2025}. Third, the proof
relies on a \emph{coalescence map} that transforms the noise driving the
gHMC chain started at $z$ into the noise driving the chain started at
$\tilde z$, in a way that forces the two chains to meet.

Our contribution~\ref{cont: a coupling for CM} comprises the following:
\begin{enumerate}[label=\ref{cont: a coupling for CM}\alph*]
\item We show that conditions the $hn\leq 1$ and $h\leq h_0$ in the theorem above can be dropped for the OBABO splitting scheme approximation of~\eqref{eq: general kinetic Langevin}, see Theorem~\ref{thm: TV bound via coalescence map} below. The same approach can be applied to other splittings schemes such as the BOAOB scheme, see Remark~\ref{rem: other splitting schemes}.
\item We identify an `optimal' coalescence map by relating its construction to a classical control problem, see Section~\ref{subsec: optimized trajectory}. Indeed, while the coalescence map constructed in~\cite{ChakMonmarche:2025} suffices for their purposes, our \emph{optimized} coalescence map yields a coupling that reproduces the \emph{exact TV distance} when $U$ is quadratic, see Remark~\ref{rem: coalescence gives optimal TV for linear case}.
\item For splitting-scheme discretizations
of~\eqref{eq: general kinetic Langevin} admitting an explicit coalescence
map (notably OBABO and BOAOB), this construction yields an explicit
coupling of the two chains whose meeting probability matches the TV
upper bound; see Section~\ref{sec: non-Markovian coupling}. The same
construction applies in principle to any gHMC scheme with a diffeomorphic
coalescence map (Remark~\ref{rem: other splitting schemes}), though an
explicit coupling requires the coalescence map to be known in closed form.
\end{enumerate}

\subsection*{\ref{cont: no Markovian coupling}: Lack of suitable Markovian couplings}
For the kinetic Langevin equation~\eqref{eq: general kinetic Langevin}
with isotropic quadratic potential $U(x) = \alpha |x|^2$, $\alpha \geq 0$,
no Markovian coupling captures the asymptotic decay rate of the TV
distance between the laws of two copies of the process with different
initial values. The associated drift matrix has eigenvalues
\begin{equation}
    \lambda_\pm=-\frac{\gamma}{2}\pm\frac{1}{2}\sqrt{\gamma^2-4\alpha};
\end{equation}
Let $(Z_t)_{t\geq 0}$ and $(\tilde Z_t)_{t\geq 0}$ be solutions
to~\eqref{eq: general kinetic Langevin} with this potential and initial
values $z, \tilde z \in \R^{2d}$, and set $Z_k^h \coloneqq Z_{hk}$,
$\tilde Z_k^h \coloneqq \tilde Z_{hk}$ for $h > 0$ and $k \in \N$. The
discrete-time Markov chains $(Z_k^h)_{k\in\N}$ and $(\tilde Z_k^h)_{k\in\N}$
are obtained by sampling exactly from the increment distribution.  The following is a condensed form of Theorem~\ref{thm: discrete Markovian couplings have the wrong asymptotic behavior}:

\begin{theorem}
\label{thm: discrete Markovian couplings have the wrong asymptotic behavior intro}
Assume that $\gamma^2>4\alpha\geq 0$. 
Then there exist $z, \tilde{z}\in \R^{2d}$ and a constant $C>0$ such that for all $h>0$, $k\in \N$ one has
\begin{equation}\label{eq:tv_dist_Gauss_disc intro}
    \dTV{\law(Z_k^h),\law(\tilde Z_k^h)} \leq C e^{\lambda_{-} hk} | \tilde{z} - z |.
\end{equation}
Conversely, for every $z\neq \tilde{z}\in \R^{2d}$, every $h>0$ and every Markovian coupling $\mu_h$ of $(Z_k^h,\tilde{Z}_k^h)_{k\in \N}$ there exist constants $k_{\mu_h}\in\N$, $c_{\mu_h} >0$ and $c>0$ (with $c$ independent of $\mu_h$) such that for all $k\geq k_{\mu_h}$ one has
\begin{equation}\label{eq:lb_couplingprob_disc intro}
    \mu_h(Z_k^h\neq \tilde{Z}_k^h) \geq c\min\big( c_{\mu_h}(hk+1)^{-\nicefrac{1}{2}}, c_{\mu_h} e^{\lambda_{+} hk},  h^{-1} e^{\lambda_{-} hk} \big).
\end{equation}
\end{theorem}

An analogous result hold in continuous time; see Theorem~\ref{thm: Markovian couplings have the wrong asymptotic behavior}.  The proof of these results is inspired by the work of \citeauthor{BanerjeeKendall:2016} \cite{BanerjeeKendall:2016}, where the authors prove a similar result for the Kolmogorov diffusion $Z_t = (W_t, \int_0^t W_s \,ds)$. 

\subsection*{\ref{cont: iterated one-shot}: Limitations of the iterated one-shot coupling}
In the setting of Theorem~\ref{thm: discrete Markovian couplings have the wrong asymptotic behavior intro} with $\alpha = 0$, the iterated one-shot coupling satisfies, for some
$c > 0$,
\begin{equation}\label{eq: bound iterated one-shot intro}
    \mu_h(Z_k^h \neq \tilde Z_k^h) \geq c\, h^{-1} e^{\lambda_- hk}.
\end{equation}
This bound saturates the third term in~\eqref{eq:lb_couplingprob_disc intro};
see Example~\ref{example: one shot matching lower bound}.

The iterated one-shot coupling was introduced for Gaussian random walks in~\cite{BubleyEtAl:1998} and was first applied directly to approximations of SDEs in \cite{RobertsRosenthal:2002}, where the term ``one-shot coupling" was coined. The coupling maximizes the probability of meeting at the
next step given the current position, and can be viewed as a
discretization of the reflection coupling for continuous processes
(see~\cite[Remark 2.6]{EberleMajka:2019} and~\cite{LindvallRogers:1986}). The terminology \emph{sticky coupling} in~\cite{DarshanEtAl:2024,DurmusEtAl:2024} refers to the same construction.

\citeauthor{DurmusMoulines:2019} used the iterated one-shot coupling to
obtain asymptotically optimal TV contraction rates for discretizations
of the overdamped Langevin equation~\eqref{eq: general overdamped Langevin}.
Adapting their argument, we derive an \emph{exact} formula for the
meeting probability of the iterated one-shot coupling on linear
discretized SDEs (Theorem~\ref{thm: iterated one-shot coupling probability}).
This formula yields the bound~\eqref{eq: bound iterated one-shot intro}
and identifies the conditions under which the iterated one-shot coupling
does and does not reproduce sharp TV bounds.

\subsection*{Related work}

The TV analysis of overdamped Langevin discretizations is well-developed.
Weak-error expansions for the Euler-Maruyama discretization go back
to~\cite{TalayTubaro:1990}; non-asymptotic TV bounds for the unadjusted
Langevin algorithm (ULA) were obtained by~\citeauthor{Dalalyan:2017}~\cite{Dalalyan:2017}
under strong convexity and sharpened
in~\cite{DurmusMoulines:2017,DurmusMoulines:2019}, and the
Metropolis-adjusted variant (MALA) eliminates the discretization bias at
the cost of an accept-reject step. The iterated one-shot coupling
in~\cite{DurmusMoulines:2019} is the canonical Markovian coupling for
these discretizations and the natural starting point for our analysis
in~\ref{cont: iterated one-shot}.

For kinetic Langevin discretizations, Wasserstein contraction under
strong convexity was established
by~\cite{ChengEtAl:2018,DalalyanRiou-Durand:2020} using synchronous
couplings, and in the non-strongly-convex case by~\cite{ChengEtAl:2020}
adapting the hybrid coupling of~\cite{EberleEtAl:2019} to discrete time.
TV bounds via Wasserstein-to-TV regularization were obtained for the
OBABO splitting in~\cite{Monmarche:2021} and for the larger generalized
HMC class in~\cite{GouraudEtAl:2025}; one-step Wasserstein contraction
rates for a broad family of splittings appear
in~\cite{LeimkuhlerEtAl:2024}, with related Wasserstein bounds for
further splittings in~\cite{SchuhWhalley:2025}. The TV bounds
of~\cite{Monmarche:2021,GouraudEtAl:2025} degenerate under vanishing
step size, motivating the multi-step Wasserstein-to-TV regularization
of~\cite{ChakMonmarche:2025}.  However, their argument 
does not  produce an explicit coupling between the noise
increments; we construct such a coupling in~\ref{cont: a coupling for CM}.

An alternative line of work analyzes TV mixing of kinetic Langevin
discretizations via functional inequalities~\cite{CamrudEtAl:2024,Lehec:2025,MaEtAl:2021,Monmarche:2024};
as argued in~\cite{ChakMonmarche:2025}, the coupling-based approach is
more flexible, accommodating in particular stochastic-gradient
approximations. Couplings have also been used extensively for
Hamiltonian Monte Carlo and its variants~\cite{Bou-RabeeEtAl:2020,Bou-RabeeEberle:2023,Bou-RabeeSchuh:2023}.

\section{Preliminaries and notation}

Throughout the paper, $\N = \{0, 1, 2, \ldots\}$. For a measurable space
$(X, \F)$ and $x \in X$, $\delta_x$ denotes the Dirac measure at $x$.

The total variation (TV) distance between two probability measures $\nu$ and $\tilde\nu$ on a countably separated measurable space $(X,\F)$ is 
\begin{equation*}
    \dTV{\nu,\tilde\nu}=\sup_{A\in\F}\{|\nu(A)-\tilde\nu(A)|\}.
\end{equation*}
Let $\Gamma(\nu,\tilde\nu)$ denote the set of couplings of $\nu$ and $\tilde\nu$, i.e.\ the set of all probability measures on $(X\times X, \F \otimes \F)$ with marginals $\nu$ and $\tilde\nu$. The TV distance admits the
coupling characterization~ \cite[Sections I.2 \& I.5]{Lindvall:2002} 
\begin{equation}
\label{eq: TV coupling characterization}
    \dTV{\nu,\tilde\nu}=\min_{\gamma\in\Gamma(\nu,\tilde\nu)}\{\gamma(\{(x,y) \in X \times X \colon x \neq y \})\}.
\end{equation}
A coupling $\gamma\in \Gamma(\nu,\tilde\nu)$ that attains this minimum is called \emph{maximal} (since it maximizes $\gamma(\{(x,y) \in X \times X \colon x = y \})$); maximal couplings exist for any two probability measures on a countably separated measurable space \cite[Section I.5]{Lindvall:2002} and can be constructed by
rejection sampling; see Section~\ref{sssec: reflection coupling} for
the Gaussian case used in this paper.

For two Gaussians on $\R^d$ with common non-singular covariance
$\Sigma$ and (possibly distinct) means $\mu, \tilde\mu$,
\begin{equation}
\label{eq: Gaussian TV distance}
    \dTV{\mathcal{N}(\mu,\Sigma),\mathcal{N}(\tilde\mu,\Sigma)}= 2\Phi\br{\frac{\l|B^{-1}(\tilde\mu-\mu)\r|}{2}}-1,
\end{equation}
where $B\in\R^{d\times d}$ satisfies $BB^T=\Sigma$, and $\Phi\colon\R\to\R$ is the standard normal CDF~\cite[Theorem 1]{BarsovUlyanov:1987}.

Throughout, $|\cdot|$ denotes the Euclidean norm on $\R^d$, and also the
induced norm on $\R^{2d}$ under the splitting $z=(x,v)$ used below; for
$A\in\R^{d\times d}$, $\|A\|$ denotes the operator norm and $\|A\|_F$ the
Frobenius norm. We write $\law(X)$ for the law of a random variable $X$.
For probability measures $\nu,\tilde\nu$ on $\R^d$, the $L^p$-Wasserstein
distance of order $p\in[1,\infty)$ is
\[
  W_p(\nu,\tilde\nu)
  \;=\;
  \Bigl(\,\inf_{\gamma\in\Gamma(\nu,\tilde\nu)}
  \int |x-\tilde x|^p\,\gamma(dx,d\tilde x)\Bigr)^{1/p}.
\]

A recurring theme of this paper is the conversion of Wasserstein control
between two distributions into total variation control between their
evolutions under a Markov chain. We refer to estimates of this form as
\emph{Wasserstein-to-TV regularization}. Let $P$ be a Markov kernel on
$\R^{2d}$ and let $\nu, \tilde\nu$ be probability measures on $\R^{2d}$;
write $\nu P^n$ for the law at time $n$ of the chain with kernel $P$ and
initial distribution $\nu$. A Wasserstein-to-TV regularization bound for
$P$ is an inequality of the form
\begin{equation}\label{eq: W-to-TV schematic}
  \dTV{\nu P^n,\,\tilde\nu P^n}
  \;\le\;
  C_{n}\,W_p(\nu,\tilde\nu),
  \qquad n\ge 1,
\end{equation}
for some $p\in[1,\infty)$ and a constant $C_n$ that depends on $P$ but not
on $\nu, \tilde\nu$. The two distances on the right- and left-hand sides
are not comparable in general: at $n=0$, two distinct Dirac masses
$\delta_z, \delta_{\tilde z}$ satisfy
$\dTV{\delta_z,\delta_{\tilde z}}=1$ while
$W_p(\delta_z,\delta_{\tilde z})=|z-\tilde z|$, so~\eqref{eq: W-to-TV schematic}
fails uniformly at $n=0$. The content of the bound is that the kernel $P$
regularizes the laws sufficiently in $n$ steps that a transport-based
control of the initial distance produces a statistical-distance control of
the resulting laws.  Equivalently, for every $n\ge 1$, the map
$\mu\mapsto\mu P^n$ is $C_n$-Lipschitz from
$(\mathcal P_p(\R^{2d}),\,W_p)$ to
$(\mathcal P(\R^{2d}),\,\dTV{\cdot,\cdot})$, whereas the identity map
between these two metric spaces is not even continuous, as the Dirac-mass
example above shows.

\section{The non-Markovian finite-time coupling}\label{sec: CM coupling}
The starting point for this section is the recent work of
\citeauthor{ChakMonmarche:2025}~\cite{ChakMonmarche:2025}, who establish
Wasserstein-to-TV regularization for a broad class of generalized
Hamiltonian Monte Carlo (gHMC) Markov chains. The OBABO discretization of
the kinetic Langevin equation, with step size $h>0$, generates a chain in
this class; we review the OBABO scheme in
Section~\ref{subsec: OBABO scheme} and specialize to it throughout the
remainder of this section. The regularization mechanism is necessarily
multi-step: the position component of a single step of the OBABO chain has
variance of order $h^3$, too degenerate to support a Gaussian coupling on
its own. To circumvent this, \cite{ChakMonmarche:2025} introduce a
\emph{coalescence map}, a measurable transport of the Gaussian increments
driving the chain from $z$ onto Gaussian increments driving the chain from
$\tilde z$; controlling the deviation of this map from the identity over
$n$ steps yields the bound of
Theorem~\ref{thm: CM Wasserstein-to-TV regularization}.

This section refines that analysis in three respects. 
First, we relax the hypotheses of Theorem~\ref{thm: CM Wasserstein-to-TV regularization}: the
step-size restrictions $hn\le 1$ and $h\le h_0$ are removed for the OBABO
chain (Theorem~\ref{thm: TV bound via coalescence map}), so that the bound
holds for an arbitrary number of integration steps $n\in\N$ and arbitrary
step sizes $h>0$, and the Lipschitz-Hessian assumption on $U$ in
\cite{ChakMonmarche:2025} is dropped, retaining only the Lipschitz-gradient
assumption.  Second, we
identify the coalescence map that is optimal for the linearized dynamics,
as the unique minimizer of a discrete linear-quadratic control problem
(Section~\ref{subsec: optimized trajectory}); this map differs from the
choice in~\cite{ChakMonmarche:2025}. The proofs of the first two
refinements share a common structural ingredient: the Jacobian of the
coalescence map associated with the OBABO scheme is lower triangular
(Lemma~\ref{lemma: trajectory based coalescene maps are differentiable}),
which makes both the inversion required for the optimal map and the bounds
needed for arbitrary $(h,n)$ tractable. Several other splitting schemes
share this lower-triangular structure
(Remark~\ref{rem: other splitting schemes}). Third, we show that the
resulting analysis admits a coupling interpretation: it defines an explicit
non-Markovian coupling of two OBABO chains started from different initial
values (Section~\ref{sec: non-Markovian coupling}), whose meeting
probability coincides with the TV bound. When $U$ is quadratic, this
coupling is maximal
(Remark~\ref{rem: coalescence gives optimal TV for linear case}).

\subsection{A TV distance bound for the OBABO scheme}

\label{subsec: OBABO scheme} 

We refine Theorem~\ref{thm: CM Wasserstein-to-TV regularization} for the
OBABO discretization
(Theorem~\ref{thm: TV bound via coalescence map} below). Combined with a
Wasserstein contraction estimate, as in~\cite{ChakMonmarche:2025}, this
bound yields convergence in total variation
(Remark~\ref{rem:how to obtain convergence in the TV distance}).

Fix the step size $h>0$ and the friction coefficient $\gamma>0$.
 The kinetic Langevin equation,
\begin{equation}
\label{eq: general kinetic Langevin CM}
    \begin{split}
        dX_t&=V_t\,dt,\\
        dV_t&=-\nabla U(X_t)\,dt-\gamma V_t\,dt+\sqrt{2\gamma}\,dW_t,
    \end{split}
\end{equation}
admits the splitting into Ornstein--Uhlenbeck, potential, and kinetic
parts, with time-$h$ flows
\begin{align*}
    O_h(a)(x,v)
    &=\bigl(x,\;e^{-\gamma h}v+(1-e^{-2\gamma h})^{1/2}a\bigr),
    && a\sim\calN(0,\id_d),\\
    \theta_h^{(\textup{B})}(x,v)
    &=\bigl(x,\;v-h\,\nabla U(x)\bigr),\\
    \theta_h^{(\textup{A})}(x,v)
    &=\bigl(x+hv,\;v\bigr).
\end{align*}
The Ornstein--Uhlenbeck flow $O_h$ is exact in law: it returns the
distribution at time $h$ of the OU SDE
$dV_t=-\gamma V_t\,dt+\sqrt{2\gamma}\,dW_t$ started from $v$, expressed as
a deterministic function of the Gaussian increment $a$. The flows
$\theta^{(\textup{A})}_h$ and $\theta^{(\textup{B})}_h$ are exact in
the pathwise sense.

The OBABO scheme is the Strang composition
\begin{equation}
\label{eq: OBABO scheme}
    O_{h/2}(\xi^{(2)})\circ\theta_{h/2}^{(\textup{B})}
    \circ\theta_h^{(\textup{A})}
    \circ\theta_{h/2}^{(\textup{B})}\circ O_{h/2}(\xi^{(1)}),
\end{equation}
with $\xi^{(1)},\xi^{(2)}\stackrel{\text{iid}}{\sim}\calN(0,\id_d)$. Setting
$\theta_h:=\theta_{h/2}^{(\textup{B})}\circ\theta_h^{(\textup{A})}
\circ\theta_{h/2}^{(\textup{B})}$, the OBABO map is
$O_{h/2}(\xi^{(2)})\circ\theta_h\circ O_{h/2}(\xi^{(1)})$.

The associated Markov chain
$(Z^h_k)_{k\in\N}=(X^h_k,V^h_k)_{k\in\N}$ is defined by $Z^h_0=z$ and
\begin{equation*}
    Z^h_{k+1}
    =O_{h/2}(\xi_{k+1}^{(2)})\circ\theta_h\circ O_{h/2}(\xi_{k+1}^{(1)})
    (Z^h_k),
\end{equation*}
with $\xi_k=(\xi^{(1)}_k,\xi^{(2)}_k)$ iid copies of
$(\xi^{(1)},\xi^{(2)})$. Equivalently,
\begin{equation}
\label{eq: OBABO as linear Markov chain}
    Z^h_{k+1}
    =A_h\,Z^h_k+L_h\,\xi_{k+1}
    -\frac{h}{2}\begin{pmatrix}
        h\,\nabla U(X^h_k)\\
        e^{-\gamma h/2}\,\bigl[\nabla U(X^h_k)+\nabla U(X^h_{k+1})\bigr]
    \end{pmatrix},
\end{equation}
where
\begin{equation}
\label{eq: OBABO matrices}
    A_h=\begin{pmatrix}
        \id_d & h\,e^{-\gamma h/2}\id_d \\
        0     & e^{-\gamma h}\id_d
    \end{pmatrix},
    \qquad
    L_h=(1-e^{-\gamma h})^{1/2}
    \begin{pmatrix}
        h\,\id_d              & 0       \\
        e^{-\gamma h/2}\id_d  & \id_d
    \end{pmatrix}.
\end{equation}

A single OBABO step, viewed as a function of its driving noise, defines a
map $\Psi_z\colon\R^{2d}\to\R^{2d}$ for each $z\in\R^{2d}$:
\begin{equation}
\label{eq: definition Psi_z}
    \Psi_z(\xi)
    =O_{h/2}(\xi^{(2)})\circ\theta_h\circ O_{h/2}(\xi^{(1)})(z),
    \qquad
    \xi=(\xi^{(1)},\xi^{(2)})\in\R^{2d},
\end{equation}
so that $Z^h_{k+1}=\Psi_{Z^h_k}(\xi_{k+1})$. Writing
$\Psi_z=(\Psi^{X}_z,\Psi^{V}_z)$ for the position and velocity components
and $z=(x,v)$,
\begin{equation}
\label{eq: Psi_z explicit}
    \begin{aligned}
        \Psi^{X}_z(\xi^{(1)})
        &=x+h\,e^{-\gamma h/2}v
        +(1-e^{-\gamma h})^{1/2}h\,\xi^{(1)}
        -\tfrac{h^2}{2}\nabla U(x),\\
        \Psi^{V}_z(\xi^{(1)},\xi^{(2)})
        &=e^{-\gamma h}v
        +(1-e^{-\gamma h})^{1/2}\bigl(e^{-\gamma h/2}\xi^{(1)}+\xi^{(2)}\bigr)
        -\tfrac{h}{2}\,e^{-\gamma h/2}
        \bigl(\nabla U(x)+\nabla U(\Psi^{X}_z(\xi^{(1)}))\bigr).
    \end{aligned}
\end{equation}
The map $\Psi_z$, like the
matrices $A_h$ and $L_h$ above, depends on $h$ and $\gamma$ through the
OBABO scheme; we suppress this dependence in the notation throughout,
retaining $h$ as an explicit subscript only on $A_h$ and $L_h$.

\begin{lemma}\label{lemma: Psi_z is a diffeomorphism}
Suppose $U\in C^2(\R^d)$. For every $z\in\R^{2d}$, the map
$\Psi_z\colon\R^{2d}\to\R^{2d}$ is a $C^1$ diffeomorphism.
\end{lemma}

\begin{proof}
By~\eqref{eq: Psi_z explicit}, $\xi^{(1)}\mapsto\Psi^{X}_z(\xi^{(1)})$ is
affine with leading coefficient $h(1-e^{-\gamma h})^{1/2}\neq 0$, hence a
$C^\infty$ bijection of $\R^d$. Given $\xi^{(1)}$,
$\xi^{(2)}\mapsto\Psi^{V}_z(\xi^{(1)},\xi^{(2)})$ is affine with leading
coefficient $(1-e^{-\gamma h})^{1/2}\neq 0$, hence a $C^\infty$ bijection
of $\R^d$. Solving sequentially recovers $(\xi^{(1)},\xi^{(2)})$ from
$\Psi_z(\xi)$, so $\Psi_z$ is a bijection. The dependence of $\Psi^{V}_z$
on $\xi^{(1)}$ through $\nabla U(\Psi^{X}_z(\xi^{(1)}))$ is $C^1$ since
$U\in C^2$; hence both $\Psi_z$ and $\Psi_z^{-1}$ are $C^1$.
\end{proof}

For $n\ge 1$, define $\Psi^n_z\colon\R^{2dn}\to\R^{2d}$ inductively by
$\Psi^1_z=\Psi_z$ and
\begin{equation}
\label{eq: definition Psi^n_z}
    \Psi_z^{n+1}(\xi_1,\dots,\xi_{n+1})
    =\Psi_{\Psi^{n}_z(\xi_1,\dots,\xi_{n})}(\xi_{n+1}),
\end{equation}
with the convention $\Psi^0_z\equiv z$. Then
$Z^h_n=\Psi^n_z(\xi_1,\dots,\xi_n)$ for every $n\in\N$.

The proof of Theorem~\ref{thm: TV bound via coalescence map} constructs a
coupling of two OBABO chains~\eqref{eq: OBABO as linear Markov chain} started
from $z$ and $\tilde z\in\R^{2d}$. Since each chain is a deterministic
function of its driving noise, the coupling is constructed at the level of
the iid sequences $\xi=(\xi_1,\xi_2,\ldots, \xi_n)$ and
$\tilde\xi=(\tilde\xi_1,\tilde\xi_2,\ldots, \tilde\xi_n)$, with $\xi_k,\tilde\xi_k \sim
\calN(0,\id_{2d})$ for each $k$.  For $n\in\N$ and $z\in\R^{2d}$, write
\begin{equation}\label{eq: def pi_n}
    \pi_n(\delta_z) := \law\bigl(\Psi^n_z(\xi)\bigr),
    \qquad \xi\sim\calN(0,\id_{2dn}),
\end{equation}
so that $\pi_n(\delta_z)=\law(Z^h_n)$ when $Z^h_0=z$.

\begin{theorem}\label{thm: TV bound via coalescence map}
Let $\gamma,h>0$ and $n\in\N$, and assume that $U\colon\R^d\to\R$ is twice
continuously differentiable with $L$-Lipschitz gradient. Then, for all
$z,\tilde z\in\R^{2d}$,
\begin{equation*}
    \dTV{\pi_n(\delta_z),\,\pi_n(\delta_{\tilde z})}
    \;\le\;
    \frac{1}{\gamma^{1/2}}\left[
        \frac{5}{(hn)^{3/2}}
        +\frac{12+5\gamma}{(hn)^{1/2}}
        +\frac{L(\gamma h)^{1/2}}{(1-e^{-\gamma h})^{1/2}}
        \left(1+\frac{hn}{1+\gamma hn}\right)(hn)^{1/2}
    \right]|\tilde z-z|.
\end{equation*}
\end{theorem}

Theorem~\ref{thm: TV bound via coalescence map} improves on
Theorem~\ref{thm: CM Wasserstein-to-TV regularization} in three respects.
First, we identify an explicit coalescence map for which the
corresponding coupling of OBABO chains is maximal in the force-free
case. Second, the bound holds for arbitrary $h>0$ and $n\in\N$, in
contrast to the step-size restrictions $hn\le 1$ and $h\le h_0$
required by Theorem~\ref{thm: CM Wasserstein-to-TV regularization}.
Third, the regularity hypothesis on $U$ weakens to $\nabla U$ Lipschitz,
the Lipschitz-Hessian assumption of
Theorem~\ref{thm: CM Wasserstein-to-TV regularization} being no longer
needed. The first improvement is achieved by an optimization argument
(Section~\ref{subsec: optimized trajectory}); the second and third both
follow from a structural feature of the OBABO scheme established in
Section~\ref{subsec: coalescence maps}: the Jacobian of the associated
coalescence map is lower triangular with identity blocks on the
diagonal, so the trace and log-determinant terms in the underlying
Pinsker-type bound vanish identically.

We outline the proof here; the full argument is given in
Sections~\ref{subsec: coalescence maps}
and~\ref{subsec: optimized trajectory}.

Following~\cite{ChakMonmarche:2025}, the proof proceeds via a
\emph{coalescence map}: a map
$\Psi^n_{z,\tilde z}\colon\R^{2dn}\to\R^{2dn}$ that transports the noise
driving the OBABO chain from $z$ to noise driving the chain from $\tilde z$,
in such a way that the two trajectories coincide at time $n$. Given any deterministic difference trajectory $(y_k)_{k=0}^{n}$ in $\R^{2d}$
with $y_0=\tilde z-z$ and $y_n=0$, we construct such a map and show that it
is a diffeomorphism of $\R^{2dn}$
(Section~\ref{subsec: coalescence maps}; see also
Figure~\ref{fig:finite-time-coalescence-map}).  The
bound of Theorem~\ref{thm: TV bound via coalescence map} then follows from
standard estimates for the TV distance between the law of a Gaussian and
that of its image under a diffeomorphism; for these estimates we follow
\cite{Monmarche:2024}.

The remaining question is which difference trajectory $(y_k)$ to choose. In
\cite{ChakMonmarche:2025} the choice is made by ansatz. In
Section~\ref{subsec: optimized trajectory} we identify the trajectory that
is optimal in the linearized regime: when $U$ is quadratic, the resulting
coalescence map induces a coupling whose meeting probability is maximal at
time $n$ (Remark~\ref{rem: coalescence gives optimal TV for linear case}).
This trajectory is characterized as the unique solution of a discrete
linear-quadratic control problem. A direct comparison with the choice of
\cite{ChakMonmarche:2025} is given in Remark~\ref{remark: CM trajectory}.

\begin{figure}[t]
\centering
\begin{tikzpicture}[
    scale=1.0,
    every node/.style={font=\small},
    pt/.style={circle, fill=black, inner sep=1.1pt}
]
  \draw[thick] (0,0) -- (1.0,0.25) -- (2.0,0.10) -- (3.0,0.22)
                    -- (4.0,0.50) -- (5.0,0.72) -- (6.0,1.00);
  \foreach \p in {(0,0),(1.0,0.25),(2.0,0.10),(3.0,0.22),
                  (4.0,0.50),(5.0,0.72),(6.0,1.00)}
    \node[pt] at \p {};

  \draw[thick, dashed] (0,2.00) -- (1.0,1.78) -- (2.0,1.92) -- (3.0,1.40)
                              -- (4.0,1.55) -- (5.0,1.18) -- (6.0,1.00);
  \foreach \p in {(0,2.00),(1.0,1.78),(2.0,1.92),(3.0,1.40),
                  (4.0,1.55),(5.0,1.18)}
    \node[pt] at \p {};

  \draw[<->, semithick, gray!65!black]
       (0,0.05) -- node[left, font=\footnotesize, text=gray!65!black] {$y_0$} (0,1.95);
  \draw[<->, semithick, gray!65!black]
       (3.0,0.27) -- node[right, font=\footnotesize, text=gray!65!black] {$y_k$} (3.0,1.35);

  \node[below, font=\footnotesize] at (0,0)   {$z$};
  \node[above, font=\footnotesize] at (0,2.0) {$\tilde z$};
  \node[right, font=\footnotesize] at (6.0,1.00) {$Z^h_n=\tilde Z^h_n$};
\end{tikzpicture}
\caption{The coalescence map $\Psi^n_{z,\tilde z}$ is parameterized by a
deterministic difference trajectory $(y_k)_{k=0}^{n}$ in $\R^{2d}$ with
$y_0=\tilde z-z$ and $y_n=0$; the trajectory forces the chain $(Z^h_k)$
started at $z$ (solid) and the chain $(\tilde Z^h_k)$ started at $\tilde z$
(dashed) to coincide at time $n$.}
\label{fig:finite-time-coalescence-map}
\end{figure}
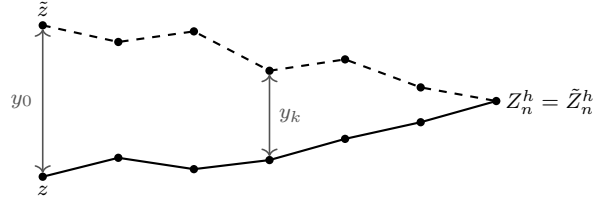

A standard conditioning argument extends
Theorem~\ref{thm: TV bound via coalescence map} to general initial
distributions, yielding the following Wasserstein-to-TV regularization
(cf.~\cite[Corollary~2.2 in the supplementary material]{ChakMonmarche:2025}).

\medskip
\begin{corollary}
\label{cor: Wasserstein-to-TV regularization}
Under the assumptions of
Theorem~\ref{thm: TV bound via coalescence map}, for any probability
measures $\nu,\tilde\nu$ on $\R^{2d}$,
\begin{equation*}
    \dTV{\pi_n(\nu),\,\pi_n(\tilde\nu)}
    \;\le\;
    \frac{1}{\gamma^{1/2}}\biggl[
        \frac{5}{(hn)^{3/2}}
        +\frac{12+5\gamma}{(hn)^{1/2}}
        +\frac{L(\gamma h)^{1/2}}{(1-e^{-\gamma h})^{1/2}}
        \biggl(1+\frac{hn}{1+\gamma hn}\biggr)(hn)^{1/2}
    \biggr]\mathcal{W}_1(\nu,\tilde\nu).
\end{equation*}
\end{corollary}

\begin{proof}
Let $(Z_0,\tilde Z_0)$ be any coupling of $\nu$ and $\tilde\nu$, and set
$\F_0:=\sigma(Z_0,\tilde Z_0)$. Conditional on $\F_0$, let $(Z,\tilde Z)$
be a maximal coupling of $\pi_n(\delta_{Z_0})$ and
$\pi_n(\delta_{\tilde Z_0})$; such a coupling exists since $\R^{2d}$ is a
standard Borel space. Integrating against $\F_0$ shows that $(Z,\tilde Z)$
is a coupling of $\pi_n(\nu)$ and $\pi_n(\tilde\nu)$, so
\begin{equation*}
    \dTV{\pi_n(\nu),\,\pi_n(\tilde\nu)}
    \;\le\;\P(Z\neq\tilde Z)
    \;=\;\E\bigl[\dTV{\pi_n(\delta_{Z_0}),\,\pi_n(\delta_{\tilde Z_0})}\bigr]
    \;\le\;C\,\E\bigl[|\tilde Z_0-Z_0|\bigr],
\end{equation*}
where $C$ denotes the bracketed constant in the statement and the last
inequality is
Theorem~\ref{thm: TV bound via coalescence map}. Taking the infimum over
couplings $(Z_0,\tilde Z_0)$ of $(\nu,\tilde\nu)$ on the right-hand side
gives $\mathcal{W}_1(\nu,\tilde\nu)$ and concludes the proof.
\end{proof}

\begin{remark}\label{rem:how to obtain convergence in the TV distance}
Corollary~\ref{cor: Wasserstein-to-TV regularization} is a Wasserstein-to-TV
regularization bound for the OBABO chain in terms of the total simulation
time $T:=hn$. For small $T$, the bound scales like $T^{-3/2}$, matching
the small-time behavior of the continuous kinetic Langevin semigroup
under hypoelliptic regularization. The $T$-dependence at larger times
comes from the trajectory chosen in
Section~\ref{subsec: optimized trajectory}, which is optimal only for
the linearized dynamics; the nonlinear force $\nabla U$ contributes an
additional term scaling with $T^{1/2}$. Combined with a Wasserstein convergence estimate,
Corollary~\ref{cor: Wasserstein-to-TV regularization} yields TV
convergence for the chain. Such estimates are available for the
OBABO scheme in~\cite[Theorem~2.1]{ChakMonmarche:2025} and for the UBU
scheme in~\cite{SchuhWhalley:2025}.
\end{remark}

\begin{remark}
The TV bound provided in Theorem~\ref{thm: TV bound via coalescence map} can also be obtained for the BOAOB splitting scheme (which does not fall under the gHMC framework discussed in~\cite{ChakMonmarche:2025}). Moreover, similar results can presumably be obtained for other splitting schemes. For details, see Remark~\ref{rem: other splitting schemes}.
\end{remark}

\subsection{Coalescence maps}
\label{subsec: coalescence maps}

Fix initial values $z,\tilde z\in\R^{2d}$, and recall from
\eqref{eq: definition Psi^n_z} that $\Psi^n_z\colon\R^{2dn}\to\R^{2d}$
sends a noise vector $\xi=(\xi_1,\dots,\xi_n)$ to the state of the OBABO
chain~\eqref{eq: OBABO as linear Markov chain} after $n$ steps, started
from $z$ with driving noise $\xi$. Two such chains, one started from $z$
with noise $\xi$ and the other from $\tilde z$ with noise $\tilde\xi$,
coincide at time $n$ if and only if
\begin{equation}\label{eq: coalescence condition}
    \Psi^n_{\tilde z}(\tilde\xi)=\Psi^n_z(\xi).
\end{equation}

\medskip
\begin{definition}\label{def: coalescence map}
Let $z,\tilde z\in\R^{2d}$ and $n\in\N_{>0}$. A map
$\Psi^n_{z,\tilde z}\colon\R^{2dn}\to\R^{2dn}$ is a \emph{coalescence map}
if \eqref{eq: coalescence condition} holds with $\tilde\xi = \Psi^n_{z,\tilde z}(\xi)$,
for every $\xi\in\R^{2dn}$.
\end{definition}

Every coalescence map yields an upper bound on
$\dTV{\pi_n(\delta_z),\pi_n(\delta_{\tilde z})}$, a consequence of the contractivity of TV under measurable maps (cf.~\cite[Lemma~3]{MadrasSezer:2010}).
\medskip

\begin{lemma}\label{lemma: coalescence maps provide TV distance bounds}
Let $z,\tilde z\in\R^{2d}$, $n\in\N_{>0}$, and let $\Psi^n_{z,\tilde z}$
be a coalescence map. Then, for $\xi\sim\calN(0,\id_{2dn})$,
\begin{equation}\label{eq: coalescence map bounds TV distance}
    \dTV{\pi_n(\delta_z),\,\pi_n(\delta_{\tilde z})}
    \;\le\;
    \dTV{\law(\xi),\,\law(\Psi^n_{z,\tilde z}(\xi))}.
\end{equation}
\end{lemma}

\begin{proof}
Let $(\tilde\xi,\tilde\zeta)$ be a maximal coupling of $\law(\xi)$ and $\law(\Psi^n_{z,\tilde z}(\xi))$.
    Applying the measurable map
$\Psi^n_{\tilde z}$ componentwise yields the pair
$(\Psi^n_{\tilde z}(\tilde\xi),\Psi^n_{\tilde z}(\tilde\zeta))$, whose
marginals are
\begin{align*}
    \law(\Psi^n_{\tilde z}(\tilde\xi))
    &=\pi_n(\delta_{\tilde z}),\\
    \law(\Psi^n_{\tilde z}(\tilde\zeta))
    &=\law\bigl(\Psi^n_{\tilde z}(\Psi^n_{z,\tilde z}(\xi))\bigr)
    =\law(\Psi^n_z(\xi))
    =\pi_n(\delta_z),
\end{align*}
the second by the coalescence map property.  Hence
$(\Psi^n_{\tilde z}(\tilde\xi),\Psi^n_{\tilde z}(\tilde\zeta))$ is a
coupling of $\pi_n(\delta_{\tilde z})$ and $\pi_n(\delta_z)$, and
\begin{equation*}
    \dTV{\pi_n(\delta_z),\pi_n(\delta_{\tilde z})}
    \;\le\;\P\bigl(\Psi^n_{\tilde z}(\tilde\xi)\neq\Psi^n_{\tilde z}(\tilde\zeta)\bigr)
    \;\le\;\P(\tilde\xi\neq\tilde\zeta)
    \;=\;\dTV{\law(\xi),\,\law(\Psi^n_{z,\tilde z}(\xi))},
\end{equation*}
as required.
\end{proof}

The right-hand side of~\eqref{eq: coalescence map bounds TV distance}
admits sharp estimates via standard techniques for the TV distance between
a Gaussian and its image under a smooth perturbation. Combining
Lemma~\ref{lemma: coalescence maps provide TV distance bounds} with
Pinsker's inequality and the proof of~\cite[Lemma~15]{Bou-RabeeEberle:2023},
we obtain the following.

\begin{lemma}\label{lem: derivative coalescence map bounds TV distance}
Let $z,\tilde z\in\R^{2d}$, $n\in\N_{>0}$, and let $\Psi^n_{z,\tilde z}$
be a diffeomorphic coalescence map. Then, for
$\xi\sim\calN(0,\id_{2dn})$,
    \begin{equation}
    \label{eq: derivative coalescence map bounds TV distance}
    \dTV{\pi_n(\delta_z),\pi_n(\delta_{\tilde z})}\leq
    \tfrac{1}{2}\left(\E\sqbr{|\Psi^n_{z,\tilde z}(\xi)-\xi|^2+2\tr(D\Psi^n_{z,\tilde z}(\xi)-\id_{2dn})-2\log\left(|\det D\Psi^n_{z,\tilde z}(\xi)|\right)} \right)^{\nicefrac{1}{2}}.
\end{equation}
\end{lemma}

We turn to the construction of coalescence maps. Following~\cite{ChakMonmarche:2025},
a coalescence map can be specified through the difference trajectory of
the two chains; the construction below is parameterized by an arbitrary
$y\in\R^{2dn}$ with $y_n=0$.

\begin{lemdef}\label{lemdef: coalescence map defined by trajectory}
Let $z,\tilde z\in\R^{2d}$, $n\in\N_{>0}$, and let
$y=(y_1,\dots,y_n)\in\R^{2dn}$ satisfy $y_n=0$ (see Figure~\ref{fig:finite-time-coalescence-map}). 
 There exists a unique
coalescence map $\Psi^n_{z,\tilde z}\colon\R^{2dn}\to\R^{2dn}$ such that,
for every $\xi=(\xi_1,\dots,\xi_n)\in\R^{2dn}$, the vector
$\tilde\zeta=(\tilde\zeta_1,\dots,\tilde\zeta_n):=\Psi^n_{z,\tilde z}(\xi)$
satisfies
\begin{equation}
\label{eq: coalescence map recursive relation}
    \Psi^k_{\tilde z}(\tilde\zeta_1,\dots\tilde\zeta_k)=\Psi^k_z(\xi_1,\dots,\xi_k)+y_k, \qquad k\in\{1,\dots n\}.
\end{equation}
The map $\Psi_{z,\tilde{z}}^{n}$ is called the \emph{coalescence map defined by the trajectory $y$}.
\end{lemdef} 

\begin{proof}
We define $\tilde\zeta_1,\dots,\tilde\zeta_n$ by induction on $k$. 

\emph{Base case ($k=1$).} The $k=1$ instance
of~\eqref{eq: coalescence map recursive relation} reads
$\Psi_{\tilde z}(\tilde\zeta_1)=\Psi_z(\xi_1)+y_1$. Since $\Psi_{\tilde z}$
is invertible (cf.~\eqref{eq: Psi_z explicit}), this uniquely determines
$\tilde\zeta_1$.

\emph{Inductive step.} Fix $k\in\{2,\dots,n\}$ and suppose
$\tilde\zeta_1,\dots,\tilde\zeta_{k-1}$ have been determined. By the
recursive definition~\eqref{eq: definition Psi^n_z} of $\Psi^k_z$ and
$\Psi^k_{\tilde z}$, the $k$-th instance
of~\eqref{eq: coalescence map recursive relation} is equivalent to
\begin{equation*}
    \Psi_w(\tilde\zeta_k)
    = \Psi_{\Psi^{k-1}_z(\xi_1,\dots,\xi_{k-1})}(\xi_k) + y_k,
    \qquad
    w := \Psi^{k-1}_{\tilde z}(\tilde\zeta_1,\dots,\tilde\zeta_{k-1}).
\end{equation*}
By the inductive hypothesis, $w\in\R^{2d}$ is known and the right-hand
side is a fixed element of $\R^{2d}$; invertibility of $\Psi_w$ uniquely
determines $\tilde\zeta_k$.

This defines $\tilde\zeta=(\tilde\zeta_1,\dots,\tilde\zeta_n)$. The
constraint $y_n=0$ gives
$\Psi^n_{\tilde z}(\tilde\zeta)=\Psi^n_z(\xi)$, so $\Psi^n_{z,\tilde z}$
is a coalescence map by Definition~\ref{def: coalescence map}. \end{proof}

We next verify that trajectory-based coalescence maps are diffeomorphisms,
as required for Lemma~\ref{lem: derivative coalescence map bounds TV distance}.
Invertibility is established below; differentiability is the content of
Lemma~\ref{lemma: trajectory based coalescene maps are differentiable}.

\medskip
\begin{lemma}
\label{lemma: trajectory based coalescene maps are invertible}  Let $z,\tilde z\in\R^{2d}$, $n\in\N_{>0}$, and let $\Psi^n_{z,\tilde z}$
be the coalescence map defined by the trajectory $y\in\R^{2dn}$. Then
$\Psi^n_{z,\tilde z}$ is invertible, with
\begin{equation*}
    (\Psi^n_{z,\tilde z})^{-1} = \Psi^n_{\tilde z,z},
\end{equation*}
where $\Psi^n_{\tilde z,z}$ is the coalescence map defined by the
trajectory $-y$.
\end{lemma}

\begin{proof}
The trajectory $-y$ satisfies $(-y)_n=0$, so
Lemma~\ref{lemdef: coalescence map defined by trajectory} produces a
coalescence map $\Psi^n_{\tilde z,z}$ defined by it. We show
$\Psi^n_{\tilde z,z}\circ\Psi^n_{z,\tilde z}=\id_{\R^{2dn}}$; the reverse
identity follows from the same argument with $(z,\tilde z,y)$ replaced by
$(\tilde z,z,-y)$.

Fix $\xi\in\R^{2dn}$ and set $\tilde\zeta:=\Psi^n_{z,\tilde z}(\xi)$,
$\zeta:=\Psi^n_{\tilde z,z}(\tilde\zeta)$. From the defining
relations~\eqref{eq: coalescence map recursive relation} for the two
coalescence maps,
\begin{equation*}
    \Psi^k_z(\zeta_1,\dots,\zeta_k)
    = \Psi^k_{\tilde z}(\tilde\zeta_1,\dots,\tilde\zeta_k) - y_k
    = \Psi^k_z(\xi_1,\dots,\xi_k),
    \qquad k\in\{1,\dots,n\}.
\end{equation*}
By the recursive definition~\eqref{eq: definition Psi^n_z}, the $k$-th
identity in this chain is equivalent to
\begin{equation*}
    \Psi_{\Psi^{k-1}_z(\zeta_1,\dots,\zeta_{k-1})}(\zeta_k)
    = \Psi_{\Psi^{k-1}_z(\xi_1,\dots,\xi_{k-1})}(\xi_k),
\end{equation*}
with the convention $\Psi^0_z\equiv z$ at $k=1$. The same chain at index
$k-1$ gives
$\Psi^{k-1}_z(\zeta_1,\dots,\zeta_{k-1})=\Psi^{k-1}_z(\xi_1,\dots,\xi_{k-1})$
(trivially at $k=1$), and invertibility of $\Psi_w$
(cf.~\eqref{eq: Psi_z explicit}) then gives $\zeta_k=\xi_k$.
\end{proof}

We derive a closed-form expression for $\Psi^n_{z,\tilde z}$. Two
properties of this expression are central to what follows:
differentiability (Lemma~\ref{lemma: trajectory based coalescene maps are differentiable})
and a block-lower-triangular Jacobian with identity blocks on the diagonal.
The latter eliminates the trace and log-determinant terms from
Lemma~\ref{lem: derivative coalescence map bounds TV distance}.

Fix $\xi\in\R^{2dn}$, set $\tilde\zeta:=\Psi^n_{z,\tilde z}(\xi)$, and for
$k\in\{0,\dots,n\}$ define
\begin{equation}\label{eq: definition z_k and tilde z_k}
    z_k := \Psi^k_z(\xi_1,\dots,\xi_k) = (x_k,v_k),
    \qquad
    \tilde z_k := \Psi^k_{\tilde z}(\tilde\zeta_1,\dots,\tilde\zeta_k)
                = (\tilde x_k,\tilde v_k),
\end{equation}
with the convention $\Psi^0_w\equiv w$ giving $z_0=z$ and
$\tilde z_0=\tilde z$. Decompose $y_k=(u_k,w_k)$ into its position and
velocity components $u_k,w_k\in\R^d$, and extend $y$ to $k=0$ via
$y_0:=\tilde z-z$.

In this notation, relation~\eqref{eq: coalescence map recursive relation}
becomes the one-step recursion
\begin{equation}\label{eq: coalescence map one step recursive relation}
    y_{k+1}
    = \tilde z_{k+1}-z_{k+1}
    = \Psi_{\tilde z_k}(\tilde\zeta_{k+1})-\Psi_{z_k}(\xi_{k+1}),
    \qquad k\in\{0,\dots,n-1\}.
\end{equation}
Substituting the explicit form~\eqref{eq: Psi_z explicit} of $\Psi_w$ and
solving for $\tilde\zeta_{k+1}$ gives
\begin{equation}\label{eq: coalescence map general y_k}
    \tilde\zeta_{k+1}
    = \xi_{k+1}+L_h^{-1}\biggl(
        y_{k+1}-A_h y_k
        +\frac{h}{2}\begin{pmatrix}
            h\,[\nabla U(\tilde x_k)-\nabla U(x_k)]\\
            e^{-\gamma h/2}\,[\nabla U(\tilde x_k)-\nabla U(x_k)+\nabla U(\tilde x_{k+1})-\nabla U(x_{k+1})]
        \end{pmatrix}
    \biggr),
\end{equation}
or, in position and velocity components,
\begin{equation}\label{eq: coalescence map explicit general y_k}
\begin{aligned}
    \tilde\zeta_{k+1}^{(1)}
    &= \xi_{k+1}^{(1)}+(1-e^{-\gamma h})^{-1/2}
       \sqbr{\tfrac{1}{h}(u_{k+1}-u_k-h\,e^{-\gamma h/2}w_k)
             +\tfrac{h}{2}\br{\nabla U(\tilde x_k)-\nabla U(x_k)}},\\
    \tilde\zeta_{k+1}^{(2)}
    &= \xi_{k+1}^{(2)}+(1-e^{-\gamma h})^{-1/2}\bigg[
         w_{k+1}-e^{-\gamma h}w_k
         -\tfrac{1}{h}e^{-\gamma h/2}(u_{k+1}-u_k-h\,e^{-\gamma h/2}w_k)\\
    &\qquad\qquad\qquad\qquad\quad
         +\tfrac{h}{2}e^{-\gamma h/2}\br{\nabla U(\tilde x_{k+1})-\nabla U(x_{k+1})}
       \bigg].
\end{aligned}
\end{equation}

The explicit form~\eqref{eq: coalescence map explicit general y_k} has
two structural features. First, $\tilde\zeta_{k+1}$ is $\xi_{k+1}$ plus a
function of $\xi_1,\dots,\xi_k$ 
and the data $z,\tilde z,y$; this yields differentiability of
$\Psi^n_{z,\tilde z}$ by induction on $k$ (Lemma~\ref{lemma: trajectory based coalescene maps are differentiable}). Second, the same structure
forces $D\Psi^n_{z,\tilde z}(\xi)$ to be lower triangular with all
diagonal entries equal to $1$, so $\det D\Psi^n_{z,\tilde z}(\xi)=1$ and
$\tr D\Psi^n_{z,\tilde z}(\xi)=2dn$ identically. The trace and
log-determinant terms in
Lemma~\ref{lem: derivative coalescence map bounds TV distance} therefore
vanish.

\begin{lemma}\label{lemma: trajectory based coalescene maps are differentiable}
Let $z,\tilde z\in\R^{2d}$, $n\in\N_{>0}$, and let $\Psi^n_{z,\tilde z}$
be the coalescence map defined by a trajectory $y\in\R^{2dn}$ with
$y_n=0$. If $U\colon\R^d\to\R$ is twice continuously differentiable, then
$\Psi^n_{z,\tilde z}$ is a $C^1$ diffeomorphism of $\R^{2dn}$ whose
Jacobian $D\Psi^n_{z,\tilde z}$ is lower triangular with all diagonal
entries equal to~$1$.
\end{lemma}

\begin{proof}
Since $U\in C^2$, the explicit form~\eqref{eq: Psi_z explicit} shows that
the map $(w,\eta)\mapsto\Psi_w(\eta)$ is continuously differentiable on
$\R^{2d}\times\R^{2d}$. By induction on $k$, the iterate
$z_k=\Psi_{z_{k-1}}(\xi_k)$ depends only on $(\xi_1,\dots,\xi_k)$ and is
$C^1$ in these arguments; the same holds for $\tilde z_k$
via~\eqref{eq: coalescence map one step recursive relation}. Moreover,
inspection of~\eqref{eq: Psi_z explicit} shows that $x_k$, and hence
$\tilde x_k$, depends only on $(\xi_1,\dots,\xi_{k-1},\xi_k^{(1)})$.

Substituting into~\eqref{eq: coalescence map explicit general y_k} gives
that $\tilde\zeta_k^{(1)}-\xi_k^{(1)}$ is $C^1$ in $(\xi_1,\dots,\xi_{k-1})$
alone, and $\tilde\zeta_k^{(2)}-\xi_k^{(2)}$ is $C^1$ in
$(\xi_1,\dots,\xi_{k-1},\xi_k^{(1)})$. Hence $\Psi^n_{z,\tilde z}$ is
$C^1$, with partial derivatives
\begin{equation*}
    \frac{\partial\tilde\zeta_k^{(1)}}{\partial\xi_k^{(1)}}=\id_d,
    \qquad
    \frac{\partial\tilde\zeta_k^{(1)}}{\partial\xi_k^{(2)}}=0,
    \qquad
    \frac{\partial\tilde\zeta_k^{(2)}}{\partial\xi_k^{(2)}}=\id_d,
    \qquad
    \frac{\partial\tilde\zeta_k}{\partial\xi_j}=0\ \text{for }j>k.
\end{equation*}
Thus $D\Psi^n_{z,\tilde z}$ is lower triangular with all diagonal entries
equal to~$1$. Combined with
Lemma~\ref{lemma: trajectory based coalescene maps are invertible},
$\Psi^n_{z,\tilde z}$ admits a $C^1$ inverse and is therefore a
diffeomorphism.
\end{proof}

\begin{corollary}\label{cor: coalescence map bound}
Under the assumptions of
Lemma~\ref{lemma: trajectory based coalescene maps are differentiable},
for $\xi\sim\calN(0,\id_{2dn})$,
\begin{equation}\label{eq: TV bound in terms of the coalescence map without trace and determinant terms}
    \dTV{\pi_n(\delta_z),\,\pi_n(\delta_{\tilde z})}
    \;\le\;
    \tfrac{1}{2}\bigl(\E\sqbr{|\Psi^n_{z,\tilde z}(\xi)-\xi|^2}\bigr)^{1/2}.
\end{equation}
\end{corollary}

Up to this point, the trajectory $y$ has been arbitrary: the construction
of $\Psi^n_{z,\tilde z}$, its diffeomorphism property, and the
bound~\eqref{eq: TV bound in terms of the coalescence map without trace and determinant terms}
all hold for any $y\in\R^{2dn}$ satisfying $y_n=0$. Different choices of
$y$ produce different coalescence maps and therefore different right-hand
sides; the freedom in the choice of $y$ is the lever by which the bound
is made sharp. To complete the proof of
Theorem~\ref{thm: TV bound via coalescence map}, it remains to choose $y$
so that this right-hand side is as small as possible. This is the content
of Section~\ref{subsec: optimized trajectory}.

\medskip
\begin{remark}\label{rem: other splitting schemes}
    In order to construct coalescence maps $\Psi^n_{z,\tilde z}$ for the OBABO splitting scheme and arrive at the bound \eqref{eq: TV bound in terms of the coalescence map without trace and determinant terms}, we argued that
    \begin{enumerate}
        \item \label{it: well-defined} $\Psi^n_{z,\tilde{z}}$ is well-defined through the relation \eqref{eq: coalescence map recursive relation};
        \item \label{it: diff} $\Psi^n_{z,\tilde{z}}$ is a diffeomorphism; and
        \item the Jacobian matrix $D\Psi^n_{z,\tilde{z}}$ is a lower-triangular matrix, so that the last two terms in equation \eqref{eq: derivative coalescence map bounds TV distance} vanish.
    \end{enumerate}
    The same line of reasoning can be applied to other gHMC methods and splitting schemes. More specifically, if the single step update map $\Psi_z$ is invertible for each $z\in\R^{2d}$, then~\eqref{it: well-defined} is satisfied. It subsequently suffices to verify that $(z,\xi)\mapsto \Psi_z(\xi)$ and $(z,y)\mapsto \Psi^{-1}_z(y)$ are continuously differentiable to  ensure~\eqref{it: diff}. \par 
    These invertibility and differentiability properties of the single-step update map $\Psi_z$ are satisfied for the gHMC method, as discussed in \cite{ChakMonmarche:2025}, provided that $U$ is twice continuously differentiable and the integration time of the corresponding velocity Verlet integrator is sufficiently small. Unfortunately, the Jacobian matrices $D\Psi^n_{z,\tilde z}$ of the resulting coalescence maps are not necessarily lower triangular. As a consequence, the final two terms in \eqref{eq: derivative coalescence map bounds TV distance} do not vanish and need to be treated separately in order to derive TV bounds, requiring the additional assumption that $U\in C^2(\R^{d})$ has a Lipschitz continuous Hessian. \par 
    We note that the framework described above extends naturally beyond the gHMC setting of \cite{ChakMonmarche:2025} to include several well-known splitting schemes associated to the kinetic Langevin equation. More specifically, the BOAOB splitting scheme has a single step update map $\Psi_z$ that satisfies the aforementioned invertibility and differentiability properties, thus giving rise to diffeomorphic coalescence maps. Similar to the OBABO scheme, these coalescence maps have an explicit expression and a Jacobian matrix that is lower triangular with all diagonal entries equal to $1$, so that \eqref{eq: TV bound in terms of the coalescence map without trace and determinant terms} again holds true. \par 
    Regarding the OABAO splitting scheme, the additional assumption that $\frac{Lh^2}{4}<1$ is needed to ensure that the single step update maps $\Psi_z$, $z\in \R^{2d}$, are invertible. The resulting coalescence maps are implicit and the final two terms in \eqref{eq: derivative coalescence map bounds TV distance} do not vanish, so that the additional assumption that $U\in C^2(\R^{d})$ has a Lipschitz continuous Hessian is again needed for the analysis. These are standard assumptions that are also required in \cite{Bou-RabeeOberdoster:2024} to ensure existence of the one-shot map for the OABAO scheme. \par 
    The splitting schemes BAOAB, ABOBA and AOBOA do not have diffeomorphic update mappings, and therefore do not fit the framework in its current state. We expect that it is feasible to extend the framework to these schemes by using the fact that $n\geq2$ steps of these schemes consist of $n-1$ steps of the OABAO, OBABO and BOAOB schemes respectively accompanied by two remaining partial updates at the beginning and the end. Finally, the framework can also be transferred over to the UBU splitting scheme under similar assumptions as for the OABAO splitting scheme. 
\end{remark}


\subsection{The optimized trajectory}
\label{subsec: optimized trajectory}

The bound of Corollary~\ref{cor: coalescence map bound} depends on the
trajectory $y$ through the $L^2$ deviation
$\E[|\Psi^n_{z,\tilde z}(\xi)-\xi|^2]$. The aim of this subsection is to
choose $y$ so as to make this deviation small.  
To see why the choice matters, consider $y_k=0$ for all
$k\in\{1,\dots,n\}$. This forces the two chains to coalesce after the
first OBABO step, since $\Psi_{\tilde z}(\tilde\zeta_1)=\Psi_z(\xi_1)$,
so the entire initial difference $\tilde z-z$ must be absorbed by the
single noise increment $\tilde\zeta_1$. When $|\tilde z-z|$ is large
relative to the per-step noise scale, the law of $\tilde\zeta_1$ is far
from $\calN(0,\id_{2d})$, and Corollary~\ref{cor: coalescence map bound}
yields a weak bound.

A better strategy spreads the coalescence over all $n$ steps, keeping
each $\tilde\zeta_k$ close in law to a standard Gaussian. Such a
trajectory is also chosen in~\cite{ChakMonmarche:2025}, where it takes
the form of an explicit ansatz that suffices to establish
Theorem~\ref{thm: CM Wasserstein-to-TV regularization} but is not
derived from an optimality principle. Here we take an optimization
approach: we identify the trajectory that minimizes the $L^2$ deviation
$\E\sqbr{|\Psi^n_{z,\tilde z}(\xi)-\xi|^2}$ in the force-free case
$\nabla U\equiv 0$, characterized as the unique solution of a discrete
linear-quadratic control problem. We then use this trajectory for general
$U$, and obtain the corresponding optimality statement when $U$ is
quadratic (Remark~\ref{rem: coalescence gives optimal TV for linear case}).
See Figure~\ref{fig:optimal-trajectory-phase} for a phase-space
visualization of the two trajectories, and
Remark~\ref{remark: CM trajectory} for a detailed comparison.

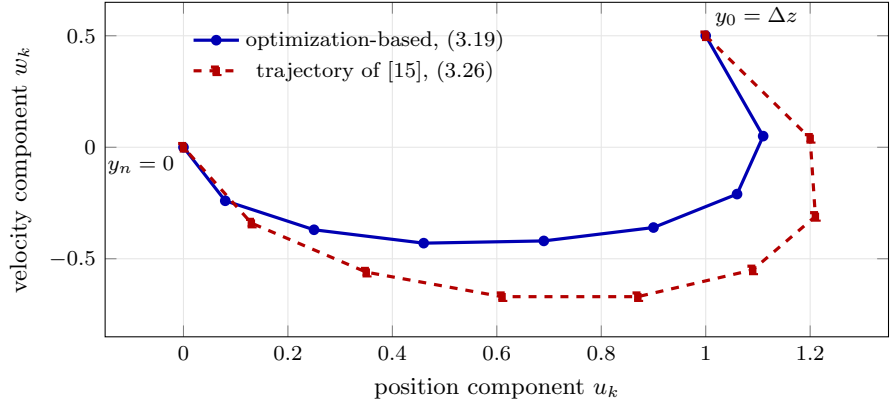
\begin{figure}[t]
\centering
\begin{tikzpicture}
\begin{axis}[
    width=0.72\linewidth, height=6cm,
    xlabel={position component $u_k$},
    ylabel={velocity component $w_k$},
    xmin=-0.15, xmax=1.35,
    ymin=-0.85, ymax=0.65,
    legend pos=south east,
    legend style={font=\footnotesize, draw=none, fill=none,
                  at={(axis description cs:0.1,0.95)},
                  anchor=north west},
    label style={font=\small},
    tick label style={font=\footnotesize},
    grid=both,
    grid style={gray!20},
    every axis plot/.append style={very thick, mark=*, mark size=1.4pt},
]
  \addplot[blue!70!black] coordinates {
    (1.00, 0.50) (1.11, 0.05) (1.06, -0.21) (0.90, -0.36)
    (0.69, -0.42) (0.46, -0.43) (0.25, -0.37)
    (0.08, -0.24) (0, 0)
  };
  \addlegendentry{optimization-based, \eqref{eq: y_k choices optimized trajectory}}
  \addplot[red!70!black, dashed, mark=square*] coordinates {
    (1.00, 0.50) (1.20, 0.04) (1.21, -0.31) (1.09, -0.55)
    (0.87, -0.67) (0.61, -0.67) (0.35, -0.56)
    (0.13, -0.34) (0, 0)
  };
  \addlegendentry{trajectory of \cite{ChakMonmarche:2025}, \eqref{eq: CM u_k and w_k choices}}
  \node[anchor=south west, font=\footnotesize] at (axis cs: 1.00, 0.50)
       {$y_0=\Delta z$};
  \node[anchor=north east, font=\footnotesize] at (axis cs: 0, 0)
       {$y_n=0$};
\end{axis}
\end{tikzpicture}
\caption{Phase-space view of two coalescence trajectories
$(y_k)_{k=0}^n=(u_k,w_k)_{k=0}^n$ in $\R^d\times\R^d$ ($d=1$ depicted),
both satisfying $y_0=\Delta z=(\Delta x,\Delta v)$ and $y_n=0$, with
$\gamma=1$, $h=0.5$, $n=8$, $\Delta z=(1,0.5)$. The
trajectory~\eqref{eq: y_k choices optimized trajectory} minimizes
$\sum_{k=1}^n|E_k|^2$ over all admissible trajectories; the trajectory
of~\cite{ChakMonmarche:2025} given by~\eqref{eq: CM u_k and w_k choices}
is determined by an explicit ansatz. The two paths differ structurally:
in~\eqref{eq: CM u_k and w_k choices} the position component follows
the noiseless dynamics $u_{k+1}=u_k+h\,e^{-\gamma h/2}w_k$ (visible as
the initial increase in $u$ driven by $w_0=\Delta v>0$),
while~\eqref{eq: y_k choices optimized trajectory} introduces direct
noise corrections to the position that damp the overshoot. 
See Remark~\ref{remark: CM trajectory} for a
detailed comparison.}
\label{fig:optimal-trajectory-phase}
\end{figure}

To identify the trajectory $y$ that minimizes the $L^2$ deviation
$\E\sqbr{|\Psi^n_{z,\tilde z}(\xi)-\xi|^2}$, recall
from~\eqref{eq: coalescence map general y_k} that
\begin{equation*}
   (\Psi_{z,\tilde z}^n(\xi) - \xi)_k
   = L_h^{-1}\br{y_{k}-A_hy_{k-1}+\frac{h}{2}\begin{pmatrix}
        h\sqbr{\nabla U(\tilde x_{k-1})-\nabla U(x_{k-1})}\\
        e^{-\gamma h/2}\sqbr{\nabla U(\tilde x_{k-1})-\nabla U(x_{k-1})
            +\nabla U(\tilde x_{k})-\nabla U(x_{k})}
     \end{pmatrix}}
\end{equation*}
for $k\in\{1,\dots,n\}$, with $(x_k,v_k)=z_k$ and
$(\tilde x_k,\tilde v_k)=\tilde z_k$ defined
in~\eqref{eq: definition z_k and tilde z_k}.

The nonlinear $\nabla U$ terms obstruct an explicit minimization of
$\E\sqbr{|\Psi^n_{z,\tilde z}(\xi)-\xi|^2}$ over $y$. We sidestep this by
first solving the optimization in the force-free case $\nabla U\equiv 0$,
where $(\Psi^n_{z,\tilde z}(\xi)-\xi)_k$ depends linearly on $y$ and the
minimization reduces to a discrete linear-quadratic control problem; the
$\nabla U$ contribution in the general case is then absorbed via the
Lipschitz property of $\nabla U$ in the proof of
Theorem~\ref{thm: TV bound via coalescence map}.

In the force-free case $\nabla U\equiv 0$, the $L^2$ deviation reduces to
the deterministic quantity
\begin{equation*}
    \sum_{k=0}^{n-1}|L_h^{-1}(y_{k+1}-A_hy_k)|^2,
\end{equation*}
which we minimize over trajectories $y$ subject to
$y_0=\Delta z\coloneqq\tilde z-z$ and $y_n=0$. Introducing the change of
variables $E_{k+1}\coloneqq -L_h^{-1}(y_{k+1}-A_hy_k)$ for
$k\in\{0,\ldots,n-1\}$, the recursion becomes $y_{k+1}=A_hy_k-L_hE_{k+1}$,
and the minimization reads
\begin{align*}
    \min_{E=(E_1,\ldots,E_n)\in\R^{2dn}} |E|^2
    \quad\text{subject to}\quad
    \begin{aligned}
        &y_0=\Delta z,\\
        &y_{k+1}=A_hy_k-L_hE_{k+1},\quad k=0,\ldots,n-1,\\
        &y_n=0.
    \end{aligned}
\end{align*}
This is the minimum-energy control problem for the discrete linear
system $(A_h,L_h)$ with terminal constraint at time $n$, a standard
problem in control theory~\cite{BaggioEtAl:2019,Kailath:1980}. Its
solution is
\begin{equation}\label{eq: def Ek}
    E_k = L_h^T(A_h^T)^{n-k}\Sigma_{h,n}^{-1}A_h^n\Delta z,
    \qquad k=1,\ldots,n,
\end{equation}
where $\Sigma_{h,n}=\sum_{k=1}^{n} A_h^{n-k}L_hL_h^T(A_h^T)^{n-k}$ is the
controllability Gramian of $(A_h,L_h)$. Equivalently, $\Sigma_{h,n}$ is
the covariance matrix of $Z^h_n$ in the force-free case.

In view of the discussion above, we take $\Psi^n_{z,\tilde z}$ to be the coalescence map defined by the trajectory $y=(y_1,\dots,y_n)$ (see Definition~\ref{lemdef: coalescence map defined by trajectory}) generated recursively from $y_0=\Delta z$ by
\begin{equation}
\label{eq: y_k choices optimized trajectory}
    y_{k+1}=A_hy_k-L_hE_{k+1}.
\end{equation}
The terminal condition $y_n=0$ holds by construction: iterating~\eqref{eq: y_k choices optimized trajectory}
and using~\eqref{eq: def Ek},
\begin{equation*}
    y_n=A_h^ny_0-\sum_{k=1}^n A_h^{n-k}L_hE_{k}=A_h^n\Delta z-\br{\sum_{k=1}^n A_h^{n-k}L_hL_h^T(A_h^T)^{n-k}}\Sigma_{h,n}^{-1}A_h^n\Delta z=0,
\end{equation*}

Substituting~\eqref{eq: y_k choices optimized trajectory} and the
explicit form of $L_h^{-1}$ from~\eqref{eq: OBABO matrices}
into~\eqref{eq: coalescence map general y_k}, the components of
$\tilde\zeta=\Psi^n_{z,\tilde z}(\xi)$ become
\begin{equation}\label{eq: def tilde zeta}
\begin{aligned}
    \tilde\zeta_{k+1}
    &= \xi_{k+1}-E_{k+1}+\frac{h}{2}L_h^{-1}
       \begin{pmatrix}
            h\sqbr{\nabla U(\tilde x_k)-\nabla U(x_k)}\\
            e^{-\gamma h/2}\sqbr{\nabla U(\tilde x_k)-\nabla U(x_k)+\nabla U(\tilde x_{k+1})-\nabla U(x_{k+1})}
       \end{pmatrix}\\
    &= \xi_{k+1}-E_{k+1}+\frac{h}{2(1-e^{-\gamma h})^{1/2}}
       \begin{pmatrix}
            \nabla U(\tilde x_k)-\nabla U(x_k)\\
            e^{-\gamma h/2}\sqbr{\nabla U(\tilde x_{k+1})-\nabla U(x_{k+1})}
       \end{pmatrix}.
\end{aligned}
\end{equation}
The Lipschitz continuity of $\nabla U$ then yields
\begin{align*}
    |\tilde\zeta_{k+1}-\xi_{k+1}|^2&\leq2|E_{k+1}|^2+\frac{h^2}{2(1-e^{-\gamma h})}\l|\begin{pmatrix}\nabla U(\tilde x_k)-\nabla U(x_k)\\ e^{-\gamma h/2}\sqbr{\nabla U(\tilde x_{k+1})-\nabla U(x_{k+1})} \end{pmatrix}\r|^2\\
    &=2|E_{k+1}|^2+\frac{h^2}{2(1-e^{-\gamma h})}\br{\l|\nabla U(\tilde x_k)-\nabla U(x_k)\r|^2+e^{-\gamma h}\l|\nabla U(\tilde x_{k+1})-\nabla U(x_{k+1})\r|^2}\\
    &\leq2|E_{k+1}|^2+\frac{h^2L^2}{2(1-e^{-\gamma h})}\br{\l|u_k\r|^2+e^{-\gamma h}\l|u_{k+1}\r|^2},
\end{align*}
where $L$ is the Lipschitz constant of $\nabla U$ and where $u_k=\tilde x_k-x_k$ is the position component of $y_k$. Summing
over $k=0,\ldots,n-1$ and using $u_n=0$ together with
$1+e^{-\gamma h}\le 2$ gives the following.

\begin{lemma}\label{lem: coalescence map square distance}
Let $z,\tilde z\in\R^{2d}$ and $n\in\N_{>0}$, and assume
$U\colon\R^d\to\R$ is twice continuously differentiable with $\nabla U$
$L$-Lipschitz. Let $y$ be the trajectory defined by $y_0=\Delta z$
together with~\eqref{eq: def Ek}
and~\eqref{eq: y_k choices optimized trajectory}, and let
$\Psi^n_{z,\tilde z}$ be the coalescence map defined by $y$.  Then
\begin{equation}
\label{eq: coalescence map squared distance}
    |\Psi^n_{z,\tilde z}(\xi)-\xi|^2 \leq 2\sum_{k=1}^n|E_k|^2+\frac{h^2L^2}{1-e^{-\gamma h}}\sum_{k=0}^{n-1}\l|u_{k}\r|^2,
\end{equation}
where $y_k=(u_k,w_k)$ and $E_k$ is given by~\eqref{eq: def Ek}.
\end{lemma}

The first term on the right-hand side
of~\eqref{eq: coalescence map squared distance} comes from the
force-free part of the dynamics, and the second captures the perturbation
due to $\nabla U$. To complete the proof of
Theorem~\ref{thm: TV bound via coalescence map} via
Corollary~\ref{cor: coalescence map bound}, it remains to bound the two
sums on the right-hand side of~\eqref{eq: coalescence map squared distance}.
This is the content of
Lemmas~\ref{lemma: potential free part bound}
and~\ref{lemma: potential part bound}.

\medskip
\begin{lemma}\label{lemma: potential free part bound}
Let $z,\tilde z\in\R^{2d}$ and $n\in\N_{>0}$, and let $E_k$ be given
by~\eqref{eq: def Ek} for $k\in\{1,\ldots,n\}$. Then
\begin{equation*}
    \sum_{k=1}^{n}|E_k|^2
    \le \br{\frac{44}{\gamma(hn)^3}+\frac{264+44\gamma^2}{\gamma hn}}|\Delta z|^2.
\end{equation*}
\end{lemma}

The key identity for the proof, which follows directly from the
definition of $E_k$ in~\eqref{eq: def Ek}, is
\begin{equation}\label{eq: identity for E_k^2}
    \sum_{k=1}^n|E_k|^2
    = \Delta z^T(A_h^T)^n\Sigma_{h,n}^{-1}
      \underbrace{\br{\sum_{k=1}^n A_h^{n-k}L_h L_h^T(A_h^T)^{n-k}}}_{=\,\Sigma_{h,n}}
      \Sigma_{h,n}^{-1}A_h^n\Delta z
    = \bigl|\Sigma_{h,n}^{-1/2}A_h^n\Delta z\bigr|^2.
\end{equation}
The remaining argument is technical and is given in
Appendix~\ref{appendix: potential free bound}.

\medskip
\begin{lemma}\label{lemma: potential part bound}
Let $z,\tilde z\in\R^{2d}$ and $n\in\N_{>0}$, and let $y_k=(u_k,w_k)$,
$k\in\{0,\ldots,n\}$, be the trajectory defined by~\eqref{eq: def Ek}
and~\eqref{eq: y_k choices optimized trajectory}. The position
components satisfy
\begin{equation}\label{eq: uk bound}
    |u_k|
    \le |\Delta x|+\frac{he^{-\gamma h/2}}{1-e^{-\gamma h}}(1-e^{-\gamma hn})|\Delta v|
    \le 2\br{1+\frac{hn}{2+\gamma hn}}|\Delta z|,
    \qquad k\in\{0,\ldots,n\}.
\end{equation}
\end{lemma}

\begin{proof}
By the explicit computation in
Appendix~\ref{app: explicit optimized trajectory}
(equation~\eqref{eq: expression for u_k}), the position component of
$y_k$ is
\begin{equation}
\label{eq: explicit optimized trajectory u_k}
    u_{k}=\frac{\alpha_{n-k}}{\alpha_n}\Delta x-\frac{ h \eta^{\nicefrac{1}{2}}}{1-\eta}
    \frac{\eta^{n}\beta_{n-k}-(1-\eta^n)\alpha_{n-k}}{\alpha_n}\Delta v,\qquad k\in \{1,\ldots,n\}
\end{equation}
where we use the shorthand $\eta =e^{-\gamma h}$, and 
\begin{align*}
    \alpha_k&=(1-\eta^2)(1-\eta^{2n})k-2\eta(1+\eta^{n-k})(1-\eta^k)(1-\eta^n),\\
    \beta_k&=(1-\eta^2)\br{\frac{1}{\eta^k}(1-\eta^k)^2n+(1-\eta^n)^2k}-2\frac{\eta}{\eta^k}(1-\eta^k)(1-\eta^n)(2-\eta^k-\eta^{n}),\qquad k\in \{0,\ldots,n\}.
\end{align*}
(See~\eqref{eq: alpha_k definition} and~\eqref{eq: beta_k definition}.)
    It thus immediately follows that we have the following bound:
    \begin{equation}\label{eq: uk bound intermediate}
        |u_k|\leq \frac{|\alpha_{n-k}|}{|\alpha_n|}|\Delta x|+\frac{h\eta^{1/2}}{1-\eta}\frac{|\eta^n\beta_{n-k}-(1-\eta^n)\alpha_{n-k}|}{|\alpha_n|}|\Delta v|, \qquad k\in \{1,\ldots,n\}.
    \end{equation}
 By \cref{lemma: alpha_k and beta_k are increasing}, the sequences
$(\alpha_k)_{0\le k\le n}$ and $(\beta_k)_{0\le k\le n}$ are increasing,
$\alpha_0=\beta_0=0$, and a direct computation gives
   \begin{align*}
       \beta_n&=(1-\eta^2)\frac{1}{\eta^n}(1-\eta^n)^2\br{1+\eta^n}n-4\frac{\eta}{\eta^n}(1-\eta^n)^3\\
       &=\frac{1}{\eta^n}(1-\eta^n)\sqbr{(1-\eta^2)(1-\eta^{2n})n-4\eta(1-\eta^n)^2}=\frac{1}{\eta^n}(1-\eta^n)\alpha_n.
   \end{align*}
    In particular, we have that
    \begin{align*}
        -(1-\eta^n)\alpha_n & \leq-(1-\eta^n)\alpha_{n-k}\leq\eta^n\beta_{n-k}-(1-\eta^n)\alpha_{n-k}
        \leq \eta^n \beta_n - (1-\eta^n)\alpha_{n-k} \\
        & =(1-\eta^n)(\alpha_n-\alpha_{n-k}) \leq(1-\eta^n)\alpha_n.
    \end{align*}
    We conclude that both $|\alpha_{n-k}|=\alpha_{n-k}\leq\alpha_n$ and
    \begin{equation*}
        |\eta^n\beta_{n-k}-(1-\eta^n)\alpha_{n-k}|\leq(1-\eta^n)\alpha_n,
    \end{equation*}
    for all $k\in\{0,\dots,n\}$, which, when combined with~\eqref{eq: uk bound intermediate}, gives the first inequality in~\eqref{eq: uk bound}.  The second
inequality follows from the exponential bounds in
Appendix~\ref{appendix: exponential bounds}:
\eqref{eq: exp bound lemma 3} gives
$he^{-\gamma h/2}/(1-e^{-\gamma h})\le 1/\gamma$,
\eqref{eq: exp bound lemma 1} gives
$1-e^{-\gamma hn}\le 2\gamma hn/(2+\gamma hn)$, and
$|\Delta x|,|\Delta v|\le|\Delta z|$.
\end{proof}

For completeness we include:\label{proof of TV bound}
\begin{proof}[Proof of Theorem~\ref{thm: TV bound via coalescence map}]
Combining Corollary~\ref{cor: coalescence map bound} with
Lemmas~\ref{lem: coalescence map square distance},
\ref{lemma: potential free part bound},
and~\ref{lemma: potential part bound} yields the bound.
\end{proof}

\medskip
\begin{remark}\label{rem: KL Renyi extension}
The proof of Theorem~\ref{thm: TV bound via coalescence map} passes
through the Pinsker-type estimate of
Lemma~\ref{lem: derivative coalescence map bounds TV distance}, in which
the trace and log-determinant terms vanish identically (by
Lemma~\ref{lemma: trajectory based coalescene maps are differentiable},
since $\det D\Psi^n_{z,\tilde z}\equiv 1$ and
$\tr D\Psi^n_{z,\tilde z}\equiv 2dn$). The Pinsker step therefore yields,
without further work, a Kullback--Leibler analogue:
\begin{equation*}
    D_{\textup{KL}}\bigl(\pi_n(\delta_z)\,\|\,\pi_n(\delta_{\tilde z})\bigr)
    \le \tfrac{1}{2}\E\sqbr{|\Psi^n_{z,\tilde z}(\xi)-\xi|^2},
\end{equation*}
which by Lemmas~\ref{lemma: potential free part bound}
and~\ref{lemma: potential part bound} is bounded by an explicit
expression in $\gamma$, $h$, $n$, $L$, and $|\Delta z|^2$. A R\'enyi
divergence analogue follows from the
framework of~\cite{Bou-RabeeMitraWibisono:2026}, which lifts
Orlicz--Wasserstein bounds to R\'enyi divergences of order
$\alpha\in(1,\infty)$ via one-shot couplings. Applying that framework
to the coalescence-map setting requires upgrading the $L^2$ moment
bound of Lemma~\ref{lem: coalescence map square distance} to an
Orlicz--Wasserstein control on $|\Psi^n_{z,\tilde z}(\xi)-\xi|$.
\end{remark}

\medskip
\begin{remark}
Suppose $\nabla U\equiv 0$ in Lemma~\ref{lem: coalescence map square distance}.
Then the coalescence map reduces to
$\Psi^n_{z,\tilde z}(\xi)=\xi-E$, where
$E\coloneqq(E_1,\ldots,E_n)\in\R^{2dn}$, so for
$\xi\sim\calN(0,\id_{2dn})$ we have
$\Psi^n_{z,\tilde z}(\xi)\sim\calN(-E,\id_{2dn})$.
By~\eqref{eq: identity for E_k^2},
$|E|^2=|\Sigma_{h,n}^{-1/2}A_h^n\Delta z|^2$, so the Gaussian TV
formula~\eqref{eq: Gaussian TV distance} gives
\begin{equation*}
    \dTV{\law(\xi),\,\law(\Psi^n_{z,\tilde z}(\xi))}
    = 2\Phi\br{\frac{|E|}{2}}-1
    = 2\Phi\br{\frac{|\Sigma_{h,n}^{-1/2}A_h^n\Delta z|}{2}}-1.
\end{equation*}
This coincides with $\dTV{\pi_n(\delta_z),\pi_n(\delta_{\tilde z})}$
(by~\eqref{eq: OBABO as linear Markov chain}
and~\eqref{eq: def pi_n}): when $\nabla U\equiv 0$, the bound of
Lemma~\ref{lemma: coalescence maps provide TV distance bounds} is
attained with equality, and the corresponding coupling of OBABO chains
is maximal.    
\end{remark}

\medskip
\begin{remark}\label{rem: coalescence gives optimal TV for linear case}
    The above framework can be used for more general linear Markov chains $Z_{k+1}=AZ_k+B\xi_{k+1}$ to find a trajectory that defines a coalescence map that exactly reproduces the TV distance. In particular, if we consider kinetic Langevin with a quadratic potential, its force can be included in the expression of $A$ and possibly $B$ in order to derive a coalescence map that exactly reproduces the TV distance. It is to be expected that this TV distance would satisfy a bound similar to that presented in \cref{thm: TV bound via coalescence map}, but without the third term (i.e., with $L=0$). Providing a fully rigorous proof of this would require deriving a result analogous to that of \cref{lemma: potential free part bound}, which would result in a rather lengthy calculation.
\end{remark}
\medskip

\begin{remark}\label{remark: CM trajectory}
The trajectory used in~\cite{ChakMonmarche:2025} differs from ours.
Recall the position-velocity splitting $\Delta z=(\Delta x,\Delta v)$
with $\Delta x=\tilde x-x$ and $\Delta v=\tilde v-v$. 
Instead of the recursive definition~\eqref{eq: y_k choices optimized trajectory},
the CM trajectory is given explicitly, for $k\in\{1,\ldots,n\}$, by
\begin{equation}\label{eq: CM u_k and w_k choices}
    \begin{gathered}
        w_k = \br{1-\frac{k}{n}-\frac{3k(n-k)}{n^2-n}}\Delta v
              -\frac{6k(n-k)}{(n^3-n)he^{-\gamma h/2}}\Delta x,\\
        u_k = \Delta x+h\,e^{-\gamma h/2}\sum_{j=0}^{k-1}w_j.
    \end{gathered}
\end{equation}
(A direct calculation verifies $w_n=0$ and $u_n=0$). 
This choice is designed to cancel the leading divergent term in $h$ of the
resulting TV bound (see~\cite[inequality~(68)]{Monmarche:2024}).

The two trajectories differ structurally as follows
(see Figure~\ref{fig:optimal-trajectory-phase}).  The position
component of~\eqref{eq: CM u_k and w_k choices} satisfies the recursion
$u_{k+1}=u_k+h\,e^{-\gamma h/2}w_k$, which is the position update of the
OBABO chain in the absence of force and noise: the
trajectory~\eqref{eq: CM u_k and w_k choices} uses noise corrections
only on the velocity component, with the positions evolving freely along
the deterministic dynamics. By contrast,
\eqref{eq: y_k choices optimized trajectory} introduces noise
corrections in both components and, by construction, minimizes
$\sum_{k=1}^n|E_k|^2$ over all admissible trajectories. The trajectory
of~\cite{ChakMonmarche:2025} is therefore suboptimal for
$\sum_{k=1}^n|E_k|^2$. However, the gap is of higher order in $h$: noise
enters the position component of the OBABO step with coefficient
$h(1-e^{-\gamma h})^{1/2}=O(h^{3/2})$, whereas it enters the velocity
component at order $h^{1/2}$ (cf.~\eqref{eq: OBABO matrices}).
Consequently, the additional position corrections
in~\eqref{eq: y_k choices optimized trajectory} contribute only
higher-order terms, and the two trajectories yield TV bounds of the same
leading order in $h$.

Both trajectories are independent of $\nabla U$: each is constructed for
the linearized problem $\nabla U\equiv 0$, and the $\nabla U$
contribution to the general case is absorbed via the Lipschitz argument
in Lemma~\ref{lem: coalescence map square distance}. Incorporating
information about $\nabla U$ into a coalescence trajectory directly is
not straightforward in either framework, since such a trajectory would
require knowing the chain trajectories in advance.  
\end{remark}

\subsection{An explicit non-Markovian coupling}\label{sec: non-Markovian coupling}
The line of argument introduced by \cite{ChakMonmarche:2025} and refined in the previous sections provides a Wasserstein-to-TV regularization bound for gHMC and several splitting methods. However, it does not directly result in an \emph{explicit} coupling between two Markov chains starting from different initial values.
In this section we use the coalescence map $\Psi^n_{z,\tilde z}$ constructed in
Section~\ref{subsec: coalescence maps} to construct such a coupling.
Specifically, we couple the driving noise sequences of the two chains so
that $\tilde\xi=\Psi^n_{z,\tilde z}(\xi)$ holds with maximal
probability; by the coalescence map property
(Definition~\ref{def: coalescence map}), the two chains coincide at time
$n$ on this event. Such a coupling
exists by the coupling characterization of the TV distance; we now give
an explicit construction by rejection sampling.

Fix $z,\tilde z\in\R^{2d}$ and $n\in\N_{>0}$.  Let $y$ be the trajectory
defined recursively by~\eqref{eq: def Ek}
and~\eqref{eq: y_k choices optimized trajectory}, and let
$\Psi^n_{z,\tilde z}$ be the coalescence map defined by $y$. Write
$\varphi_{2dn}$ for the density of $\calN(0,\id_{2dn})$ on $\R^{2dn}$ and
$\psi^n_{z,\tilde z}$ for the density of $\Psi^n_{z,\tilde z}(\xi)$ when
$\xi\sim\calN(0,\id_{2dn})$. By
Lemma~\ref{lemma: trajectory based coalescene maps are differentiable},
$\Psi^n_{z,\tilde z}$ is a $C^1$ diffeomorphism with
$\det D\Psi^n_{z,\tilde z}\equiv 1$, so
\begin{equation*}
    \psi^n_{z,\tilde z}(t)
    = \frac{1}{(2\pi)^{dn}}\exp\br{-\tfrac{1}{2}\,
        |(\Psi^n_{z,\tilde z})^{-1}(t)|^2},
    \qquad t\in\R^{2dn},
\end{equation*}
where the inverse $(\Psi^n_{z,\tilde z})^{-1}$ is the coalescence map
defined by the trajectory $-y$
(Lemma~\ref{lemma: trajectory based coalescene maps are invertible});
its explicit form in terms of $y$ is given
by~\eqref{eq: coalescence map one step recursive relation}.

In order to construct a coupling $(\xi,\tilde\xi)$ that has the marginals $\calN(0,\id_{2dn})$ and maximizes the probability that $\tilde\xi=\Psi^n_{z,\tilde z}(\xi)$, we follow the following rejection sampling procedure. Let $\xi\sim\calN(0,\id_{2dn})$, $U\sim\textup{Unif}(0,1)$,
$(\eta_j)_{j\ge 1}$ iid $\calN(0,\id_{2dn})$, and $(V_j)_{j\ge 1}$ iid
$\textup{Unif}(0,1)$, with all four jointly independent. Set
$\tilde\zeta:=\Psi^n_{z,\tilde z}(\xi)$ and define the acceptance event
\begin{equation}\label{eq: acceptance criterion}
    A := \l\{U\le\frac{\varphi_{2dn}(\tilde\zeta)}{\psi^n_{z,\tilde z}(\tilde\zeta)}\r\}
\end{equation}
and the rejection-trial index
\begin{equation}\label{eq: rejection criterion}
    N := \inf\l\{j\ge 1:V_j\le 1-\frac{\psi^n_{z,\tilde z}(\eta_j)}{\varphi_{2dn}(\eta_j)}\r\},
\end{equation}
which is almost surely finite. The coupling is then
\begin{equation}\label{eq: coupling definition}
    \tilde\xi := \tilde\zeta\,\ind_A + \eta_N\,\ind_{A^c}.
\end{equation}

By construction, the pair $(\xi,\tilde\xi)$ satisfies
$\xi\sim\calN(0,\id_{2dn})$; the construction's two branches (acceptance
and rejection) are illustrated in Figure~\ref{fig:coupling-branches}.
The next proposition asserts that $\tilde\xi$ also has marginal
$\calN(0,\id_{2dn})$ and that the coupling is maximal.

\begin{figure}[t]
\centering
\begin{tikzpicture}[
    every node/.style={font=\small},
    pt/.style={circle, fill=black, inner sep=1.1pt}
]
  \begin{scope}
    \node[anchor=south, font=\footnotesize, align=center] at (3,2.3)
      {acceptance event $A$\\[-2pt]
       probability $1-\dTV{\law(\xi),\law(\Psi^n_{z,\tilde z}(\xi))}$};

    \draw[thick] (0,0) -- (1.0,0.25) -- (2.0,0.10) -- (3.0,0.22)
                      -- (4.0,0.50) -- (5.0,0.72) -- (6.0,1.00);
    \foreach \p in {(0,0),(1.0,0.25),(2.0,0.10),(3.0,0.22),
                    (4.0,0.50),(5.0,0.72),(6.0,1.00)}
      \node[pt] at \p {};

    \draw[thick, dashed] (0,1.80) -- (1.0,1.55) -- (2.0,1.70) -- (3.0,1.25)
                                -- (4.0,1.35) -- (5.0,1.10) -- (6.0,1.00);
    \foreach \p in {(0,1.80),(1.0,1.55),(2.0,1.70),(3.0,1.25),
                    (4.0,1.35),(5.0,1.10)}
      \node[pt] at \p {};

    \node[below, font=\footnotesize] at (0,0)    {$z$};
    \node[above, font=\footnotesize] at (0,1.8)  {$\tilde z$};
    \node[right, font=\footnotesize] at (6.0,1.0) {$Z^h_n=\tilde Z^h_n$};
  \end{scope}

  \begin{scope}[xshift=9cm]
    \node[anchor=south, font=\footnotesize, align=center] at (3,2.3)
      {rejection event $A^c$\\[-2pt]
       probability $\dTV{\law(\xi),\law(\Psi^n_{z,\tilde z}(\xi))}$};

    \draw[thick] (0,0) -- (1.0,0.25) -- (2.0,0.10) -- (3.0,0.22)
                      -- (4.0,0.50) -- (5.0,0.72) -- (6.0,1.00);
    \foreach \p in {(0,0),(1.0,0.25),(2.0,0.10),(3.0,0.22),
                    (4.0,0.50),(5.0,0.72),(6.0,1.00)}
      \node[pt] at \p {};

    \draw[thick, dashed] (0,1.80) -- (1.0,1.95) -- (2.0,1.40) -- (3.0,1.65)
                                -- (4.0,1.30) -- (5.0,1.50) -- (6.0,1.70);
    \foreach \p in {(0,1.80),(1.0,1.95),(2.0,1.40),(3.0,1.65),
                    (4.0,1.30),(5.0,1.50),(6.0,1.70)}
      \node[pt] at \p {};

    \node[below, font=\footnotesize] at (0,0)    {$z$};
    \node[above, font=\footnotesize] at (0,1.8)  {$\tilde z$};
    \node[right, font=\footnotesize] at (6.0,1.0) {$Z^h_n$};
    \node[right, font=\footnotesize] at (6.0,1.70) {$\tilde Z^h_n$};
  \end{scope}
\end{tikzpicture}
\caption{The two branches of the rejection-sampling coupling produced by
\Cref{prop: coupling has correct marginals and is optimal}. On the acceptance event $A$ (left), the noise coupling
satisfies $\tilde\xi=\Psi^n_{z,\tilde z}(\xi)$ and the chains coincide at
time $n$. On the rejection event $A^c$ (right), the driving noise
$\tilde\xi$ is sampled independently and the chains generically diverge.
}
\label{fig:coupling-branches}
\end{figure}
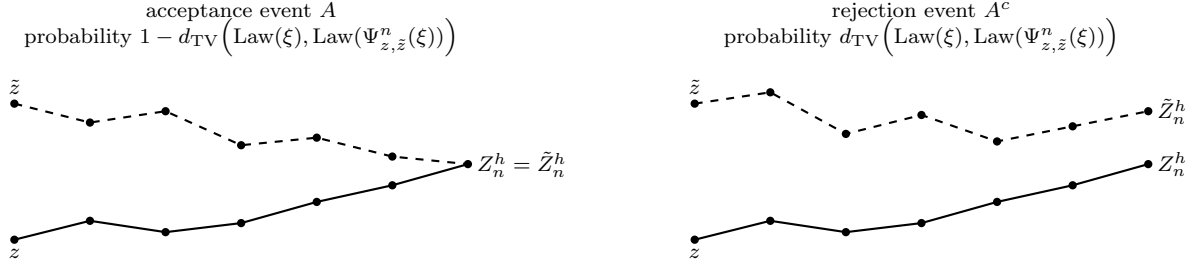

\begin{proposition}
\label{prop: coupling has correct marginals and is optimal}
The pair $(\xi,\tilde\xi)$ defined by~\eqref{eq: coupling definition}
satisfies $\xi\sim\calN(0,\id_{2dn})$, $\tilde\xi\sim\calN(0,\id_{2dn})$,
and
\begin{equation*}
    \dTV{\law(\xi),\,\law(\Psi^n_{z,\tilde z}(\xi))}
    = \P\bigl(\tilde\xi\neq\Psi^n_{z,\tilde z}(\xi)\bigr)
    = \P\bigl(\Psi^n_z(\xi)\neq\Psi^n_{\tilde z}(\tilde\xi)\bigr).
\end{equation*}
 \end{proposition}

Via the chain maps $\Psi^n_z$ and $\Psi^n_{\tilde z}$, the noise coupling
$(\xi,\tilde\xi)$ induces a coupling
$(\Psi^n_z(\xi),\Psi^n_{\tilde z}(\tilde\xi))$ of the chain laws
$\pi_n(\delta_z)$ and $\pi_n(\delta_{\tilde z})$
(Figure~\ref{fig:coupling-branches}). The coalescence-map
property gives the inclusion
$\{\tilde\xi=\Psi^n_{z,\tilde z}(\xi)\}\subseteq
\{\Psi^n_z(\xi)=\Psi^n_{\tilde z}(\tilde\xi)\}$, and the reverse
inclusion holds almost surely
(Lemma~\ref{lemma: no coincidental coalescence}, applied on the rejection
branch where $\tilde\xi$ is independent of $\xi$). Hence the noise-level
and chain-level disagreement events coincide a.s., and combined with the
maximal-coupling identity (the first equality in the proposition),
\begin{equation*}
    \P\bigl(\Psi^n_z(\xi)\neq\Psi^n_{\tilde z}(\tilde\xi)\bigr)
    = \dTV{\law(\xi),\,\law(\Psi^n_{z,\tilde z}(\xi))}.
\end{equation*}
The coupling characterization of TV distance then yields
\begin{equation*}
    \dTV{\pi_n(\delta_z),\,\pi_n(\delta_{\tilde z})}
    \le \dTV{\law(\xi),\,\law(\Psi^n_{z,\tilde z}(\xi))}.
\end{equation*}
This is the bound of
Lemma~\ref{lemma: coalescence maps provide TV distance bounds}, now
realized by an explicit coupling. The inequality is in general strict;
equality holds when $\nabla U$ is affine (equivalently, when $U$ is a
quadratic polynomial); see
Remark~\ref{rem: coalescence gives optimal TV for linear case}.

\medskip

In order to prove this proposition, we need the following basic result.
\medskip
\begin{lemma}
\label{lemma: no coincidental coalescence}
    Let $m,n\in\N_{>0}$, let $X$ be a $\R^m$-valued random variable and let $Y$ be a non-atomic $\R^n$-valued random variable, i.e., $\P(Y=y)=0$ for all $y\in\R^n$. Suppose that $X$ and $Y$ are independent. In addition, let $f\colon\R^m\to\R^n$ be a measurable map. Then
    \begin{equation*}
        \P(f(X)=Y)=0.
    \end{equation*}
\end{lemma}
\begin{proof}
    Let $\mu_X$ and $\mu_Y$ be the probability measures of X on $\R^m$ and $Y$ on $\R^n$ respectively. Since $X$ and $Y$ are independent, their joint probability measure on $\R^m\times\R^n$ is given by $\mu_{X,Y}=\mu_X\otimes\mu_Y$. Let $A=\{(x,y)\in\R^{m}\times\R^n:f(x)=y\}$; one readily verifies that $A$ is measurable. Thus, by Tonelli's theorem we have
    \begin{align*}
        \P(f(X)=Y)&=\int \ind_{A}(x,y) \,(\mu_X\otimes \mu_Y)(dx\,dy)=\int\int\ind_{A}(x,y)\,\mu_Y(dy)\,\mu_X(dx)\\
        &=\int\int\ind_{f(x)}(y)\,\mu_Y(dy)\,\mu_X(dx)=\int\mu_Y(\{f(x)\})\,\mu_X(dx)=0,
    \end{align*}
    where the last equality comes from the assumption that $\mu_Y$ is non-atomic.
\end{proof}

\begin{proof}[Proof of Proposition~\ref{prop: coupling has correct marginals and is optimal}]
Write $\varphi:=\varphi_{2dn}$ and $\psi:=\psi^n_{z,\tilde z}$ throughout
the proof, and let $B\in\mathcal{B}(\R^{2dn})$.

\emph{Step 1: marginal of $\tilde\xi$.} Splitting by the acceptance event,
\begin{equation*}
    \P(\tilde\xi\in B)
    = \P(\{\tilde\zeta\in B\}\cap A) + \P(\{\eta_N\in B\}\cap A^c).
\end{equation*}
For the first term, since $U$ is independent of $\tilde\zeta$ and
$\tilde\zeta$ has density $\psi$,
\begin{equation}\label{eq: acceptance density}
    \P(\{\tilde\zeta\in B\}\cap A)
    = \E\sqbr{\ind_{\{\tilde\zeta\in B\}}\br{1\wedge\frac{\varphi(\tilde\zeta)}{\psi(\tilde\zeta)}}}
    = \int_B (\varphi\wedge\psi)(y)\,dy;
\end{equation}
taking $B=\R^{2dn}$ yields $\P(A)=\int(\varphi\wedge\psi)\,dy$. Define the
per-trial acceptance events
\begin{equation*}
    B_j := \l\{V_j\le 1-\frac{\psi(\eta_j)}{\varphi(\eta_j)}\r\},
    \qquad j\ge 1,
\end{equation*}
so that $N=\inf\{j\ge 1:B_j\}$. The same calculation as
in~\eqref{eq: acceptance density}, applied to $\eta_1$ in place of
$\tilde\zeta$, gives
\begin{equation}\label{eq: rejection density}
    \P(\{\eta_1\in B\}\cap B_1)
    = \int_B \bigl(\varphi(y)-\varphi(y)\wedge\psi(y)\bigr)\,dy,
\end{equation}
and in particular $\P(B_1) = 1-\int(\varphi\wedge\psi)\,dy = \P(A^c)$.
Since the sequence $(\eta_j,V_j)_{j\ge 1}$ is independent of $(\xi,U)$,
the stopping time $N$ is geometrically distributed on the
$B_j$-trials, and $\eta_N$ is independent of $\ind_{A^c}$ with law
\begin{equation*}
    \P(\eta_N\in B)
    = \P(\eta_1\in B\mid B_1)
    = \frac{1}{\P(B_1)}\int_B\bigl(\varphi(y)-\varphi(y)\wedge\psi(y)\bigr)\,dy.
\end{equation*}
Hence
$\P(\{\eta_N\in B\}\cap A^c) = \int_B(\varphi-\varphi\wedge\psi)(y)\,dy$,
and summing the two contributions,
\begin{equation*}
    \P(\tilde\xi\in B) = \int_B \varphi(y)\,dy,
\end{equation*}
so $\tilde\xi\sim\calN(0,\id_{2dn})$.

\emph{Step 2: disagreement equalities.} On $A$, by construction
$\tilde\xi=\tilde\zeta=\Psi^n_{z,\tilde z}(\xi)$, and the coalescence-map
property forces $\Psi^n_z(\xi)=\Psi^n_{\tilde z}(\tilde\xi)$. On $A^c$,
$\tilde\xi=\eta_N$; by~\eqref{eq: rejection density}, $\eta_N$ has a
density (in particular it is non-atomic) and is independent of $\xi$, so
Lemma~\ref{lemma: no coincidental coalescence} gives
\begin{equation*}
    \P\bigl(\tilde\xi=\Psi^n_{z,\tilde z}(\xi),\,A^c\bigr) = 0.
\end{equation*}
Moreover, $\Psi^n_z(\xi)$ is absolutely continuous with respect to
Lebesgue measure on $\R^{2d}$
(by~\eqref{eq: OBABO as linear Markov chain} its law is the convolution
of a Dirac mass with non-degenerate Gaussian increments) and is
independent of $\eta_N$, so a second application of
Lemma~\ref{lemma: no coincidental coalescence} gives
\begin{equation*}
    \P\bigl(\Psi^n_z(\xi)=\Psi^n_{\tilde z}(\eta_N),\,A^c\bigr) = 0.
\end{equation*}
Combining, the disagreement events at the noise and chain levels both
coincide a.s.\ with $A^c$:
\begin{equation*}
    \P\bigl(\tilde\xi\neq\Psi^n_{z,\tilde z}(\xi)\bigr)
    = \P(A^c)
    = \P\bigl(\Psi^n_z(\xi)\neq\Psi^n_{\tilde z}(\tilde\xi)\bigr).
\end{equation*}

\emph{Step 3: TV identification.} The quantity
$\P(A^c)=1-\int(\varphi\wedge\psi)\,dy$ is the density-overlap form of
the total variation distance,
\begin{equation*}
    \dTV{\law(\xi),\,\law(\Psi^n_{z,\tilde z}(\xi))}
    = \dTV{\varphi,\psi}
    = 1-\int(\varphi\wedge\psi)\,dy
\end{equation*}
(see, e.g.,~\cite[Section~I.5]{Lindvall:2002}), which completes the
proof.
\end{proof}

\begin{remark}
\cref{prop: coupling has correct marginals and is optimal} shows that the coupling $(\xi,\tilde\xi)$ between the two standard Gaussians $\xi$, $\tilde \xi$ that we construct above is such that 
\begin{equation*}
    \dTV{\law(\xi),\law\br{\Psi^n_{z,\tilde z}(\xi)}}=\P\br{Z^h_n\neq\tilde Z^h_n},
\end{equation*}
where $Z^h_n=\Psi^n_z(\xi)$ and $\tilde Z^h_n=\Psi^n_{\tilde z}(\tilde\xi)$ are the final values of the two OBABO discretizations with initial values $z$, $\tilde z$ and increments $\xi$ and $\tilde \xi$ respectively. The coupling $(\xi,\tilde\xi)$ thus induces a coupling $(Z^h_k,\tilde Z^h_k)_{0\leq k\leq n}=(\Psi^k_{z}(\xi_1,\dots\xi_k),\Psi^k_{\tilde z}(\tilde \xi_1,\dots\tilde\xi_k))_{0\leq k\leq n}$ of the Markov chains that has the property that $\P(Z^h_n\neq \tilde Z^h_n)$ matches the TV bound of \cref{thm: TV bound via coalescence map}. This coupling is \emph{non-Markovian} in the sense that the joint chain $(Z^h_k,\tilde Z^h_k)_{0\leq k\leq n}$ is not a Markov chain with respect to the filtration $\F_k\coloneqq\sigma(\{(Z^h_i,\tilde Z^h_i)\mid 0\leq i \leq k\})$. This is due to the fact that the vector of increments $\tilde\xi$ depends on the entire vector of increments $\xi$ at once. \par 
Another aspect of the coupling constructed above is that if $z\neq \tilde{z}$, then typically
\begin{equation*}
\P( Z_k^h = \tilde{Z}_k^h) = 0,\quad k\in \{1,\ldots,n-1\}.
\end{equation*}
In other words, the processes $(Z_k^h)_{k=0}^{n}$ and $(Z_k^h)_{k=0}^{n}$ fail to meet prior to the final time $n$. This is due to the fact that on $A$ it holds that
\begin{equation*}
\Psi_z^{k}(\xi) = \Psi_{\tilde{z}}^k(\xi) + y_k
\end{equation*}
(see Lemma/Definition~\ref{lemdef: coalescence map defined by trajectory}), and typically $y_k\neq 0$ for $k\in \{1,\ldots, n-1\}$ (see~\eqref{eq: def Ek} and~\eqref{eq: y_k choices optimized trajectory}). In particular, the coupling constructed above is designed specifically for $n$ time steps, and cannot be extended to the whole time-line in an obvious way.\par
The non-Markovianity is unavoidable: Section~\ref{sec: non-optimality Markovian couplings}
shows that no Markovian coupling achieves the optimal asymptotic TV
decay for the linear kinetic Langevin equation.
\end{remark}

\begin{remark}
    \label{remark: extension coalescence coupling to other settings}
    The construction of the above coupling relies solely on the fact that the coalescence mapping $\Psi^n_{z,\tilde z}$ is a diffeomorphism. Recall from~\cref{rem: other splitting schemes} that other gHMC methods as well as other splitting methods give rise to diffeomorphic coalescence maps. In particular, the same reasoning can be applied to obtain a coupling for these sampling methods. However, in order to explicitly implement these couplings, we need an explicit (closed form) expression of the corresponding coalescence map $\Psi^n_{z,\tilde z}$. Out of the splitting schemes mentioned in \cref{rem: other splitting schemes}, such an explicit expression of $\Psi^n_{z,\tilde z}$ is only available for the OBABO and the BOAOB schemes.
\end{remark}

\section{There are no asymptotically optimal Markovian couplings}
\label{sec: non-optimality Markovian couplings}

In the search for explicit couplings achieving the correct asymptotic TV
decay rate, Markovian couplings are a natural first candidate: they have
been used successfully for TV mixing of the overdamped Langevin
dynamics~\cite{Eberle:2016} and for Wasserstein mixing of the kinetic
Langevin dynamics~\cite{EberleEtAl:2019}. The main result of this
section is that this approach fails for the kinetic Langevin equation
with quadratic potential:
Theorems~\ref{thm: Markovian couplings have the wrong asymptotic behavior}
and~\ref{thm: discrete Markovian couplings have the wrong asymptotic behavior}
below show that no Markovian coupling -- continuous-time or discrete --
can achieve the optimal asymptotic TV decay rate. The argument
follows~\cite{BanerjeeKendall:2016}, where an analogous suboptimality is
established for the Kolmogorov diffusion (the joint process of a
standard Brownian motion and its time integral).

Let us briefly introduce the setting. We consider the kinetic Langevin equation \eqref{eq: general kinetic Langevin} with the (isotropic) quadratic potential $U(x)=\alpha |x|^2$, where $\alpha\geq 0$. This results in the following $2d$-dimensional SDE:
\begin{align}
\label{eq: kinetic Langevin with quadratic potential}
\begin{split}
    dX_t&=V_t\,dt,\\ 
    dV_t&=-\alpha X_t\,dt-\gamma V_t\,dt+\sqrt{2\gamma}\,dW_t,
\end{split}
\end{align}
where $W$ is a standard $d$-dimensional Brownian motion. Throughout this section we fix $x,\tilde{x},v,\tilde{v}\in \R^d$ and let $(Z_t)_{t\geq 0}=(X_t,V_t)_{t\geq 0}$ and $(\tilde{Z}_t)_{t\geq 0}=(\tilde{X}_t,\tilde{V}_t)_{t\geq 0}$ denote solutions to~\eqref{eq: kinetic Langevin with quadratic potential} with initial values $z=(x,v)$ and $\tilde{z}=(\tilde{x},\tilde{v})$, respectively. We set $\Delta x = \tilde{x}-x$ and $\Delta v = \tilde {v} - v$. We also define $\lambda_{-},\lambda_{+}\in \C$ by
\begin{equation}\label{eq:ev_of_A}
    \lambda_\pm=-\frac{\gamma}{2}\pm\frac{1}{2}\sqrt{\gamma^2-4\alpha};
\end{equation}
these are the eigenvalues of the drift matrix associated with~\eqref{eq: kinetic Langevin with quadratic potential} (see also~\eqref{eq: kinetic Langevin with quadratic potential (as matrix)} below).

\medskip
\begin{definition}
A coupling $\mu$ of the joint processes $(Z_t,\tilde Z_t)_{t\geq 0}$ is called \textit{Markovian} if $(Z_t,\tilde Z_t)_{t\geq0}$ is a Markov process under $\mu$ (with respect to the filtration $(\F_t)_{t\geq0}$ generated by $(Z_t,\tilde{Z}_t)_{t \geq 0}$). A coupling $\mu$ satisfies the \emph{now-equals-forever} property if for every $0\leq s\leq t$ we have that $\mu(\{Z_s=\tilde Z_s\}\cap\{Z_t\neq\tilde Z_t\})=0$.
\end{definition}

\begin{remark}
The \emph{now-equals-forever} property seems to have been introduced in~\cite{Rosenthal:1997}, couplings with this property are also sometimes called \emph{sticky}. Somewhat confusingly, the term \emph{faithful} is also sometimes used for the \emph{now-is-forever} property; this concept was also introduced in~\cite{Rosenthal:1997} but with a distinctly different meaning.
\end{remark}

\begin{theorem}
\label{thm: Markovian couplings have the wrong asymptotic behavior}
Assume that $\gamma^2>4\alpha\geq 0$ and assume $  \lambda_{-} \Delta x - \Delta v =0$ and $\Delta x\neq 0$. 
Then there exists a constant $C>0$ such that for all $t>0$ one has
\begin{equation}\label{eq:tv_dist_Gauss}
    \dTV{\law(Z_t),\law(\tilde Z_t)} \leq C e^{\lambda_{-} t} |\Delta z|.
\end{equation}
On the other hand, for all choices of $\Delta x,\Delta v\in\R^{d}$ and for every Markovian coupling $\mu$ of $(Z_t,\tilde{Z}_t)_{t\geq 0}$ there exists constants $t_\mu,c_{\mu}>0$ such that for all $t\geq t_\mu$ one has
\begin{equation}\label{eq:lb_couplingprob}
    \mu(Z_t\neq \tilde{Z}_t) \geq c_{\mu} \min\big(t^{-\nicefrac{1}{2}}, e^{\lambda_{+} t}\big).
\end{equation}
If moreover $\mu$ satisfies the now-equals-forever property, then \eqref{eq:lb_couplingprob} holds true for all $t>0$.
\end{theorem}
\medskip 
Note that if $\lambda_{-}\Delta x - \Delta v=0$, then $\Delta z = (\Delta x, \Delta v)$ is in the eigenspace of the drift operator corresponding to eigenvalue $\lambda_{-}$, i.e.\, $\Delta z$ lies in the direction where the drift of $Z$ has the steepest decline. This results in the decline in TV distance given by~\eqref{eq:tv_dist_Gauss}. However, as one sees in the proofs below, a Markovian coupling cannot result in a meet-up if it stays in this eigenspace, which results in the suboptimal probability of not meeting~\eqref{eq:lb_couplingprob}. Note moreover that $\lambda_{+}=0$ if and only if $\alpha=0$; this is the situation that the probability of not meeting is bounded from below by $c_{\mu} t^{-\nicefrac{1}{2}}$.\par 
It is typically easier to construct couplings for discrete-time processes than for continuous-time processes. While discrete-time couplings indeed provide some more `wiggle room', the following theorem demonstrates that as the step size tends to $0$, the meeting probability of discrete Markovian couplings exhibits the same issues as in the continuous time setting.

For $h>0$ and $k\in \N$ we set $Z^h_k=Z_{hk}$ and $\tilde{Z}^h_k =\tilde{Z}_{hk}$; the Markov processes $(Z^h_k)_{k\in\N}$ and $(\tilde{Z}^h_k)_{k\in \N}$ represent discrete approximations of $(Z_t)_{t\geq 0}$ and $(\tilde{Z}_t)_{t \geq 0}$ for which one samples exactly from the distribution of the increments.

\medskip
\begin{definition}\label{def: discrete Markovian coupling}
    A coupling $\mu_h$ of the joint chain $(Z^h_k,\tilde Z^h_k)_{k\in\N}$ is called \textit{Markovian} if $(Z^h_k,\tilde Z^h_k)_{k\in\N}$ is a Markov chain under $\mu_h$ (with respect  to the filtration $(\F_k)_{k\in\N}$ generated by $(Z^h_k,\tilde Z^h_k)_{k\in\N}$). A coupling $\mu_h$ satisfies the \emph{now-equals-forever} property if for every $0\leq k\leq \ell$ we have that $\mu_h(\{Z_k=\tilde Z_k\}\cap\{Z_\ell\neq\tilde Z_\ell\})=0$.
\end{definition}

\medskip
\begin{theorem}
\label{thm: discrete Markovian couplings have the wrong asymptotic behavior}
Assume that $\gamma^2>4\alpha\geq 0$ and assume $  \lambda_{-} \Delta x - \Delta v =0$ and $\Delta x \neq 0$. 
Then there exists a constant $C>0$ such that for all $h>0$, $k\in \N$ one has
\begin{equation}\label{eq:tv_dist_Gauss_disc}
    \dTV{\law(Z_k^h),\law(\tilde Z_k^h)} \leq C e^{\lambda_{-} hk} |\Delta z|.
\end{equation}
On the other hand, for all choices of $\Delta x,\Delta v\in\R^{d}$, every $h>0$ and every Markovian coupling $\mu_h$ of $(Z_k^h,\tilde{Z}_k^h)_{k\in \N}$ there exist constants $k_{\mu_h}\in\N$, $c_{\mu_h} >0$ and $c>0$ (with $c$ independent of $\mu_h$) such that for all $k\geq k_{\mu_h}$ one has
\begin{equation}\label{eq:lb_couplingprob_disc}
    \mu_h(Z_k^h\neq \tilde{Z}_k^h) \geq c\min\big( c_{\mu_h}(hk+1)^{-\nicefrac{1}{2}}, c_{\mu_h} e^{\lambda_{+} hk},  h^{-1} e^{\lambda_{-} hk} \big).
\end{equation}
If moreover $\mu_h$ satisfies the now-equals-forever property, then \eqref{eq:lb_couplingprob_disc} holds true for all $k\in\N$. 
Notably, if $\lambda_{-} \Delta x - \Delta v =0$ and $\Delta x \neq 0$, then one has for every $h>0$ and every Markovian coupling $\mu_h$ of $(Z_k^h,\tilde{Z}_k^h)_{k\in \N}$ that
\begin{equation}\label{eq:lb_couplingprob_disc_limsup}
    \liminf_{k\rightarrow \infty} 
    \frac{\mu_h(Z_k^h\neq \tilde{Z}_k^h)}{\dTV{\law(Z_k^h),\law(\tilde Z_k^h)}}
    \geq \frac{c}{C h},
\end{equation}
where $c,C>0$ are constants independent of $h$ and $\mu_h$.
\end{theorem}

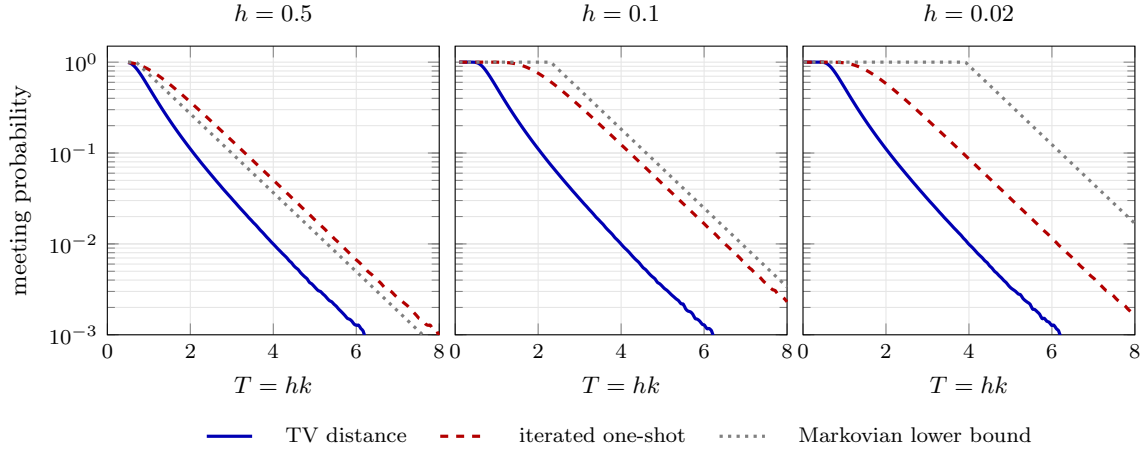
\begin{figure}[t]
\centering
\begin{tikzpicture}
\tikzset{
    declare function={
        asT(\x)     = 1/(1 + 0.3275911*\x);
        myerf(\x)   = 1 - ( asT(\x)*( 0.254829592
                          + asT(\x)*(-0.284496736
                          + asT(\x)*( 1.421413741
                          + asT(\x)*(-1.453152027
                          + asT(\x)*  1.061405429 )))) ) * exp(-(\x)^2);
        sxx(\T)     = 2*\T - 4*(1-exp(-\T)) + (1-exp(-2*\T));
        sxv(\T)     = (1-exp(-\T))^2;
        svv(\T)     = 1-exp(-2*\T);
        Bsq(\T)     = (svv(\T) + 2*sxv(\T) + sxx(\T))
                      / (sxx(\T)*svv(\T) - sxv(\T)^2);
        TV(\T)         = myerf( exp(-\T)*sqrt(Bsq(\T)) / (2*sqrt(2)) );
        oneshot(\T,\h) = myerf( sqrt(Bsq(\h))
                               * sqrt((1-exp(-2*\h))/(exp(2*\T)-1))
                               / (2*sqrt(2)) );
        markov(\T,\h)  = min(1, exp(-\T)/\h);
    },
}
\begin{groupplot}[
    group style={
        group size=3 by 1,
        horizontal sep=6pt,
        ylabels at=edge left,
        yticklabels at=edge left,
    },
    width=0.36\linewidth, height=5.4cm,
    xmin=0, xmax=8,
    ymode=log, ymin=1e-3, ymax=1.5,
    xlabel={$T = hk$},
    ylabel={meeting probability},
    label style={font=\small},
    tick label style={font=\footnotesize},
    title style={font=\small},
    grid=both, grid style={gray!20},
    every axis plot/.append style={very thick, samples=140, smooth},
]

\nextgroupplot[title={$h = 0.5$}]
\addplot[blue!70!black,             domain=0.5:8] {TV(x)};
\addplot[red!70!black, dashed,      domain=0.5:8] {oneshot(x, 0.5)};
\addplot[gray, dotted, line width=1.2pt, domain=0.5:8] {markov(x, 0.5)};

\nextgroupplot[title={$h = 0.1$}]
\addplot[blue!70!black,             domain=0.1:8] {TV(x)};
\addplot[red!70!black, dashed,      domain=0.1:8] {oneshot(x, 0.1)};
\addplot[gray, dotted, line width=1.2pt, domain=0.1:8] {markov(x, 0.1)};

\nextgroupplot[
    title={$h = 0.02$},
    legend to name=fig:sharpness:legend,
    legend style={
        legend columns=-1,
        column sep=1em,
        font=\footnotesize,
        draw=none, fill=none,
    },
]
\addplot[blue!70!black,             domain=0.02:8] {TV(x)};
\addlegendentry{TV distance}
\addplot[red!70!black, dashed,      domain=0.02:8] {oneshot(x, 0.02)};
\addlegendentry{iterated one-shot}
\addplot[gray, dotted, line width=1.2pt, domain=0.02:8] {markov(x, 0.02)};
\addlegendentry{Markovian lower bound}

\end{groupplot}

\node[anchor=north, yshift=-2pt]
  at (group c2r1.below south) {\ref{fig:sharpness:legend}};
\end{tikzpicture}
\caption{Sharpness of
Theorem~\ref{thm: discrete Markovian couplings have the wrong asymptotic behavior}.
Meeting probability vs.\ total time $T=hk$ on log scale, for the
exact discretization of the free kinetic Langevin equation
($\alpha=0$, $\gamma=1$)with $\Delta z=(\Delta x,-\gamma\Delta x)$ in the $\lambda_-$
eigenspace; the plotted curves correspond to $|\Delta x|=1$ (the
proportionality constants in the lower bound are suppressed). \emph{Solid:} TV distance
$\dTV{\pi_k^h(\delta_z),\pi_k^h(\delta_{\tilde z})}$, evaluated via
the closed-form Gaussian expression~\eqref{eq: Gaussian TV distance}.
The non-Markovian coupling of
Section~\ref{sec: non-Markovian coupling} reproduces this curve exactly
(Remark~\ref{rem: coalescence gives optimal TV for linear case}).
\emph{Dashed:} meeting probability of the iterated one-shot coupling,
evaluated via the closed-form expression of
Theorem~\ref{thm: iterated one-shot coupling probability}
applied to Example~\ref{example: one shot matching lower bound}.
\emph{Dotted:} the Markovian lower bound
$h^{-1}e^{\lambda_{-}hk}$ from~\eqref{eq:lb_couplingprob_disc}
(constant suppressed). As $h\downarrow 0$, the dashed and dotted
curves stay close to $1$ over a longer interval before decaying,
illustrating the $1/h$ degradation
of~\eqref{eq:lb_couplingprob_disc_limsup}.}
\label{fig:sharpness-impossibility}
\end{figure}

\medskip 
Unfortunately, we are not able to control how the constant $c_{\mu_h}$ in~\eqref{eq:lb_couplingprob_disc} depends on $\mu_h$. Nevertheless, the lower bound~\eqref{eq:lb_couplingprob_disc}
illustrates that while asymptotically optimal meeting probabilities can
be achieved for a discrete Markovian coupling when fixing a step size
$h$ and letting $hk\rightarrow\infty$, one cannot obtain optimal
meeting probabilities for a fixed time $hk$ when letting $h\downarrow 0$;
see Figure~\ref{fig:sharpness-impossibility}. \par 
The upper bound~\eqref{eq:tv_dist_Gauss} is established in Section~\ref{subsec: ub tvdist Gauss} below (note that $\eqref{eq:tv_dist_Gauss}$ implies \eqref{eq:tv_dist_Gauss_disc}). The lower bounds~\eqref{eq:lb_couplingprob} and~\eqref{eq:lb_couplingprob_disc} are established in Sections~\ref{subsec: Markovian suboptimality continuous quadratic} and~\ref{subsec: Markovian suboptimality discrete free} below. A key ingredient for both proofs relies on the analysis of a process that is zero if and only if $Z$ is in the eigenspace of $\lambda_{-}$. This process, which we denote by $(Q_t)_{t\geq 0}$, is introduced and analyzed in Section~\ref{subsec:Qprocess}. We also note that the proof of the lower bounds for the case $\alpha=0$ deviates slightly from the proof for the case $\alpha>0$; we provide a proof for~\eqref{eq:lb_couplingprob} in the case $\alpha>0$ and a proof for~\eqref{eq:lb_couplingprob_disc} in the case $\alpha = 0$. The remaining proofs are entirely analogous and are left to the reader.

\medskip
\begin{remark}
    The assumption that we are in the overdamped regime, i.e. that $\gamma^2>4\alpha$, is crucial in Theorems~\ref{thm: Markovian couplings have the wrong asymptotic behavior} and~\ref{thm: discrete Markovian couplings have the wrong asymptotic behavior}. Indeed, without this assumption, the rate of the exponential contraction of the means can be the same in all directions. This would prevent us from finding a subspace of initial values for which the TV distance between solutions contracts at a larger rate than elsewhere, which forms the key distinction between the TV distance and the probability of not meeting of Markovian couplings.
\end{remark}

\subsection{Proof of upper bound\texorpdfstring{~\eqref{eq:tv_dist_Gauss}}{ }}\label{subsec: ub tvdist Gauss}

The linear SDE \eqref{eq: kinetic Langevin with quadratic potential} can be represented in terms of $(Z_t)_{t\geq0}=(X_t,V_t)_{t\geq0}$ as
\begin{equation}
\label{eq: kinetic Langevin with quadratic potential (as matrix)}
    dZ_t=AZ_t\,dt+B\,d\bar W_t,
\end{equation}
where $(\bar W_t)_{t\geq0}$ is now a $2d$-dimensional Brownian motion and $A,B\in\R^{2d\times 2d}$ are given by
\begin{equation*}
    A=\begin{pmatrix}  0& \id_d\\ -\alpha\id_d & -\gamma\id_d \end{pmatrix},\qquad B=\begin{pmatrix}  0& 0\\ 0 & \sqrt{2\gamma}\id_d \end{pmatrix}
\end{equation*}
The solution to this SDE at time $t>0$ is given by
\begin{equation*}
    Z_t=e^{tA}Z_0+\int_0^te^{(t-s)A}Bd \bar W_s.
\end{equation*}
In the case of a deterministic initial value $Z_0=z=(x,v)\in\R^{2d}$ this means that the solution is distributed as $Z_t\sim\calN(e^{tA}z,\Sigma_t)$, where the covariance $\Sigma_t\in\R^{2d\times 2d}$ is given by
\begin{equation*}
\Sigma_t=\int_0^t e^{sA}BB^Te^{sA^T}ds.
\end{equation*}
By \eqref{eq: Gaussian TV distance}, the TV distance of two solutions $(Z_t)_{t\geq0}$, $(\tilde Z_t)_{t\geq0}$ of \eqref{eq: kinetic Langevin with quadratic potential} with initial values $z,\tilde z\in\R^{2d}$ at time $t>0$ is therefore
\begin{equation}
\label{eq: TV distance kinetic Langevin with Gaussian potential}
    \dTV{\law(Z_t),\law(\tilde Z_t)}=2\Phi\br{\frac{|\Sigma_t^{-1/2}e^{tA}\Delta z|}{2}}-1,
\end{equation}
where $\Delta z=\tilde z-z$. 

Note that the eigenvalues of the matrix $A$ are given by~\eqref{eq:ev_of_A}. If $\lambda_{-}\Delta x - \Delta v = 0$, then $\Delta z$ is in the eigenspace of $\lambda_{-}$ so that $e^{tA}\Delta z = e^{\lambda_{-} t}\Delta z$. As $2\Phi(x) -1 \leq \frac{\sqrt{2}}{\sqrt{\pi}} x$ for all $x\geq 0$, all that remains to establish~\eqref{eq:tv_dist_Gauss} is to prove that the operator norm of $\Sigma_t^{-1/2}$ appearing in \eqref{eq: TV distance kinetic Langevin with Gaussian potential} can be bounded uniformly in time. For this, we employ the following lemma.
\medskip
\begin{lemma}
\label{lemma: Sigma_t operator norm uniform bound}
    Let $(M_t)_{t>0}$ be sequence of $d$-dimensional positive definite matrices such that $M_t\succeq M_s$ for all $t\geq s$. Then for all $t\geq s$ we have
    \begin{equation*}
        \big\|M_t^{-1/2}\big\|\leq\big\|M_s^{-1/2}\big\|.
    \end{equation*}
\end{lemma}
\begin{proof}
    The matrix $M_t$ is positive definite for any $t>0$ and is therefore symmetric and invertible. Hence,
    \begin{equation*}
        \big\|M_t^{-1/2}\big\|^2=\lambda_{\max}(M_t^{-1})=\br{\lambda_{\min}(M_t)}^{-1},
    \end{equation*}
    where $\lambda_{\max}(M_t^{-1}),\lambda_{\min}(M_t)>0$ denote the largest and smallest eigenvalue of $M_t^{-1}$ and $M_t$ respectively. Suppose that $t\geq s$, then for all $z\in\R^{d}$ we have 
    \begin{equation*}
        z^TM_tz=z^TM_{s}z+z^T(M_t-M_{s})z\geq z^TM_{s}z.
    \end{equation*}
    Taking the minimum over $z\in\R^{d}$ with $|z|=1$ on both sides and using that $M_t$ and $M_{s}$ are symmetric we obtain
    \begin{equation*}
        \lambda_{\min}(M_t)=\min_{|z|=1}z^TM_tz\geq\min_{|z|=1}z^TM_{s}z=\lambda_{\min}(M_{s}).
    \end{equation*}
    This means that for all $t\geq s$ we have
    \begin{equation*}
        \big\|M_t^{-1/2}\big\|=\br{\lambda_{\min}(M_t)}^{-1/2}\leq\br{\lambda_{\min}(M_{s})}^{-1/2}=\big\|M_s^{-1/2}\big\|,
    \end{equation*}
    which completes the proof.
\end{proof}
The sequence of matrices $(\Sigma_t)_{t>0}$ satisfies the assumptions of the lemma above, as for ${0<s<t}$ we have that $\Sigma_t-\Sigma_s=\int_s^t e^{rA}BB^Te^{rA^T}dr$, which is positive definite. In view of the discussion preceding Lemma~\ref{lemma: Sigma_t operator norm uniform bound}, this completes the proof of the upper bound~\eqref{eq:tv_dist_Gauss}.

\subsection{The process \texorpdfstring{$(Q_t)_{t\geq 0}$}{Q}}\label{subsec:Qprocess}

For a solution $(Z_t)_{t\geq0}=(X_t,V_t)_{t\geq0}$ to \eqref{eq: kinetic Langevin with quadratic potential}, let 
\begin{equation*}
    Q_t=X_t-\lambda_-^{-1}V_t.
\end{equation*}
This choice is exactly such that $Q_t=0$ if and only if $Z_t$ is contained in the eigenspace of the drift matrix $A$ corresponding to $\lambda_-$ (see~\eqref{eq:ev_of_A} and \eqref{eq: kinetic Langevin with quadratic potential (as matrix)}).

We can rewrite the SDE \eqref{eq: kinetic Langevin with quadratic potential} in terms of $(X_t,Q_t)_{t\geq0}$ by noting that
\begin{align*}
    dQ_t&=dX_t-\lambda_-^{-1}\,dV_t\\
    &=\alpha\lambda_-^{-1}X_t\,dt+(1+\gamma\lambda_-^{-1})V_t\,dt-\sqrt{2\gamma}\lambda_-^{-1}\,dW_t\\
    &=(\alpha\lambda_-^{-1}+\gamma+\lambda_-)X_t\,dt-(\gamma+\lambda_-)Q_t\,dt-\sqrt{2\gamma}\lambda_-^{-1}\,dW_t
\end{align*}
Quick calculations show that $\lambda_-(\alpha\lambda_-^{-1}+\gamma+\lambda_-)=\alpha+\gamma\lambda_-+\lambda_-^2=0$ and $-(\gamma+\lambda_-)=\lambda_+$, so that \eqref{eq: kinetic Langevin with quadratic potential} is equivalent to
\begin{subequations}
\label{eq:dX_dQ_group} 
\begin{align}
    dX_t &= \lambda_-X_t\,dt - \lambda_-Q_t\,dt, \label{eq:dX_SDE} \\
    dQ_t &= \lambda_+Q_t\,dt - \frac{\sqrt{2\gamma}}{\lambda_{-}}\,dW_t. \label{eq:dQ_SDE}
\end{align}
\end{subequations}
We see that on its own, $(Q_t)_{t\geq0}$ is an Ornstein--Uhlenbeck process, with exact solution 
\begin{equation}\label{eq: Q_asOU}
    Q_t=e^{\lambda_+ t}q-\frac{\sqrt{2\gamma}}{\lambda_{-}}\int_0^te^{\lambda_+(t-s)}\,dW_s,
\end{equation}
where $q=x-\lambda_-^{-1}v$ is its initial value. In particular, if $(Q_t)_{t\geq0}$, $(\tilde Q_t)_{t\geq0}$ satisfy~\eqref{eq: Q_asOU} (in the distributional sense) with initial values $q,\tilde q\in \R^d$, then clearly $q=\tilde q$ implies that $\law(Q_t)=\law(\tilde{Q}_t)$ for all $t\geq 0$. Moreover, $Q_t$ is normally distributed with mean $e^{\lambda_+t}q$ and covariance that for $\alpha>0$ equals
\begin{equation*}
    \Cov(Q_t)=\frac{2\gamma}{\lambda_{-}^2}\br{\int_0^te^{2\lambda_+s}\,ds}\id_d=-\frac{\gamma}{\lambda_+\lambda_{-}^2}(1-e^{2\lambda_+t})\id_d=-\frac{\gamma}{\alpha\lambda_{-}}(1-e^{2\lambda_+t})\id_d,
\end{equation*}
where we use that $\lambda_+\lambda_-=\alpha$, while for $\alpha=0$ we likewise have covariance ${\Cov(Q_t)=2\gamma^{-1}t}$. Thus, by \eqref{eq: Gaussian TV distance} the processes $(Q_t)_{t\geq0}$, $(\tilde Q_t)_{t\geq0}$ satisfying~\eqref{eq: Q_asOU} with initial values $q,\tilde{q}\in \R^d$ have a TV distance that for $\alpha>0$ is given by
\begin{equation}
    \label{eq: TV distance Q_t}
    \dTV{\law(Q_t),\law(\tilde Q_t)}=2\Phi\br{\frac{\kappa e^{\lambda_+ t}|\Delta q|}{2\sqrt{1-e^{2\lambda_+t}}}}-1=2\Phi\br{\frac{\kappa|\Delta q|}{2\sqrt{e^{-2\lambda_+t}-1}}}-1,
\end{equation}
where $\kappa=\sqrt{-\alpha\lambda_-/\gamma}$ and $\Delta q=\tilde q-q$, and for $\alpha=0$ is given by
\begin{equation}
    \label{eq: TV distance Q_t (alpha=0)}
    \dTV{\law(Q_t),\law(\tilde Q_t)}=2\Phi\br{\frac{|\Delta q|}{2\sqrt{2\gamma^{-1}t}}}-1.
\end{equation}

\subsection{Proof of the \texorpdfstring{lower bound \eqref{eq:lb_couplingprob} for $\alpha>0$}{first lower bound}}
\label{subsec: Markovian suboptimality continuous quadratic}

The following lemma demonstrates that the probability of not meeting $\mu(Z_t\neq \tilde Z_t)$ of a Markovian coupling eventually contracts at best exponentially in $t$ with rate $|\lambda_+|$. Note that the lower bound~\eqref{eq:lb_couplingprob} in Theorem~\ref{thm: Markovian couplings have the wrong asymptotic behavior} is a direct consequence of this lemma, the proof of which is an adaptation of the proof of Lemma 1 in \cite{BanerjeeKendall:2016} to our setting.
\medskip
\begin{lemma}
\label{lemma: Markovian coupling asymptotic behavior}
Assume that $\gamma^2>4\alpha>0$ and let $\mu$ be a Markovian coupling of $(Z_t,\tilde Z_t)_{t\geq 0}$ of two solutions to \eqref{eq: kinetic Langevin with quadratic potential} with initial values such that $\Delta z\neq0$. Then there exist constants $t_\mu\geq0$, $c_\mu>0$ such that for all $t\geq t_{\mu}$ we have 
\begin{equation}\label{eq: coupling bound}
    \mu(Z_t\neq\tilde Z_t)\geq c_\mu e^{\lambda_+t}.
\end{equation}
Furthermore, if $\mu$ satisfies the now-equals forever property, then we can pick $t_\mu=0$. 
\end{lemma}
\begin{proof}
Let $(Q_t)_{t\geq 0}$ and $(\tilde{Q}_t)_{t\geq 0}$ satisfy $Q_t = X_t - \lambda_{-}^{-1} V_t$ and $\tilde{Q}_t = \tilde{X}_t - \lambda_{-}^{-1} \tilde{V}_t$, $t\geq 0$.  
Suppose first that $\mu(Q_t=\tilde Q_t)=1$ for all $t\geq0$. Fubini's theorem implies that
\begin{equation}
    \mu(Q_t=\tilde Q_t \textup{ for almost every } t\geq0)=1.
\end{equation}
Since the paths of $Q$ and $\tilde Q$ are continuous this implies $Q$ and $\tilde Q$ are synchronously coupled $\mu$-almost surely. In particular, by \eqref{eq:dX_dQ_group} we have that 
\begin{equation*}
    \mu(\tilde X_t-X_t=e^{\lambda_-t}\Delta x \textup{ for all } t>0)=1,
\end{equation*}
showing that $\mu(Z_t\neq\tilde Z_t)=1$ for all $t>0$.

Now assume that there exists some $t_0\geq0$ such that $\mu(Q_{t_0}\neq\tilde Q_{t_0})>0$. Since the coupling $\mu$ is Markovian, the shifted process $(Q_{t_0+s},\tilde Q_{t_0+s})_{s\geq 0}$ conditioned on $\F_{t_0}$ constitutes a coupling between two OU processes that have different initial values with positive probability. 
Since $Z_t=\tilde Z_t$ implies that $Q_t=\tilde Q_t$, we have for every $t>t_0$ that
\begin{equation}
\label{eq: coupling probability simple bound (Markovian proof)}
    \mu(Z_t\neq\tilde Z_t)\geq \mu(Q_t\neq\tilde Q_t)=\E_\mu\sqbr{\E^{\F_{t_0}}_\mu\sqbr{\ind_{\{Q_t\neq\tilde Q_t\}}}}.
\end{equation}
The TV distance of the two OU processes $Q$ and $\tilde{Q}$ is given by \eqref{eq: TV distance Q_t}; for $t>t_0$ we have
\begin{equation}
\label{eq: dTV Q_t0 (Markovian proof)}
    \E^{\F_{t_0}}_\mu\sqbr{\ind_{\{Q_t\neq\tilde Q_t\}}}\geq2\Phi\br{\frac{\kappa|\tilde Q_{t_0}-Q_{t_0}|}{2\sqrt{e^{-2\lambda_+(t-t_0)}-1}}}-1.
\end{equation}
If $y\in\R^d$ satisfies $|y|\leq \tfrac{2}{\kappa}\sqrt{e^{-2\lambda_+(t-t_0)}-1}$, then we can bound
\begin{align}
\label{eq: Q_t TV lower bound (Markovian proof)}
    2\Phi\br{\frac{\kappa|y|}{2\sqrt{e^{-2\lambda_+(t-t_0)}-1}}}-1&=\P\br{\l|\calN(0,1)\r|\leq\frac{\kappa|y|}{2\sqrt{e^{-2\lambda_+(t-t_0)}-1}}}\nonumber\\
    &\geq\frac{\kappa}{\sqrt{2\pi e}}\frac{|y|}{\sqrt{e^{-2\lambda_+(t-t_0)}-1}},
\end{align}
using the fact that the standard normal density function can be bounded as $\varphi(z)\geq 1/\sqrt{2\pi e}$ for $z\in[-1,1]$.

Since $\lambda_+<0$, we have that $\E_\mu\sqbr{\ind_{\l\{|\tilde Q_{t_0}-Q_{t_0}|\leq \tfrac{2}{\kappa}\sqrt{e^{-2\lambda_+(t-t_0)}-1}\r\}}|\tilde Q_{t_0}-Q_{t_0}|}$ monotonically increases to $\E_\mu\sqbr{|\tilde Q_{t_0}-Q_{t_0}|}$ as $t\to\infty$, and thus there exists a $t_\mu>t_0$ such that for all $t\geq t_\mu$ we have that
\begin{equation}
\label{eq: Q_t difference restricted lower bound (Markovian proof)}
    \E_\mu\sqbr{\ind_{\l\{|\tilde Q_{t_0}-Q_{t_0}|\leq \tfrac{2}{\kappa}\sqrt{e^{-2\lambda_+(t-t_0)}-1}\r\}}|\tilde Q_{t_0}-Q_{t_0}|}\geq \frac{1}{2}\E_\mu\sqbr{|\tilde Q_{t_0}-Q_{t_0}|}=\frac{\sqrt{2\pi e}}{\kappa}c_\mu,
\end{equation}
where $c_\mu\coloneqq \tfrac{\kappa}{2\sqrt{2\pi e}}\E_\mu\sqbr{|\tilde Q_{t_0}-Q_{t_0}|}$. Note that $c_\mu>0$ by definition of $t_0$. 

Combining the above observations we see that the probability of not meeting at time $t\geq t_\mu$ can be bounded from below as follows
\begin{align*}
    \mu(Z_t\neq\tilde Z_t)&\stackrel{\eqref{eq: coupling probability simple bound (Markovian proof)}}{\geq}\E_\mu\l[\E_\mu^{\F_{t_0}}[\ind_{\{Q_t\neq\tilde Q_t\}}]\r]\\
    &\stackrel{\eqref{eq: dTV Q_t0 (Markovian proof)}}{\geq}\E_\mu\l[2\Phi\br{\frac{\kappa|\tilde Q_{t_0}-Q_{t_0}|}{2\sqrt{e^{-2\lambda_+(t-t_0)}-1}}}-1\r]\\
    &\geq\E_\mu\l[\ind_{\l\{|\tilde Q_{t_0}-Q_{t_0}|\leq \tfrac{2}{\kappa}\sqrt{e^{-2\lambda_+(t-t_0)}-1}\r\}}\l\{2\Phi\br{\frac{\kappa|\tilde Q_{t_0}-Q_{t_0}|}{2\sqrt{e^{-2\lambda_+(t-t_0)}-1}}}-1\r\}\r]\\
    &\stackrel{\eqref{eq: Q_t TV lower bound (Markovian proof)}}{\geq} \E_\mu\sqbr{\ind_{\l\{|\tilde Q_{t_0}-Q_{t_0}|\leq \tfrac{2}{\kappa}\sqrt{e^{-2\lambda_+(t-t_0)}-1)}\r\}}\frac{\kappa}{\sqrt{2\pi e}}\frac{|\tilde Q_{t_0}-Q_{t_0}|}{\sqrt{e^{-2\lambda_+(t-t_0)}-1}}}\\
    &\stackrel{\eqref{eq: Q_t difference restricted lower bound (Markovian proof)}}{\geq}\frac{c_\mu}{\sqrt{e^{-2\lambda_+(t-t_0)}-1}}\geq c_\mu e^{\lambda_+(t-t_0)}\geq c_\mu e^{\lambda_+t}.
\end{align*}
This completes the proof of \eqref{eq: coupling bound}. If in addition $\mu$ satisfies the now-equals forever property, then we have for all $0\leq t<t_\mu$ that $\mu(Z_t\neq\tilde Z_t)\geq\mu(Z_{t_\mu}\neq \tilde Z_{t_\mu})$. The bound \eqref{eq: coupling bound} can therefore be extended to all $t\geq0$ by appropriately scaling $c_\mu$.
\end{proof}

\subsection{Proof of the \texorpdfstring{lower bound \eqref{eq:lb_couplingprob_disc} for $\alpha=0$}{second lower bound}}\label{subsec: Markovian suboptimality discrete free}

Note that when $\alpha=0$, then the equation~\eqref{eq: kinetic Langevin with quadratic potential} admits an exact solution $(Z_t)_{t\geq0}=(X_t,V_t)_{t\geq0}$ for any initial value $z=(x,v)\in\R^{2d}$,
which is given by
\begin{equation*}
    X_t=x+\gamma^{-1}(1-e^{-\gamma t})v+\sqrt{2\gamma^{-1}}\int_0^t(1-e^{-\gamma(t-s)})\,dW_s,
\end{equation*}
\begin{equation*}
    V_t=e^{-\gamma t}v+\sqrt{2\gamma}\int_0^te^{-\gamma(t-s)}\,dW_s.
\end{equation*}
In particular, at time $t>0$ the solution $(Z_t)_{t\geq 0}$ to the potential free kinetic Langevin equation with initial value $z\in\R^{2d}$ is distributed as $Z_t\sim\mathcal{N}(e^{tA}z,\Sigma_t)$, where $e^{tA},\Sigma_t\in\R^{2d\times 2d}$ are given by
\begin{align}
\label{eq: free kinetic Langevin mean and covariance matrices}
\begin{split}
    e^{tA}&=
    \begin{pmatrix}
    \id_d & \gamma^{-1}\br{1-e^{-\gamma t}}\id_d\\
    0 & e^{-\gamma t}\id_d
    \end{pmatrix},\\
    \Sigma_t&=
    \begin{pmatrix}
    \br{2\gamma^{-1}t-4\gamma^{-2}\br{1-e^{-\gamma t}}+\gamma^{-2}(1-e^{-2\gamma t})}\id_d & \gamma^{-1}\br{1-e^{-\gamma t}}^2\id_d\\
    \gamma^{-1}\br{1-e^{-\gamma t}}^2\id_d & \br{1-e^{-2\gamma t}}\id_d
    \end{pmatrix}.
\end{split}
\end{align}
Moreover, in this case the eigenvalues of the drift matrix $A$ in~\eqref{eq:ev_of_A} are simply given by
\begin{equation*}
\lambda_{-} = -\gamma,\qquad \lambda_{+}=0.
\end{equation*}

Recall that given processes $(Z_t)_{t\geq 0}=(X_t,V_t)_{t\geq 0}$ and $(\tilde{Z}_t)_{t\geq 0}=(\tilde X_t,\tilde V_t)_{t\geq 0}$ we define $Q_t = X_t + \gamma^{-1}V_t$ and $\tilde{Q}_t = \tilde{X}_t + \gamma^{-1} \tilde V_t$, $t\geq 0$. Recall that $Q_t=0$ whenever $Z_t$ is in the eigenspace of the drift matrix $A$ corresponding to eigenvalue $\lambda_{-}=-\gamma$. We also define $Q^h_k = Q_{hk}$ and $\tilde{Q}^h_k = \tilde{Q}_{hk}$, $h>0$ and $k\in \N$. \par 
The main difference compared to the argument providing a lower bound in the continuous time setting is that the assumption that $\mu_h(\cap_{k=0}^\infty\{Q^h_k=\tilde Q^h_k\})=1$ and $\Delta x\neq 0$ does \emph{not} guarantee that $\mu_h(\cap_{k=0}^{\infty} \{ Z^h_k\neq \tilde Z^h_k\})=1$. This is inherent to the discrete set-up, where we can not control the behavior of the processes on the time intervals $(hk,h(k+1))$, on which the positions could be driven together. The value of the step size, however, crucially limits the meeting probability, due to the fact that the variance of the positions over these time intervals scales as $h^3$ for small $h$. This is demonstrated in Lemma~\ref{lem: hitting prob when in eigenspace on grid} below. First, however, we use arguments analogous to the proof of Lemma~\ref{lemma: Markovian coupling asymptotic behavior} to deal the situation that $\mu_h(\cap_{k=0}^\infty\{Q^h_k=\tilde Q^h_k\})<1$. The Lemmas~\ref{lemma: discrete Markovian coupling asymptotic behavior nonzero Q} and~\ref{lem: hitting prob when in eigenspace on grid} together provide the lower bound~\eqref{eq:lb_couplingprob_disc} when $\alpha = 0$.
\medskip
\begin{lemma}
\label{lemma: discrete Markovian coupling asymptotic behavior nonzero Q}
Assume that $\gamma^2> \alpha =0$, let $h>0$, and let $\mu_h$ be a Markovian coupling of $(Z_k^{h},\tilde{Z}_k^h)_{k\in \N} = (Z_{hk},\tilde Z_{hk})_{k\in \N}$ where $(Z_t,\tilde Z_t)_{t\geq 0}$ are two solutions to \eqref{eq: kinetic Langevin with quadratic potential} with initial values such that $\Delta z\neq 0$. Suppose that there is a $k\in\N$ such that $\mu_h(Q^h_{k}\neq \tilde Q^h_{k})>0$. Then there exist constants $k_{\mu_h}\in\N$, $c_{\mu_h}>0$ such that for all $k\geq k_{\mu_h}$ we have 
\begin{equation}
\label{eq: coupling bound discrete}
    \mu_{h}(Z^h_k\neq \tilde Z^h_k)\geq \frac{c_{\mu_{h}}}{\sqrt{hk+1}}.
\end{equation}
Furthermore, if in addition $\mu$ satisfies the now-equals forever property, then we can pick $k_{\mu_h}=0$.
\end{lemma}
\begin{proof}
Let $k_0=\min\{k\in\N:\mu_h(Q^h_k\neq\tilde Q^h_k)>0\}$. Since the coupling $\mu_h$ is Markovian, the shifted chain $(Q^h_{k_0+k},\tilde Q^h_{k_0+k})_{k\in\N}$ conditioned on $\F_{k_0}$ constitutes a coupling between two processes following the (discretized) dynamics of~\eqref{eq:dQ_SDE} and having different initial values with positive probability under $\mu_h$. Since $Z^h_k=\tilde Z^h_k$ implies that $Q^h_k=\tilde Q^h_k$, we have for every $k>k_0$ that
\begin{equation}
\label{eq: coupling probability simple bound (discrete Markovian proof)}
    \mu_h(Z^h_k\neq\tilde Z^h_k)\geq\mu_h(Q^h_k\neq \tilde Q^h_k)=\E_{\mu_h}\sqbr{\E^{\F_{k_0}}_{\mu_h}\sqbr{\ind_{\{Q^h_k\neq\tilde Q^h_k\}}}}.
\end{equation}
The TV distance of the two (discretized) processes $Q^h_k = Q_{hk}$ and $\tilde Q^h_k = \tilde{Q}_{hk}$ is given by \eqref{eq: TV distance Q_t (alpha=0)}; for $k>k_0$ we have
\begin{equation}
\label{eq: dTV Q^h_k0 (discrete Markovian proof)}
    \E^{\F_{k_0}}_{\mu_h}\sqbr{\ind_{\{Q^h_k\neq\tilde Q^h_k\}}}\geq2\Phi\br{\frac{|\tilde Q^h_{k_0}-Q^h_{k_0}|}{2\sqrt{2\gamma^{-1}h(k-k_0)}}}-1.
\end{equation}
If $y\in\R^d$ satisfies $|y|\leq 2\sqrt{2\gamma^{-1}h(k-k_0)}$, then we can bound
\begin{equation}
    \label{eq: Q^h_n TV lower bound (discrete Markovian proof)}
    2\Phi\br{\frac{|y|}{2\sqrt{2\gamma^{-1}h(k-k_0)}}}-1\geq \frac{|y|}{2 \sqrt{\pi e\gamma^{-1}h(k-k_0)}}.
\end{equation}
Note that 
\begin{equation*} 
\E_{\mu_h}\sqbr{\ind_{\l\{|\tilde Q^h_{k_0}-Q^h_{k_0}|\leq 2\sqrt{2\gamma^{-1}h(k-k_0)}\r\}}|\tilde Q^h_{k_0}-Q^h_{k_0}|} \nearrow \E_{\mu_h}\sqbr{|\tilde Q^h_{k_0}-Q^h_{k_0}|} \quad \textnormal{as } k\to\infty.
\end{equation*}
Thus, there exists a $k_{\mu_h}>k_0$ such that for all $k\geq k_{\mu_h}$ we have that
\begin{equation}
\label{eq: Q^h_n difference restricted lower bound (discrete Markovian proof)}
    \E_{\mu_h}\sqbr{\ind_{\l\{|\tilde Q^h_{k_0}-Q^h_{k_0}|\leq 2\sqrt{2\gamma^{-1}h(k-k_0)}\r\}}|\tilde Q^h_{k_0}-Q^h_{k_0}|}\geq \frac{1}{2}\E_{\mu_h}\sqbr{|\tilde Q^h_{k_0}-Q^h_{k_0}|}=2\sqrt{\pi e \gamma^{-1}}c_{\mu_h},
\end{equation}
where $c_{\mu_h}\coloneqq \E_{\mu_h}\sqbr{|\tilde Q^h_{k_0}-Q^h_{k_0}|}/\br{4\sqrt{\pi e \gamma^{-1}}}$. Note that $c_{\mu_h}>0$ by definition of $k_0$.

Combining the above observations we see that for $k\geq k_{\mu_{h}}$ we have the lower bound
\begin{align*}
    \mu_h(Z^h_k\neq\tilde Z^h_k)&\stackrel{\eqref{eq: coupling probability simple bound (discrete Markovian proof)}}{\geq}\E_{\mu_h}\l[\E_{\mu_h}^{\F_{k_0}}[\ind_{\{Q^h_k\neq\tilde Q^h_k\}}]\r]\\
    &\stackrel{\eqref{eq: dTV Q^h_k0 (discrete Markovian proof)}}{\geq}\E_{\mu_h}\l[2\Phi\br{\frac{|\tilde Q^h_{k_0}-Q^h_{k_0}|}{2\sqrt{2\gamma^{-1}h(k-k_0)}}}-1\r]\\
    &\geq\E_{\mu_h}\l[\ind_{\l\{|\tilde Q^h_{k_0}-Q^h_{k_0}|\leq 2\sqrt{2\gamma^{-1}h(k-k_0)}\r\}}\l\{2\Phi\br{\frac{|\tilde Q^h_{k_0}-Q^h_{k_0}|}{2\sqrt{2\gamma^{-1}h(k-k_0)}}}-1\r\}\r]\\
    &\stackrel{\eqref{eq: Q^h_n TV lower bound (discrete Markovian proof)}}{\geq} \E_{\mu_h}\sqbr{\ind_{\l\{|\tilde Q^h_{k_0}-Q^h_{k_0}|\leq 2\sqrt{2\gamma^{-1}h(k-k_0)}\r\}}\frac{|\tilde Q^h_{k_0}-Q^h_{k_0}|}{2\sqrt{\pi e\gamma^{-1}h(k-k_0)}}}\\
    &\stackrel{\eqref{eq: Q^h_n difference restricted lower bound (discrete Markovian proof)}}{\geq}\frac{c_{\mu_h}}{\sqrt{h(k-k_0)}}\geq\frac{c_{\mu_h}}{\sqrt{hk+1}}.
\end{align*}
This completes the proof of \eqref{eq: coupling bound discrete}. If in addition $\mu_h$ satisfies the now-equals forever property, then we have for all $0\leq k<k_{\mu_h}$ that $\mu(Z^h_k\neq\tilde Z^h_k)\geq\mu(Z_{k_{\mu_h}}\neq \tilde Z_{k_{\mu_h}})$. The bound \eqref{eq: coupling bound discrete} can therefore be extended to all $k\in\N$ by appropriately scaling $c_{\mu_h}$.
\end{proof}

\medskip
\begin{lemma}\label{lem: hitting prob when in eigenspace on grid}
Assume that $\gamma^2> \alpha =0$, let $h>0$, and let $\mu_h$ be a Markovian coupling of $(Z_k^{h},\tilde{Z}_k^h)_{k\in \N} = (Z_{hk},\tilde Z_{hk})_{k\in \N}$ where $(Z_t,\tilde Z_t)_{t\geq 0}$ are two solutions to \eqref{eq: kinetic Langevin with quadratic potential} with initial values such that $\Delta z\neq 0$. Suppose that $\mu_h(\cap_{k=0}^\infty\{Q^h_{k}=\tilde Q^h_k\})=1$. Then there exists a $c>0$ (independent of $h$ and $\mu_h$) such that for all $k\in \N$ one has 
    \begin{equation*}
        \mu_h(Z_k^h\neq \tilde Z_k^h)\geq c \min(1, h^{-1} e^{-\gamma hk} ).
    \end{equation*}
\end{lemma}
\begin{proof}
Note that by construction (see~\eqref{eq: kinetic Langevin with quadratic potential (as matrix)} and~\eqref{eq: free kinetic Langevin mean and covariance matrices}) we have that $(Z_k^h)_{k\in \N}=(X_k^h,V_k^h)_{k\in \N} = (X_{hk},V_{hk})_{k\in \N} $ satisfies
\begin{align}
X^h_{k+1} &=X^h_k+\gamma^{-1}(1-e^{-\gamma h}) V^h_k+\begin{pmatrix}\id_d & 0\end{pmatrix}B_h
    \xi_{k+1},\\
V^h_{k+1} &=e^{-\gamma h} V^h_k+\begin{pmatrix} 0 & \id_d \end{pmatrix}B_h
    \xi_{k+1}
\end{align}
for all $k\in \N$, where $(\xi_k)_{k\in \N_{>0}}$ is a sequence of of i.i.d.\ $\mathcal{N}(0,\id_{2d})$-distributed random variables and ${B_h\in \R^{2d\times 2d}}$ is such that $B_h B_h^T = \Sigma_h$ with $\Sigma_h$ given by~\eqref{eq: free kinetic Langevin mean and covariance matrices}. Analogous relations hold for $(\tilde Z_k^h)_{k\in \N}=(\tilde X_k^h,\tilde V_k^h)_{k\in \N}$; these involve a sequence of i.i.d.\ $\mathcal{N}(0,\id_{2d})$-distributed random variables $(\tilde \xi_k)_{k\in \N_{>0}}$. Setting $\Delta \xi_k = \tilde \xi_k - \xi_k$, $k\in \N_{>0}$, we obtain the following relation for $\Delta X^h_{k+1}=X^h_{k+1}-\tilde X^h_{k+1}$, where we use that $\mu_h( \Delta Q_k^h = 0 )=1$:
\begin{align*}
    \Delta X^h_{k+1}&=\Delta X^h_k+\gamma^{-1}(1-e^{-\gamma h})\Delta V^h_k+\begin{pmatrix}\id_d & 0\end{pmatrix}B_h
    \Delta \xi_{k+1}\\
    &=e^{-\gamma h}\Delta X^h_k+(1-e^{-\gamma h}) \Delta Q^h_k+\begin{pmatrix}\id_d & 0\end{pmatrix}B_h
    \Delta \xi_{k+1}\\
    &=e^{-\gamma h}\Delta X^h_k+\begin{pmatrix}\id_d & 0\end{pmatrix}B_h\Delta \xi_{k+1},
\end{align*}
for all $k\in\N$ $\mu_h$-almost surely.\par 
Consider the Markov chains $(Y^h_k)_{k\in\N}$, $(\tilde Y^h_k)_{k\in\N}$ given recursively by $Y^h_0=x$, $\tilde Y^h_0=\tilde x$ and 
\begin{equation*}
    Y^h_{k+1}=e^{-\gamma h}Y^h_k+\begin{pmatrix}\id_d & 0 \end{pmatrix}B_h\xi_{k+1},\qquad\tilde Y^h_{k+1}=e^{-\gamma h}\tilde Y^h_k+\begin{pmatrix}\id_d & 0 \end{pmatrix}B_h\tilde \xi_{k+1}. 
\end{equation*}
Note that $\mu_h(\cap_{k=0}^\infty\{\Delta X^h_k=\Delta Y^h_k\})=1$. In addition, the two chains are distributed as $Y^h_n\sim\calN \br{e^{-\gamma hn}x,\Sigma_{h,n}^{xx}}$ and $\tilde Y^h_n\sim\calN \br{e^{-\gamma hn}\tilde x,\Sigma_{h,n}^{xx}}$, with covariance
\begin{equation*}
    \Sigma_{h,n}^{xx}=\sum_{k=0}^{n-1} e^{-2\gamma hk}\begin{pmatrix}\id_d & 0 \end{pmatrix}
    B_hB_h^T\begin{pmatrix}\id_d \\ 0 \end{pmatrix}=\frac{1-e^{-2\gamma hn}}{1-e^{-2\gamma h}}\sigma_{h}^{xx}\id_d,
\end{equation*}
where $B_hB_h^T=\Sigma_h$ is given by \eqref{eq: free kinetic Langevin mean and covariance matrices} and
\begin{equation*}
    \sigma_{h}^{xx}=2\gamma^{-1}h-4\gamma^{-2}\br{1-e^{-\gamma h}}+\gamma^{-2}(1-e^{-2\gamma h}).
\end{equation*}
There exists a constant $C>0$ (independent of $h$ and $\gamma$) such that\footnote{The function $f\colon\R\to\R$ given by $f(x)=2x-4(1-e^{-x})+1-e^{-2x}$ has derivative $f^\prime(x)=2(1-e^{-x})^2$, so that $f^\prime(x)\leq2\min(1,x^2)$ for all 
$x\geq0$. Consequently, $f(x)\leq2\min(x,\frac{1}{3} x^3)$ for all $x\geq0$, and therefore $\sigma^{xx}_h=\gamma^{-2}f(\gamma h)\leq 2\gamma^{-2}\min(\gamma h,\frac{1}{3}(\gamma h)^3)$.} $\sigma_h^{xx}\leq C\min(h,h^3)$. Additionally, there exists a constant $c>0$ such that $1-e^{-2\gamma h}\geq c\min(1,h)$. In conclusion, there exists a constant (possibly different) $C>0$ such that
\begin{equation*}
    \frac{\sigma^{xx}_h}{1-e^{-2\gamma h}}\leq C\min(h,h^2)\leq Ch^2.
\end{equation*}
Putting the above observations together we see for all $n\in\N$ that (after passing to a different $C>0$ if necessary) one has
\begin{align*}
    \mu_h(Z^h_n\neq \tilde Z^h_n)& \geq\mu_h(X^h_n\neq \tilde X^h_n)=\mu_h(Y^h_n\neq \tilde Y^h_n)\geq\dTV{\law(Y^h_n),\law(\tilde Y^h_n)}\\
    & =2\Phi\br{\frac{e^{-\gamma hn}|(\Sigma_{h,n}^{xx})^{-1/2}\Delta x|}{2}}-1=2\Phi\br{\frac{e^{-\gamma hn}|\Delta x|}{2\sqrt{\sigma^{xx}_h}\sqrt{\frac{1-e^{-2\gamma h n}}{1-e^{-2\gamma h}}}}}-1\\
    & \geq 2\Phi\br{\frac{e^{-\gamma hn}|\Delta x|}{Ch}}-1.
\end{align*}
In view of the fact that $2\Phi(x) -1 \geq \frac{\sqrt{2}}{\sqrt{\pi e}} \min(1,x)$, the proof is now completed. 
\end{proof}

\medskip
\begin{remark}
    The Markovian couplings in this section are only allowed to depend on the step size $h$. Theorem~\ref{thm: discrete Markovian couplings have the wrong asymptotic behavior} therefore does not rule out the existence of a set of Markovian couplings $(\mu_{h,n})_{h>0,n\in\N}$ depending on both the step size $h$ and the terminal time $T=hn$ that both reproduce the right asymptotic behavior in $T$ and is stable under $h\downarrow 0$.
\end{remark}\medskip

\section{Quantifying the lack of optimality for iterated one-shot couplings}
\label{sec: iterated one-shot coupling}

The iterated one-shot coupling is a coupling for Markov chains; it is defined such that the probability of two chains meeting at time $n\in \N$ is maximized given the state at time $n-1$. In particular, this is Markovian coupling satisfying the now-equals-forever property in the sense of Definition~\ref{def: discrete Markovian coupling}, and it is the natural discretization of the well-known reflection coupling for continuous diffusion equations. The coupling can also be thought of as a greedy algorithm for maximizing the meeting probability iteratively. It has been used effectively in e.g.~\cite{DurmusMoulines:2019} to bound the TV distance of Euler discretizations of the overdamped Langevin equation, as well as in other places, such as \cite{JacobEtAl:2020}.

While~\cite{DurmusMoulines:2019} demonstrates that the iterated one-shot coupling provides asymptotically optimal meeting probabilities for the \emph{overdamped} Langevin equation, Theorem~\ref{thm: discrete Markovian couplings have the wrong asymptotic behavior} above implies that this cannot be the case for the \emph{kinetic} Langevin equation. Indeed, Theorem~\ref{thm: iterated one-shot coupling probability} below provides an explicit (exact) expression for the meeting probability under the iterated one-shot coupling for linear Markov chains with Gaussian increments, i.e.\, for the type of Markov chain arising when discretizing a linear Langevin equation. This explicit expression for the meeting probability reveals that at least one of terms obtained as a lower bound in Theorem~\ref{thm: discrete Markovian couplings have the wrong asymptotic behavior} is sharp, see Example~\ref{example: one shot matching lower bound} below. \cref{thm: iterated one-shot coupling probability} and the subsequent remarks are also of independent interest: they provide guidelines on when the iterated one-shot coupling may provide asymptotically optimal bounds for the TV distance for more general diffusion equations. Both the notation and most of the arguments in this section are inspired by results in Section 6 of \cite{DurmusMoulines:2019}.

\subsection{Iterated one-shot coupling for linear Markov chains}

As mentioned above, Theorem~\ref{thm: iterated one-shot coupling probability} below concerns establishing an expression for the meeting probability for an iterated one-shot coupling for Markov chains with Gaussian increments. More specifically, the covariance of the increment is assumed to be independent of the state (though it may be time-dependent). The goal of this section is to provide the construction of an iterated one-shot coupling in this setting. In order to do so, one needs a maximal (i.e.\ optimal) coupling for Gaussian distributions on $\R^d$ with the same non-singular covariance matrix $\Sigma$ but with (possibly) different means $\mu$ and $\tilde\mu$. 
Multiple such maximal couplings exist. In this section we use the \emph{reflection coupling}, which we outline below; see e.g.\ \cite{Bou-RabeeEberle:2021,BubleyEtAl:1998,DurmusMoulines:2019} for a more extensive treatment.

\subsubsection{Reflection coupling for Gaussians with shared covariance}\label{sssec: reflection coupling}

Let $\mu,\tilde{\mu}\in \R^d$ and let $\Sigma \in \R^{d\times d}$. As announced above, we provide the construction of a maximal coupling for $\mathcal{N}(\mu,\Sigma^{\nicefrac{1}{2}})$ and $\mathcal{N}(\tilde{\mu},\Sigma^{\nicefrac{1}{2}})$ known as the reflection coupling.

Let $B\in\R^{d\times d}$ by a matrix such that $BB^T=\Sigma$ and define
\begin{align*}
    E=B^{-1}(\tilde\mu-\mu), \qquad e=
    \begin{cases}
        \frac{E}{|E|} &\textup{if }  E\neq0,\\
        0 &\textup{if }  E=0.
    \end{cases}
\end{align*}
and
\begin{equation*}
    p(z)=\frac{\varphi(E-z)}{\varphi(z)}=\exp\br{-|E|^2/2+\langle E,z\rangle}, \qquad z\in\R^d,
\end{equation*}
where $\varphi\colon\R\to\R$ is the density function of a one-dimensional standard Gaussian. Let $\xi\sim\mathcal{N}(0,\id_d)$ and $U\sim U(0,1)$ be independent from each other and set $X=\mu+B\xi$. Furthermore, consider the random variable
\begin{equation}
\label{eq: one-shot coupling}
    \tilde\xi=\ind_{\{U\leq p(\xi)\}}(\xi-E)+\ind_{\{U>p(\xi)\}}(\id_d-2ee^T)\xi
\end{equation}
and set $\tilde X=\tilde\mu+B\tilde\xi$. Note that $(\id_d-2ee^T)$ describes a reflection in the linear subspace orthogonal to $E=B^{-1}(\tilde\mu-\mu)$, so that $X=\tilde X$ on $\{U\leq p(\xi)\}$ and $B^{-1}(\tilde X-\tilde\mu)$ is the reflection of $B^{-1}(X-\mu)$ in this linear subspace on $\{U>p(\xi)\}$.

The random variable $\tilde\xi$ is distributed as $\mathcal{N}(0,\id_d)$ \cite[Section 2.3.2]{Bou-RabeeEtAl:2020}. Consequently, the random variable $(X,\tilde X)$ has marginals $\mathcal{N}(\mu,\Sigma)$ and $\mathcal{N}(\tilde\mu,\Sigma)$ and is therefore a realization of a coupling between these distributions. As mentioned, this coupling is optimal in the sense that 
\begin{equation}
\label{eq: one-shot TV distance}
    \P(X\neq \tilde X)=\dTV{\mathcal{N}(\mu,\Sigma),\mathcal{N}(\tilde\mu,\Sigma)}=2\Phi\br{\frac{\l|B^{-1}(\tilde\mu-\mu)\r|}{2}}-1,
\end{equation}
where $\Phi\colon\R\to\R$ is the cumulative distribution function of the one-dimensional standard Gaussian distribution, see \cite[Section 2.3.2]{Bou-RabeeEtAl:2020} (and see Equations~\eqref{eq: TV coupling characterization} and~\eqref{eq: Gaussian TV distance}). We note that if $\mu=\tilde\mu$, then $E=0$, $p\equiv1$, and thus $\tilde\xi=\xi$ and $\tilde X=X$.

\subsubsection{An iterated one-shot coupling}\label{sssec: iterated one-shot}
We now construct an iterated one-shot coupling for discrete time Markov chains with Gaussian increments based on the reflection coupling introduced above. First we introduce the Markov chains: for $k\in\N_{>0}$ let $B_k\in\R^{d\times d}$ be a non-singular matrix and let $h_k\colon\R^d\to\R^d$ some function that models the drift at step $k$. Consider the two Markov chains $(Z_k)_{k\in\N}$, $(\tilde Z_k)_{k\in\N}$ with initial values $Z_0=z$, $\tilde Z_0=\tilde z$ and, for $k\in\N$,
\begin{equation}
\label{eq: autoregressive Markov chains}
    Z_{k+1}=h_{k+1}(Z_k)+B_{k+1}\xi_{k+1}, \quad \tilde Z_{k+1}=h_{k+1}(\tilde Z_k)+B_{k+1}\tilde\xi_{k+1},
\end{equation}
where $(\xi_k)_{k\in \N_{>0}}$ and $(\tilde{\xi}_k)_{k\in \N_{>0}}$ are two sequences of i.i.d.\ $\mathcal{N}(0,\id_d)$-distributed random variables on a probability space $(\Omega,\mathcal{F},\P)$.\par 
The iterated one-shot coupling for $(Z_k,\tilde{Z}_k)_{k\in \N}$ is constructed by repeated application of the reflection coupling introduced above. More specifically, we assume the i.i.d.\ sequence of $\mathcal{N}(0,\id_d)$-distributed random variables $(\xi_k)_{k\in \N_{>0}}$ to be given, and construct the i.i.d.\ sequence of $\mathcal{N}(0,\id_d)$-distributed random variables $(\tilde{\xi}_k)_{k\in \N_{>0}}$ from $(\xi_k)_{k\in \N_{>0}}$. To this end, we let $(U_k)_{k\in \N_{>0}}$ be an i.i.d.\ $U(0,1)$-distributed sequence of random variables on $(\Omega,\mathcal{F},\P)$ that is independent of $(\xi_k)_{k\in \N_{>0}}$.
Setting $Z_0=z_0$ and $\tilde{Z}_0=\tilde{z}_0$, we inductively define, for $k\in\N_{>0}$,
\begin{align*}
    E_k=B_k^{-1}\br{h_k(\tilde Z_{k-1})-h_k(Z_{k-1})}, \qquad e_k=
    \begin{cases}
        \frac{E_k}{|E_k|} &\textup{if }  E_k\neq0,\\
        0 &\textup{if }  E_k=0,
    \end{cases}
\end{align*}
and 
\begin{equation}
\label{eq: multi-shot acceptance probability}
    p_k(z)=\frac{\varphi(E_k-z)}{\varphi(z)}=\exp\br{-|E_k|^2/2+\langle E_k,z\rangle}, \qquad z\in\R^d.
\end{equation}
In addition, we let $\tilde\xi_k$ be the random variable given by
\begin{equation}
\label{eq: multi-shot coupling}
    \tilde\xi_{k}=\ind_{\{U_k\leq p_k(\xi_k)\}}(\xi_{k}-E_k)+\ind_{\{U_k>p_k(\xi_k)\}}(\id_d-2e_ke_k^T)\xi_{k}.
\end{equation}
Finally, $Z_{k}$ and $\tilde{Z}_k$ are defined by~\eqref{eq: autoregressive Markov chains}. The resulting sequence $(Z_k,\tilde{Z}_k)_{k\in \N}$ is a Markov chain with respect to the filtration $(\mathcal{F}_k)_{k\in \N}$ generated by $(Z_k,\tilde{Z}_k)_{k\in \N}$; it is known as the iterated one-shot coupling for~\eqref{eq: autoregressive Markov chains} (under reflection coupling). \par \medskip 
Note that once the chains meet under the iterated one-shot coupling, they will stick together: on the set $Z_k=\tilde{Z}_k$, one has $E_j=0$ and $\xi_j = \tilde{\xi}_j$ for all $j\geq k$. Therefor, the iterated one-shot coupling satisfies the now-equals forever property. Moreover, 
\begin{equation}
\label{eq: conditioned one-step rejection rate}
    \P(Z_k\neq \tilde Z_k)=\E\l[\E^{\F_{k-1}}\l[\ind_{\{Z_k\neq\tilde Z_k\}}\r]\r]=\E\l[2\Phi\br{\frac{|E_k|}{2}}-1\r],
\end{equation}
by combination of the fact that the chain is Markovian and \eqref{eq: one-shot TV distance}.

\subsection{The meeting probability for the iterated one-shot coupling}
Theorem~\ref{thm: iterated one-shot coupling probability} below provides an exact expression for the meeting probability of the one-shot coupling introduced above in the case that the Markov chains are linear. More specifically, we assume that the functions $h_k$ in \eqref{eq: autoregressive Markov chains} are given by $h_k(z)=A_kz$ for some non-singular $A_k\in\R^{d\times d}$. The dynamics of the Markov chains $(Z_k)_{k\in \N}$ and $(\tilde{Z}_k)_{k\in \N}$ are thus given by
\begin{equation}
\label{eq: linear Markov chain}
    Z_{k+1}=A_{k+1}Z_k+B_{k+1}\xi_{k+1}, \quad \tilde Z_{k+1}=A_{k+1}\tilde Z_k+B_{k+1}\tilde\xi_{k+1}.
\end{equation}
\medskip
\begin{theorem}
\label{thm: iterated one-shot coupling probability}
    Let the Markov chain $(Z_k,\tilde Z_k)_{k\in \N}$ be given by the iterated one-shot coupling introduced in Section~\ref{sssec: iterated one-shot} applied to the Markov chains \eqref{eq: linear Markov chain}. The probability that the two chains have not met after $n$ steps is given by
    \begin{equation}
    \label{eq: multi-shot exact bound general}
        \P(Z_n\neq \tilde Z_n)=2\Phi\br{\frac{1}{2\Theta_{n}^{1/2}}}-1,
    \end{equation}
    where 
    \begin{equation*}
        \Theta_{n}=\sum_{k=1}^n\frac{1}{\l|B_{k}^{-1}\Pi_k\Delta z\r|^{2}},
    \end{equation*}
    and $\Pi_k\coloneqq A_kA_{k-1}\dots A_1$.
\end{theorem}

This theorem is a variation on Theorem 19 in \cite{DurmusMoulines:2019}, which provides an upper bound (instead of an equality) for the probability in~\eqref{eq: multi-shot exact bound general} in a setting equivalent to Markov chains of the form \eqref{eq: autoregressive Markov chains} under a Lipschitz assumption on the drift functions $h_k$. In particular, the proof of Theorem~\ref{thm: iterated one-shot coupling probability}, which is presented in Section~\ref{sec: proof of iterated one-shot coupling probability} below, follows the lines of the the proof of~\cite[Theorem 19]{DurmusMoulines:2019}. However, extra work is needed to obtain an equality instead of an upper bound, this is mainly contained in the preparatory lemmas presented in the following section (Lemmas \ref{lemma: rejection equivalences} and \ref{lemma: e_k expression} below). We stress that linearity of the drift is crucial for obtaining an equality instead of an inequality.  

\medskip
\begin{remark}
    Since both $Z_n$ and $\tilde Z_n$ given by \eqref{eq: linear Markov chain} are Gaussian random variables with the same covariance matrix, we can calculate their TV distance directly using \eqref{eq: one-shot TV distance}. \cref{thm: iterated one-shot coupling probability} should therefore not be interpreted as a result that helps to bound the TV distance. Rather, it serves to establish sharp lower bounds for the probability of not meeting, and it helps to determine whether the iterated one-shot coupling can produce useful TV distance bounds when linear Markov chains like \eqref{eq: linear Markov chain} are perturbed with some nonlinear term.
\end{remark}

\subsubsection{Some preparatory lemmas}
The following two lemmas provide some insight in the iterated one-shot coupling defined in Section~\ref{sssec: iterated one-shot} applied to the Markov chains~\eqref{eq: linear Markov chain}. For the sake of brevity, we use the notation $\Delta Z_k\coloneqq\tilde Z_k-Z_k$ and $\Delta z\coloneqq\tilde z-z$. Since $\P(Z_n\neq \tilde Z_n)=0$ if $\Delta z=0$, we assume without loss of generality that $\Delta z\neq 0$. 

The assumption that each $A_k$ is non-singular ensures that whether the two chains have met in the $k^{\text{th}}$ step is determined solely by the combination of whether they have met before and whether $U_k\leq p_k(\xi_k)$, as shown in the following lemma.

\medskip
\begin{lemma}
\label{lemma: rejection equivalences}
     Let $(Z_k,\tilde Z_k)_{k\in \N}$ be given by the iterated one-shot coupling defined in Section~\ref{sssec: iterated one-shot} applied to the Markov chains~\eqref{eq: linear Markov chain}. Then for all $k\in \N$ we have
     \begin{align*}
         \Delta Z_k\neq0&\iff E_k\neq0 \textup{ and }U_k>p_k(\xi_k),\\
         E_k\neq0&\iff\Delta Z_{k-1}\neq0.
     \end{align*}
\end{lemma}
\begin{proof}
    By definition of the iterated one-shot coupling, either $E_k=0$ or $U_k\leq p_k(\xi_k)$ implies that $\Delta Z_k=0$, so that $\Delta Z_k\neq0$ implies both $E_k\neq0$ and $U_k>p_k(\xi_k)$. On the other hand, if both $E_k\neq0$ and $U_k>p_k(\xi_k)$ we have by \eqref{eq: multi-shot coupling} that
    \begin{equation*}
        |B_k^{-1}\Delta Z_k|=\l|B_{k}^{-1}\br{A_k\Delta Z_{k-1}+B_k(\tilde\xi_k-\xi_k)}\r|=\l|E_k-2\langle e_k,\xi_k\rangle e_k\r|=\l||E_k|-2\langle e_k,\xi_k\rangle\r|.
    \end{equation*}
    Furthermore, note that $U_k>p_k(\xi_k)$ ensures that $p_k(\xi_k)<1$, which by definition of $p_k$ implies that 
    \begin{equation}
    \label{eq: no crossing of reflection plane}
        |E_{k}|-2\langle e_{k},\xi_{k}\rangle>0.
    \end{equation}
    Combination of these two observations shows that if both $E_k\neq0$ and $U_k>p_k(\xi_k)$ we also have that $\Delta Z_k\neq 0$, proving the first equivalence. The second equivalence follows immediately from the fact that $E_k=B_{k}^{-1}A_k\Delta Z_{k-1}$ and that $A_k$ and $B_k$ are non-singular. 
\end{proof}
If the two chains have not met before the $k^{\text{th}}$ step, the vector $e_k$ determines the linear subspace in which $\xi_k$ is reflected to obtain $\tilde\xi_k$ in the case that $U_k>p_k(\xi_k)$. The following lemma provides a deterministic expression for $e_k$.

\medskip
\begin{lemma}
\label{lemma: e_k expression}
     Let $(Z_k,\tilde Z_k)_{k\in \N}$ be given by the iterated one-shot coupling defined in Section~\ref{sssec: iterated one-shot} applied to the Markov chains~\eqref{eq: linear Markov chain}. For $k\in \N$ consider the event $\{Z_{k-1}\neq\tilde Z_{k-1}\}$, i.e.\ the event that the chains have not met before the $k^{\text{th}}$ step. On this event we have that
     \begin{equation*}
         e_{k}=\frac{B_{k}^{-1}
         \Pi_{k}\Delta z}{\l|B_{k}^{-1}\Pi_{k}\Delta z\r|},
     \end{equation*}
    where $\Pi_{k}\coloneqq A_{k}A_{k-1}\dots A_1$.
\end{lemma}
\begin{proof}
    We will prove this by induction. The statement for $k=0$ follows directly from the fact that $E_1=B_1^{-1}A_1\Delta z$. So let $k\geq 1$ and assume the statement is true for $k-1$. On the event $\{Z_{k}\neq \tilde Z_k\}$ we have that $\tilde \xi_{k}=\xi_{k}-2\langle e_{k},\xi_{k}\rangle e_{k}$, so that
    \begin{align*}
        \Delta Z_{k}&=A_k\Delta Z_{k-1}-2\langle e_{k},\xi_{k}\rangle B_{k}e_{k}= B_kE_k-2\langle e_{k},\xi_{k}\rangle B_ke_{k}\\
        &=\br{|E_k|-2\langle e_{k},\xi_{k}\rangle}B_ke_{k}.
    \end{align*}
    This means that on $\{Z_k\neq\tilde Z_k\}$ we have
    \begin{equation}
    \label{eq: E_k expression}
        E_{k+1}=B_{k+1}^{-1}A_{k+1}\Delta Z_{k}=\br{|E_k|-2\langle e_{k},\xi_{k}\rangle}B_{k+1}^{-1}A_{k+1}B_ke_{k}.
    \end{equation}
    On $\{Z_{k}\neq \tilde Z_k\}$ we also have $|E_{k}|-2\langle e_{k},\xi_{k}\rangle>0$, as stated in \eqref{eq: no crossing of reflection plane}. Combined with the induction hypothesis, which can be applied since $Z_{k}\neq \tilde Z_k$ implies that $Z_{k-1}\neq \tilde Z_{k-1}$, we see that $E_{k+1}$ is a positive multiple of $B_{k+1}^{-1}A_{k+1}\Pi_k\Delta z=B_{k+1}^{-1}\Pi_{k+1}\Delta z$ on $\{Z_{k}\neq \tilde Z_k\}$. The statement now follows from the fact that $e_{k+1}=E_{k+1}/|E_{k+1}|$ is normalized.
\end{proof}

We will also need the following standard integral; see \cite[Lemma 20]{DurmusMoulines:2019} for proof.
\medskip
\begin{lemma}
\label{lemma: standard integral Durmus}
For all $a>0$ and $t\geq0$ we have
    \begin{equation*}
        \int_{-\infty}^\infty\varphi(y)\l\{1-\left(1\wedge\frac{\varphi(t-y)}{\varphi(y)}\right)\r\}\l\{2\Phi\br{\frac{|t-2y|}{2a}}-1\r\}\,dy=2\Phi\br{\frac{t}{2\br{1+a^2}^{1/2}}}-1,
    \end{equation*}
    where $\varphi$ denotes the density of the $1$-dimensional standard Gaussian $\mathcal{N}(0,1)$.
\end{lemma}

\subsubsection{Proof of Theorem~\ref{thm: iterated one-shot coupling probability}}\label{sec: proof of iterated one-shot coupling probability}
    Fix $n\in \N$. We will show by backward induction that 
    \begin{equation}
    \label{eq: induction hypthesis}
        \P(Z_n\neq \tilde Z_n)=\E\l[2\Phi\br{\frac{|E_k|}{2\Xi_{k,n}^{1/2}}}-1\r],
    \end{equation}
    for all $k\in\{1,\dots,n\}$, where
    \begin{equation*}
        \Xi_{k,n}=\sum_{j=k}^n\frac{|B_{k}^{-1}\Pi_k\Delta z|^2}{|B_j^{-1}\Pi_j\Delta z|^2}.
    \end{equation*}
    The proof is then completed by considering the case $k=1$ and noting that $\Theta_{n}=\Xi_{1,n}/|E_1|^2$.

    First, the base case $k=n$ is covered by \eqref{eq: conditioned one-step rejection rate}. So let $k\in\{1,\dots,n-1\}$ and suppose that \eqref{eq: induction hypthesis} is true for $k+1$. Since $\Delta Z_k=0$ implies that $E_{k+1}=0$, which in turn implies that $2\Phi(|E_{k+1}|/2)-1=0$, we can write
    \begin{equation}
    \label{eq: induction hypothesis multi-meeting proof}
        \P(Z_n\neq\tilde Z_n)=\E\l[2\Phi\br{\frac{|E_{k+1}|}{2\Xi_{k+1,n}^{1/2}}}-1\r]=\E\l[\ind_{\{Z_{k}\neq\tilde Z_{k}\}}\l\{2\Phi\br{\frac{|E_{k+1}|}{2\Xi_{{k+1},n}^{1/2}}}-1\r\}\r].
    \end{equation}
    Thanks to the expression for $E_{k+1}$ on $\{Z_k\neq\tilde Z_k\}$ given by \eqref{eq: E_k expression} and the results of \cref{lemma: rejection equivalences} and \cref{lemma: e_k expression}, we have
    \begin{align*}
        \ind_{\{Z_{k}\neq\tilde Z_{k}\}}|E_{k+1}|&\stackrel{\phantom{xx}\eqref{eq: E_k expression}\phantom{xx}}{=}\ind_{\{Z_{k}\neq\tilde Z_{k}\}}\l||E_{k}|-2\langle e_{k},\xi_{k}\rangle\r|\l|B_{k+1}^{-1}A_{k+1}B_{k}e_{k}\r|\\
        &\stackrel{\text{Lemma } \ref{lemma: rejection equivalences}}{=}\ind_{\{Z_{k-1}\neq\tilde Z_{k-1}\}}\ind_{\{U_k>p_k(\xi_k)\}}\l||E_{k}|-2\langle e_{k},\xi_{k}\rangle\r|\l|B_{k+1}^{-1}A_{k+1}B_{k}e_{k}\r|\\
        &\stackrel{\text{Lemma } \ref{lemma: e_k expression}}{=}\ind_{\{Z_{k-1}\neq\tilde Z_{k-1}\}}\ind_{\{U_{k}>p_k(\xi_k)\}}\l||E_{k}|-2\langle e_{k},\xi_{k}\rangle\r|\frac{\l|B_{k+1}^{-1}\Pi_{k+1}\Delta z\r|}{\l|B_{k}^{-1}
         \Pi_{k}\Delta z\r|}.
    \end{align*}
    The term inside of the expectation on the right hand side of \eqref{eq: induction hypothesis multi-meeting proof} therefore has, conditioned on $\F_{k-1}$, the expectation
    \begin{multline*}
        \E^{\F_{k-1}}\l[\ind_{\{Z_{k}\neq\tilde Z_{k}\}}\l\{2\Phi\br{\frac{|E_{k+1}|}{2\Xi_{k+1,n}^{1/2}}}-1\r\}\r]\\
        =\ind_{\{Z_{k-1}\neq\tilde Z_{k-1}\}}\E^{\F_{k-1}}\l[\ind_{\{U_{k}>p_k(\xi_k)\}}\l\{2\Phi\br{\frac{\l||E_{k}|-2\langle e_k,\xi_k\rangle\r|}{2\frac{\l|B_k^{-1}\Pi_{k}\Delta z\r|}{\l|B_{k+1}^{-1}\Pi_{k+1}\Delta z\r|}\Xi_{k+1,n}^{1/2}}}-1\r\}\r].
    \end{multline*}
    The uniformly distributed random variable $U_k$ is independent of $\F_{k-1}$ and of $\xi_k$. Consequently, the expected value of $\ind_{\{U_{k}>p_k(\xi_k)\}}$ conditioned on both $\F_{k-1}$ and $\xi_k$ is equal to $1-(1\wedge p_k(\xi_k))=1-(1\wedge\exp\br{-|E_k|^2/2+|E_k|\langle e_k,\xi_k\rangle})$ and hence we have by the tower property that
    \begin{multline*}
        \E^{\F_{k-1}}\l[\ind_{\{Z_{k}\neq\tilde Z_{k}\}}\l\{2\Phi\br{\frac{|E_{k+1}|}{2\Xi_{k+1,n}^{1/2}}}-1\r\}\r]\\   
        =\ind_{\{Z_{k-1}\neq\tilde Z_{k-1}\}}\E^{\F_{k-1}}\l[\br{1-1\wedge \exp\br{-\frac{|E_k|^2}{2}+|E_k|\langle e_k,\xi_k\rangle}}\l\{2\Phi\br{\frac{\l||E_{k}|-2\langle e_k,\xi_k\rangle\r|}{2\frac{\l|B_k^{-1}\Pi_{k}\Delta z\r|}{\l|B_{k+1}^{-1}\Pi_{k+1}\Delta z\r|}\Xi_{k+1,n}^{1/2}}}-1\r\}\r],
    \end{multline*}
    Furthermore, if $e_k\neq0$, the random variable $\langle e_{k},\xi_{k}\rangle$ is a projection of $\xi_{k}\sim\mathcal{N}(0,\id_d)$ onto the $1$-dimensional linear subspace spanned by $e_k$. This means it has the conditional distribution $\law\br{\langle e_{k},\xi_{k}\rangle|\F_{k-1}}=\ind_{\{e_k\neq0\}}\mathcal{N}(0,1)+\ind_{\{e_k=0\}}\delta_0$, where the event $\{e_k\neq0\}$ is equivalent to $\{Z_{k-1}\neq\tilde Z_{k-1}\}$ by \cref{lemma: rejection equivalences}. We can therefore apply the integral of \cref{lemma: standard integral Durmus} with $a=\br{\l|B_k^{-1}\Pi_{k}\Delta z\r|/\l|B_{k+1}^{-1}\Pi_{k+1}\Delta z\r|}\Xi_{k+1,n}^{1/2}$ and $t=|E_k|$ to obtain 
    \begin{equation*}
        \E^{\F_{k-1}}\l[\ind_{\{Z_{k}\neq\tilde Z_{k}\}}\l\{2\Phi\br{\frac{|E_{k+1}|}{2\Xi_{k+1,n}^{1/2}}}-1\r\}\r]\\
        =\ind_{\{Z_{k-1}\neq\tilde Z_{k-1}\}}\l\{2\Phi\br{\frac{|E_{k}|}{2\Xi_{k,n}^{1/2}}}-1\r\},
    \end{equation*}
    where we use that
    \begin{equation*}
        1+\frac{\l|B_k^{-1}\Pi_{k}\Delta z\r|^2}{\l|B_{k+1}^{-1}\Pi_{k+1}\Delta z\r|^2}\Xi_{k+1,n}=\Xi_{k,n}
    \end{equation*}
    by the definition of $\Xi_{k,n}$. We thus obtain, from \eqref{eq: induction hypothesis multi-meeting proof} and the calculations above, that
    \begin{equation*}
        \P(Z_{n}\neq \tilde Z_{n})=\E\sqbr{\ind_{\{Z_{k-1}\neq\tilde Z_{k-1}\}}\l\{2\Phi\br{\frac{|E_k|}{2\Xi_{k,n}^{1/2}}}-1\r\}}.
    \end{equation*}
    Since $Z_{k-1}=\tilde Z_{k-1}$ implies that $2\Phi(|E_k|/(2\Xi_{k,n}^{1/2}))-1=0$, we can drop the indicator function and obtain
    \begin{equation*}
        \P(Z_{n}\neq \tilde Z_{n})=\E\sqbr{2\Phi\br{\frac{|E_k|}{2\Xi_{k,n}^{1/2}}}-1},
    \end{equation*}
    concluding the proof of \cref{thm: iterated one-shot coupling probability}.
    
\subsection{Examples}
We will now use \cref{thm: iterated one-shot coupling probability} to assess the performance of the iterated one-shot coupling in several settings. First, we will give a sufficient condition for when the coupling is optimal in the sense that \eqref{eq: multi-shot exact bound general} is equal to $d_\textup{TV}(\law(Z_n),\law(\tilde Z_n))$ and give an example where this condition is met. In addition, we will discuss the homogeneous setting with a contracting drift matrix and isotropic noise, where the iterated one-shot reproduces the TV distance up to a constant factor. Afterwards, we will turn to the exact discretization of the kinetic Langevin equation, which does not satisfy the aforementioned sufficient condition, and show that the coupling performance deteriorates when the step size $h$ decreases to $0$, in line with the third term on the right-hand side of~\eqref{eq:lb_couplingprob_disc_limsup}.

\medskip
\begin{example} 
\label{example: optimality condition of the multi-shot coupling}
   Consider the homogeneous case of \eqref{eq: linear Markov chain}, where the Markov chains are given by $Z_0=z$, $\tilde Z_0=\tilde z$ and
    \begin{equation}
    \label{eq: homogenous linear Markov chain}
        Z_{k+1}=AZ_k+B\xi_{k+1}, \quad \tilde Z_{k+1}=A\tilde Z_k+B\tilde\xi_{k+1},
    \end{equation}
    for some non-singular matrices $A,B\in\R^{d\times d}$. Then the iterated one-shot coupling introduced in Section~\ref{sssec: iterated one-shot} applied to \eqref{eq: homogenous linear Markov chain} is optimal, i.e.\, $\P(Z_n\neq \tilde Z_n)=d_\textup{TV}(\law(Z_n),\law(\tilde Z_n))$ if 
    \begin{equation*}
        ABB^T A^T=a^2BB^T
    \end{equation*}
    for some nonzero constant $a\in\R$. Indeed, in this case, the TV distance between $\law(Z_n)$ and $\law(\tilde{Z}_n)$ is given by
    \begin{equation}
        \label{eq: coupling probability optimal multi-shot setting}
        \dTV{\law(Z_n),\law(\tilde Z_n)}=2\Phi\br{\frac{|B^{-1}\Delta z|}{2\br{\sum_{k=1}^na^{-2k}}^{1/2}}}-1=\P(Z_n\neq \tilde Z_n),
    \end{equation}
    see equations~\eqref{eq: multi-shot optimal TV} and~\eqref{eq: interated one shot coupling prob spec} in \cref{appendix: sufficient condition optimality iterated one-shot coupling} (the first identity is based on the explicit expression for the TV distance between Gaussians, see \eqref{eq: Gaussian TV distance}; the second identity is based on the explicit expression for the meeting probability provided by Theorem~\ref{thm: iterated one-shot coupling probability}).
\end{example}

\medskip
\begin{remark}
    The assumption that $ABB^TA^T=a^2BB^T$ is equivalent to stating that $\frac{1}{a}B^{-1}AB$ is orthogonal, and is for example satisfied in the case where $\frac{1}{a}A$ is an orthogonal matrix that commutes with $BB^T$. 
\end{remark}

\medskip
\begin{remark}
\label{remark: optimality of the iterated one-shot coupling for inhomogeneous chains}
    The result of \cref{example: optimality condition of the multi-shot coupling} can be generalized to the iterated one-shot coupling introduced in Section~\ref{sssec: iterated one-shot} applied to the inhomogeneous linear Markov chains \eqref{eq: linear Markov chain} under the assumption that for each $k\in\{1,\dots,n-1\}$ there exists a nonzero constant $a_k\in\R$ such that
    \begin{equation*}
        A_{k+1}B_kB_k^TA_{k+1}^T=a_k^2B_{k+1}B_{k+1}^T,
    \end{equation*}
    which is equivalent to stating that $\frac{1}{a_k}B_{k+1}^{-1}A_{k+1}B_k$ is orthogonal.
\end{remark}

\medskip
\begin{example}
\label{example: inhomogeneous scaling chains}
    Consider the case where the Markov chains are of the form \eqref{eq: linear Markov chain} with $A_k=\bar\omega_k\id_d$ and $B_k=\sigma_k\id_d$ for nonzero constants $\bar\omega_k,\sigma_k\in\R$. In this case, we find that $\Theta_n=|\Delta z|^{-2}\sum_{k=1}^n\br{\sigma_k^2/\prod_{j=1}^k\bar\omega_j^2}$, so that in light of \cref{remark: optimality of the iterated one-shot coupling for inhomogeneous chains} we have
\begin{equation}
     \dTV{\law(Z_n),\law(\tilde Z_n)}=2\Phi\br{\frac{|\Delta z|}{2\Xi_{n}^{1/2}}}-1,
\end{equation}
where $\Xi_n=\sum_{k=1}^n\br{\sigma_k^2/\prod_{j=1}^k\bar\omega_j^2}$. This shows that the upper bound on the TV distance between two functional autoregressive chains of the form $Z_{k+1}=h_{k+1}(Z_k)+\sigma_{k+1}\xi_{k+1}$ provided by the iterated one-shot coupling as stated in Theorem 19 of \cite{DurmusMoulines:2019} is attained if all $h_k$ are of the form $h_{k}(x)=\bar\omega_k x$. 
\end{example}

One specific instance of the above example, where the upper bound on the TV distance derived in \cite{DurmusMoulines:2019} is actually an equality, is given by the setting of independent OU-processes with a common mean reversion rate:
\medskip
\begin{example}
    Let $(X_t)_{t\geq0}$ consist of $d$ independent OU processes with common mean reversion rate $\gamma>0$, i.e.\ it is the solution to the $d$-dimensional stochastic differential equation
    \begin{equation*}
        dX_t=-\gamma X_t\,dt+\sqrt{2\gamma}\,dW_t,
    \end{equation*}
    where $(W_t)_{t\geq 0}$ is a $d$-dimensional standard Brownian motion. The solution of this SDE for some initial value $X_0=x\in\R^d$ is given by
    \begin{equation}
    \label{eq: OU example exact solution}
        X_t=e^{-\gamma t}x+\sqrt{2\gamma}\int_0^te^{-\gamma(t-s)}\,dW_s.
    \end{equation}    
    Consider for some fixed step size $h>0$ the Markov chains $(X^h_k)_{k\in\N}$, $(\tilde X^h_k)_{k\in\N}$ defined inductively by $X^h_0=x$, $\tilde X^h_0=\tilde x$ and 
    \begin{equation*}
        X^h_{k+1}=e^{-\gamma h}X^h_k+(1-e^{-2\gamma h})^{1/2}\xi_{k+1}, \quad \tilde X^h_{k+1}=e^{-\gamma h}\tilde X^h_k+(1-e^{-2\gamma h})^{1/2}\tilde \xi_{k+1},
    \end{equation*}
    where $(\xi_k)_{k\in \N_{>0}}$ and $(\tilde{\xi}_k)_{k\in \N_{>0}}$ are two sequences of i.i.d.\ $\mathcal{N}(0,\id_d)$-distributed random variables. The distribution of each of these Markov chains is the same as that of the exact solution \eqref{eq: OU example exact solution} (with initial value $x$ and $\tilde x$ respectively) at times $t_k=hk$. This setting is an instance of \eqref{eq: homogenous linear Markov chain} with $A=e^{-\gamma h}\id_d$ and $B=(1-e^{-2\gamma h})^{1/2}\id_d$. In particular, we have that $ABB^T A^T=a^2BB^T$ with $a=e^{-\gamma h}$. \cref{example: optimality condition of the multi-shot coupling} thus guarantees that the iterated one-shot coupling is optimal, and if we apply it to $(X_k^h,\tilde X^h_k)_{k\in\N}$ we have for any $n\in\N$ that
    \begin{equation*}
        \dTV{\law(X^h_n),\law(\tilde X^h_n)}=\P(X^h_n\neq \tilde X_n^h)=2\Phi\br{\frac{(1-e^{-2\gamma h})^{-1/2}|\Delta x|}{2\br{\sum_{k=1}^ne^{2\gamma hk}}^{1/2}}}-1
    \end{equation*}
    by \eqref{eq: coupling probability optimal multi-shot setting}. The geometric series in this expression can be written as
    \begin{equation*}
        \sum_{k=1}^ne^{2\gamma hk}=\frac{e^{2\gamma h(n+1)}-e^{2\gamma h}}{e^{2\gamma h}-1}=\frac{e^{2\gamma hn}-1}{1-e^{-2\gamma h}}=\frac{e^{2\gamma T}-1}{1-e^{-2\gamma h}},
    \end{equation*}
    where $T=hn$. The TV distance is thus given by
    \begin{equation*}
        \dTV{\law(X^h_n),\law(\tilde X^h_n)}=\P(X^h_n\neq \tilde X_n^h)=2\Phi\br{\frac{|\Delta x|}{2\br{e^{2\gamma T}-1}^{1/2}}}-1.
    \end{equation*}
    As argued in \cref{example: inhomogeneous scaling chains}, this expression equals the upper bound given in Theorem 19 of \cite{DurmusMoulines:2019}.
\end{example}

The iterated one-shot coupling also performs well in more general homogeneous settings with a contracting drift matrix and isotropic noise, as the next example shows.

\medskip
\begin{example}
    Consider homogeneous linear Markov chains of the form \eqref{eq: homogenous linear Markov chain} with covariance matrix $BB^T=\sigma^2\id_d$ and a contracting drift matrix $A$ satisfying $\|A\|\leq\alpha<1$. The covariance after $n$ steps is $\Sigma_n=\sigma^2\sum_{k=1}^nA^k(A^T)^k$. In particular,
    \begin{equation*}
        \|\Sigma_n\|\leq\sigma^2\sum_{k=1}^n\|A\|^{2k}\leq\sigma^2\sum_{k=1}^n\alpha^{2k}=\sigma^2\frac{1-\alpha^{2n}}{\alpha^{-2}-1},
    \end{equation*}
    and since $\Sigma_n$ is symmetric, this shows that
    \begin{equation*}
        |\Sigma_n^{-1/2}A^n\Delta z|\geq\|\Sigma_n\|^{-1/2}|A^n\Delta z|\geq\frac{\sqrt{\alpha^{-2}-1}}{\sigma\sqrt{(1-\alpha^{2n})}}|A^n\Delta z|.
    \end{equation*}
    Let $(Z_k,\tilde Z_k)_{k\in \N}$ be given by the iterated one-shot coupling defined in Section~\ref{sssec: iterated one-shot} applied to the current setting. The factor $\Theta_n$ can be lower bounded as
    \begin{equation*}
        \Theta_n=\sigma^2\sum_{k=1}^n\frac{1}{|A^k\Delta z|^2}\geq\frac{\sigma^2}{|A^n\Delta z|^2},
    \end{equation*}
    so that
    \begin{equation*}
        \frac{1}{\Theta_n^{1/2}}\leq\frac{|A^n\Delta z|}{\sigma}\leq\frac{\sqrt{1-\alpha^{2n}}}{\sqrt{\alpha^{-2}-1}}|\Sigma_n^{-1/2}A^n\Delta z|.
    \end{equation*}
    Hence, \cref{thm: iterated one-shot coupling probability} shows that 
    \begin{equation*}
        \P(Z_n\neq\tilde Z_n)\leq2\Phi\br{\frac{\sqrt{1-\alpha^{2n}}|\Sigma_n^{-1/2}A^n\Delta z|}{2\sqrt{\alpha^{-2}-1}}}-1.
    \end{equation*} 
    If we compare this to the TV distance after $n$ steps,
    \begin{equation*}
        \dTV{\law(Z_n),\law(\tilde Z_n)}=2\Phi\br{\frac{|\Sigma_n^{-1/2}A^n\Delta z|}{2}}-1,
    \end{equation*}
    which follows from \eqref{eq: Gaussian TV distance} (see also \eqref{eq: Gaussian TV distance spec}), we see that for small values of $|\Sigma_n^{-1/2}A^n\Delta z|$, the probability of not meeting under the iterated one-shot coupling differs from the actual TV distance up to a factor of at most $1/\sqrt{\alpha^{-2}-1}$.   
\end{example}

The next example shows that the iterated one-shot coupling applied to the exact discretization of the potential free kinetic Langevin equation gives a probability of not meeting that deteriorates as the terminal time $T\coloneqq hn$ is fixed while decreasing the step size $h$. This is the same behavior as the third term in \eqref{eq:lb_couplingprob_disc_limsup} of \cref{thm: discrete Markovian couplings have the wrong asymptotic behavior}, showing that this part of the bound is sharp.

\medskip
\begin{example}\label{example: one shot matching lower bound}
    Consider the exact discretization of the potential free kinetic Langevin equation for $h>0$ and $k\in \N$ given by $Z^h_k=Z_{hk}$ and $\tilde{Z}^h_k =\tilde{Z}_{hk}$, where $(Z_t)_{t\geq0}$ and $(\tilde Z_t)_{t\geq0}$ denote the solutions to the linear system \eqref{eq: kinetic Langevin with quadratic potential} with $\alpha = 0$ and with initial values $z$ and $\tilde z$ respectively. This means that we have two homogeneous linear chains satisfying the recursive relation 
\begin{equation}
\label{eq: linear Markov chain 2}
    Z^h_{k+1}=A_{k+1}Z^h_k+B_{k+1}\xi_{k+1}, \quad \tilde Z^h_{k+1}=A_{k+1}\tilde Z^h_k+B_{k+1}\tilde\xi_{k+1},
\end{equation}
where $(\xi_k)_{k\in \N_{>0}}$ and $(\tilde{\xi}_k)_{k\in \N_{>0}}$ are two sequences of i.i.d.\ $\mathcal{N}(0,\id_{2d})$-distributed random variables and where $A_k=e^{hA}$ and $B_k=B_h$ for each $k\in\{1,\dots,n\}$, with $B_h\in\R^{2d\times 2d}$ being such that $B_hB_h^T=\Sigma_h$ and $e^{hA}$ and $\Sigma_h$ given by \eqref{eq: free kinetic Langevin mean and covariance matrices}. \cref{thm: iterated one-shot coupling probability} shows that if $(Z^h_k,\tilde Z^h_k)_{k\in\N}$ is given by the iterated one-shot coupling defined in Section~\ref{sssec: iterated one-shot}, then
    \begin{equation}
    \label{eq: multi-shot exact bound}
        \P(Z_n^h\neq \tilde Z_n^h)=2\Phi\br{\frac{1}{2\Theta_{n,h}^{1/2}}}-1,
    \end{equation}
    where 
    \begin{equation*}
        \Theta_{n,h}=\sum_{k=1}^n\frac{1}{\l|B_h^{-1}A_h^k\Delta z\r|^{2}}.
    \end{equation*}
    Suppose $\Delta z=(\Delta x,\Delta v)$ is such that $\gamma \Delta x+\Delta v=0$, so that $\Delta z$ is contained in the eigenspace of $A$ corresponding to the eigenvalue $-\gamma$. Then $e^{hA}\Delta z=e^{-\gamma h}\Delta z$, so that
    \begin{equation*}
        \Theta_{n,h}=\l|B_h^{-1}\Delta z\r|^{-2}\sum_{k=1}^ne^{2\gamma hk}=\l|B_h^{-1}\Delta z\r|^{-2}\frac{e^{2\gamma hn}-1}{1-e^{-2\gamma h}}.
    \end{equation*}   
    In order evaluate $|B_h^{-1}\Delta z|^{2}=\Delta z^T \Sigma_h^{-1}\Delta z$, let us write
    \begin{equation*}
        \Sigma_h=\begin{pmatrix}
            \sigma^{xx}_h\id_d & \sigma^{xv}_h\id_d \\  \sigma^{xv}_h\id_d & \sigma^{vv}_h\id_d
        \end{pmatrix}.
    \end{equation*}
    Recalling that $\Delta v=-\gamma \Delta x$ and using that $\sigma^{xx}_h,\sigma^{xv}_h>0$ we find that 
    \begin{equation*}
        \l|B_h^{-1}\Delta z\r|^{2}=\frac{(\sigma^{vv}_h+2\gamma\sigma^{xv}_h+\gamma^2\sigma^{xx}_h)}{\sigma^{xx}_h\sigma^{vv}_h-(\sigma^{xv}_h)^2}|\Delta x|^2\geq\frac{\sigma^{vv}_h}{\sigma^{xx}_h\sigma^{vv}_h}|\Delta x|^2=\frac{1}{\sigma^{xx}_h}|\Delta x|^2.       
    \end{equation*}
    We therefore have that
    \begin{equation*}
        \P(Z^h_n\neq \tilde Z^h_n)=2\Phi\br{\frac{1}{2\Theta_{n,h}^{1/2}}}-1\geq2\Phi\br{\frac{|\Delta x|}{2\sqrt{\sigma^{xx}_h}\sqrt{\frac{e^{2\gamma hn}-1}{1-e^{-2\gamma h}}}}}-1.
    \end{equation*}
    The expression on the right hand side is the same as that obtained in the final stage of the proof of \cref{lem: hitting prob when in eigenspace on grid}, where a lower bound is derived. With this lower bound we can conclude that there exists a constant $c>0$ (depending on $\gamma$ but not on $h$) such that
    \begin{equation*}
        \P(Z^h_n\neq \tilde Z^h_n)\geq c\min(1,h^{-1}e^{-\gamma hn}|\Delta z|).
    \end{equation*}
    This shows us that the probability of not meeting under the iterated one shot coupling behaves as the third term in \eqref{eq:lb_couplingprob_disc_limsup}. This result is no surprise, as \cref{lemma: e_k expression} shows that if $\Delta z$ is in the eigenspace of $A$ corresponding to the eigenvalue $-\gamma$, then so is $B_he_{k+1}$ for all $k\in\N$. In particular, $\Delta Z^h_k$ is contained in this eigenspace too for all $k\in\N$, and we thus have that $\Delta Q^h_k=0$ for all $k\in\N$, so that the iterated one-shot coupling satisfies the conditions of \cref{lem: hitting prob when in eigenspace on grid}. 
    
    We also observe that for the iterated one-shot coupling, the probability $\P(Z^h_n\neq \tilde Z^h_n)$ deteriorates for small values of $h$ regardless of the choice of initial values. Indeed, if we instead assume that $\Delta v=0$, so that $\Delta z$ is contained in the eigenspace of $A$ corresponding to the eigenvalue $0$, we find that
    \begin{equation*}
        \Theta_{n,h}=\sum_{k=1}^n\l|B_h^{-1}e^{hkA}\Delta z\r|^{-2}=|B_h^{-1}\Delta z|^{-2}=n\sigma^{xx}_h|\Delta z|^{-2},
    \end{equation*}
    so that a similar argument as before shows that there exists a constant $c>0$ such that
    \begin{equation*}
        \P(Z^h_n\neq \tilde Z^h_n)=2\Phi\br{\frac{1}{2\Theta_{n,h}^{1/2}}}-1\geq c \min(1,h^{-1}(hn)^{-1/2}|\Delta z|).
    \end{equation*}
    \end{example}

\newpage
\appendix
\section{Auxiliary inequalities}
\label{appendix: exponential bounds}

In order to analyze the performance of the coalescence map based on the optimized trajectory as discussed in \cref{subsec: optimized trajectory} we will need a few auxiliary bounds on terms involving exponential functions.

\medskip
\begin{lemma}
\label{lemma: auxiliary exponential bounds}
    For $x\geq0$ we have
    \begin{equation}
        \label{eq: exp bound lemma 1}
        \frac{x}{1+x}\leq 1-e^{-x} \leq\frac{2x}{2+x},
    \end{equation}
    \begin{equation}
        \label{eq: exp bound lemma 2}
        \frac{6-2x}{6+4x+x^2}\leq e^{-x}\leq\frac{6}{6+x^3},
    \end{equation}
    \begin{equation}
        \label{eq: exp bound lemma 3}
        0\leq\frac{1-e^{-2x}}{2x}-e^{-x}\leq\frac{2}{3}x^2.
    \end{equation}
    In addition, for $a\geq b\geq0$ we have
    \begin{equation}
        \label{eq: exp bound lemma 4}
        2ax(1+e^{-x})-4b(1-e^{-x})\geq \frac{2ax^3}{6+4x+x^2},
    \end{equation}
    \begin{equation}
        \label{eq: exp bound lemma 5}
        2ax-4b(1-e^{-x})+b(1-e^{-2x})\leq 2(a-b)x+\frac{4}{3}bx^3.
    \end{equation}
    \begin{equation}
        \label{eq: exp bound lemma 6}
        2ax-2b(1-e^{-x})-b(1-e^{-x})^2\geq\frac{24ax^3+13ax^4+2ax^5}{(6+4x+x^2)^2}.
    \end{equation}
\end{lemma}
\begin{proof}
    Rewriting $1+x\leq e^x$ readily gives the first inequality of \eqref{eq: exp bound lemma 1}. The second inequality of \eqref{eq: exp bound lemma 1} is equivalent to
    \begin{equation*}
        f(x)\coloneqq\br{1+\frac{x}{2}}e^{-x}+\frac{x}{2}-1\geq0.
    \end{equation*}
    Since $f(0)=0$ and
    \begin{equation*}
        f^\prime(x)=-\frac{1}{2}\br{1+x}e^{-x}+\frac{1}{2}\geq 0,
    \end{equation*}
    where the final inequality uses that $1+x\leq e^{x}$, we see that indeed $f(x)\geq0$ for all $x\geq 0$. The first inequality of \eqref{eq: exp bound lemma 2} is equivalent to
    \begin{equation*}
        \frac{(6+4x+x^2)e^{-x}-6+2x}{6+4x+x^2}\geq0, 
    \end{equation*}
    which is true if the numerator is nonnegative. We thus let $g(x)=(6+4x+x^2)e^{-x}-6+2x$ and note that $g(0)=0$ and
    \begin{equation*}
        g^\prime(x)=-(2+2x+x^2)e^{-x}+2=2\br{1-\br{1+x+\frac{1}{2}x^2}e^{-x}}\geq0,
    \end{equation*}
    which is nonnegative due to the fact that $1+x+\frac{1}{2}x^2\leq e^{x}$. This shows that $g(x)\geq0$ for all $x\geq0$, proving the first inequality of \eqref{eq: exp bound lemma 2}. The second inequality follows immediately from $e^x\geq 1+x+\frac{1}{2}x^2+\frac{1}{6}x^3\geq1+\frac{1}{6}x^3$ for $x\geq0$. The first inequality of \eqref{eq: exp bound lemma 3} follows from
    \begin{equation*}
        \frac{1-e^{-2x}}{2x}-e^{-x}=e^{-x}\br{\frac{e^{x}-e^{-x}}{2x}-1}\geq 0
    \end{equation*}
    using that $e^x-e^{-x}\geq 2x$ for $x\geq 0$. For the second inequality, we use the inequalities $e^{-x}\geq1-x+\frac{1}{2}x^2-\frac{1}{6}x^3$ and $e^{-x}\geq 1-x$, to obtain
    \begin{equation*}
        \frac{1-e^{-2x}}{2x}-e^{-x}\leq\frac{2x-2x^2+\frac{4}{3}x^3}{2x}-1+x=\frac{2}{3}x^2.
    \end{equation*}
    We see that \eqref{eq: exp bound lemma 4} is a consequence of the first inequality of \eqref{eq: exp bound lemma 2}, as
    \begin{align*}
        2ax(1+e^{-x})-4b(1-e^{-x})&=2ax-4b+(2ax+4b)e^{-x}\\&
        \geq2ax-4b+\frac{(2ax+4b)(6-2x)}{6+4x+x^2}\\
        &=\frac{24(a-b)x+4(a-b)x^2+2ax^3}{6+4x+x^2}\\
        &\geq\frac{2ax^3}{6+4x+x^2}.
    \end{align*}
    Moving on to \eqref{eq: exp bound lemma 5}, by the inequalities
    \begin{equation*}
        1-x+\frac{1}{2}x^2-\frac{1}{6}x^3\leq e^{-x}\leq1-x+\frac{1}{2}x^2,
    \end{equation*}
    we have $1-e^{-x}\geq x-\frac{1}{2}x^2$ and $1-e^{-2x}\leq 2x-2x^2+\frac{4}{3}x^3$. Applying these bounds and collecting terms of the same order gives
    \begin{equation*}
        2ax-4b(1-e^{-x})+b(1-e^{-2x})\leq 2(a-b)x+\frac{4}{3}bx^3,
    \end{equation*}
    showing \eqref{eq: exp bound lemma 5}. Finally, for \eqref{eq: exp bound lemma 6} we have by the first inequality of \eqref{eq: exp bound lemma 2} that
    \begin{equation*}
        1-e^{-x}\leq1-\frac{6-2x}{6+4x+x^2}\leq\frac{6x+x^2}{6+4x+x^2}.
    \end{equation*}
    Applying this twice we get
    \begin{align*}
        2ax-2b(1-e^{-x})-b(1-e^{-x})^2&\geq 2ax-2b\frac{6x+x^2}{6+4x+x^2}-b\frac{(6x+x^2)^2}{(6+4x+x^2)^2}\\
        &=\frac{12(a-b)x+(8a-2b)x^2+2ax^3}{6+4x+x^2}-\frac{36bx^2+12bx^3+bx^4}{(6+4x+x^2)^2}\\
        &\geq \frac{6ax^2+2ax^3}{6+4x+x^2}-\frac{36bx^2+12bx^3+bx^4}{(6+4x+x^2)^2}\\
        &\geq\frac{36(a-b)x^2+(36a-12b)x^3+(14a-b)x^4+2ax^5}{(6+4x+x^2)^2}\\
        &\geq\frac{24ax^3+13ax^4+2ax^5}{(6+4x+x^2)^2},
    \end{align*}
    showing \eqref{eq: exp bound lemma 6}. 
\end{proof}

\section{The optimized trajectory of the OBABO scheme}
In this section we derive an explicit expression of the position part of the optimized trajectory for the OBABO scheme as introduced in \cref{subsec: optimized trajectory}. This forms a crucial ingredient for the proof of \cref{lemma: potential part bound}. 

Throughout this section we fix $h>0$ and $n\in\N$. To improve readability, we will suppress the explicit notation of dependence on $h$ unless needed. We will also use the notation $\eta=e^{-\gamma h}$. 

\subsection{The OBABO mean and the covariance matrices}
\label{appendix: optimized trajectory calculation}

Recall that the matrices $A=A_h$, $L=L_h$ that govern the potential free part of the OBABO scheme are given by \eqref{eq: OBABO as linear Markov chain}. Denoting $\eta=e^{-\gamma h}$, these matrices are
\begin{equation*}
    A=\begin{pmatrix} \id_d & h\eta^{1/2}\id_d \\ 0 & \eta\id_d  \end{pmatrix}, \qquad  L=(1-\eta)^{1/2}\begin{pmatrix} h\id_d & 0 \\ \eta^{1/2}\id_d & \id_d  \end{pmatrix}.
\end{equation*}
We will now calculate the matrices $A^n$ and $\Sigma_n=\sum_{k=0}^{n-1}A^kLL^T(A^T)^k$ that govern respectively the mean and the covariance of the OBABO chain after $n$ steps if $\nabla U\equiv 0$. First, using the upper triangular structure of $A$, we have
\begin{equation}\
\label{eq: OBABO n step mean}
    A^n=\begin{pmatrix} \id_d & b_n\id_d \\ 0 & c_n\id_d  \end{pmatrix},
\end{equation}
where
\begin{equation*}
    b_n=h\eta^{1/2}\sum_{k=0}^{n-1}\eta^k=\frac{h\eta^{1/2}}{1-\eta}(1-\eta^n),\qquad c_n=\eta^n.
\end{equation*}
Consequently,
\begin{equation*}
    A^kLL^T(A^T)^k=(1-\eta)\begin{pmatrix} (h^2+2h\eta^{1/2}b_k+(1+\eta)b_k^2)\id_d & (h\eta^{1/2}c_k+(1+\eta)b_kc_k)\id_d \\ (h\eta^{1/2}c_k+(1+\eta)b_kc_k)\id_d & (1+\eta)c_k^2\id_d  \end{pmatrix},
\end{equation*}
so that the $n$ step covariance matrix is
\begin{equation}
\label{eq: OBABO n step covariance}
    \Sigma_{n}=\sum_{k=0}^{n-1}A^kLL^T(A^T)^k=\begin{pmatrix} \sigma^{xx}_{n}\id_d & \sigma^{xv}_{n}\id_d \\ \sigma^{xv}_{n}\id_d & \sigma^{vv}_{n}\id_d  \end{pmatrix},
\end{equation}
with
\begin{equation*}
    \begin{aligned}
        \sigma^{xx}_n&=(1-\eta)\sum_{k=0}^{n-1}(h^2+2h\eta^{1/2}b_k+(1+\eta)b_k^2),\\
        \sigma^{xv}_n&=(1-\eta)\sum_{k=0}^{n-1}(h\eta^{1/2}c_k+(1+\eta)b_kc_k),\\
        \sigma^{vv}_n&=(1-\eta)(1+\eta)\sum_{k=0}^{n-1}c_k^2.
    \end{aligned}
\end{equation*}
The various summations in these expressions are given by
\begin{align*}
    &\sum_{k=0}^{n-1}b_k=\frac{h\eta^{1/2}}{1-\eta}\sum_{k=0}^{n-1}(1-\eta^k)=\frac{h\eta^{1/2}}{1-\eta}\br{n-\frac{1-\eta^n}{1-\eta}},\\
    &\sum_{k=0}^{n-1}b_k^2=\frac{h^2\eta}{(1-\eta)^2}\sum_{k=0}^{n-1}(1-2\eta^k+\eta^{2k})=\frac{h^2\eta}{(1-\eta)^2}\br{n-2\frac{1-\eta^n}{1-\eta}+\frac{1-\eta^{2n}}{1-\eta^2}},\\
    &\sum_{k=0}^{n-1}c_k=\sum_{k=0}^{n-1}\eta^k=\frac{1-\eta^n}{1-\eta}\\
    &\sum_{k=0}^{n-1}b_kc_k=\frac{h\eta^{1/2}}{1-\eta}\sum_{k=0}^{n-1}(\eta^k-\eta^{2k})=\frac{h\eta^{1/2}}{1-\eta}\br{\frac{1-\eta^n}{1-\eta}-\frac{1-\eta^{2n}}{1-\eta^2}},\\
    &\sum_{k=0}^{n-1}c_k^2=\sum_{k=0}^{n-1}\eta^{2k}=\frac{1-\eta^{2n}}{1-\eta^2}.
\end{align*}
In particular, using the fact that $(1+\eta)(1-\eta)=(1-\eta^2)$ throughout, we have
\begin{align*}
    &\begin{aligned}
        \sigma^{xx}_{n}=&\frac{h^2}{1-\eta}\sqbr{\br{(1-\eta)^2+2(1-\eta)\eta+(1+\eta)\eta}n-2\eta\br{1+\frac{1+\eta}{1-\eta}}(1-\eta^n)+\frac{\eta(1+\eta)}{1-\eta^2}(1-\eta^{2n})}\\
        =&\frac{h^2}{1-\eta}\sqbr{\br{1+\eta}n-\frac{4\eta}{1-\eta}(1-\eta^n)+\frac{\eta}{1-\eta}(1-\eta^{2n})}\\
        =&\frac{h^2}{(1-\eta)^2}\sqbr{(1-\eta^2)n-4\eta(1-\eta^n)+\eta(1-\eta^{2n})},
    \end{aligned}\\
    &\begin{aligned}
        \sigma^{xv}_{n}&=h\eta^{1/2}\br{1+\frac{1+\eta}{1-\eta}}(1-\eta^n)-\frac{h\eta^{1/2}(1+\eta)}{1-\eta^2}(1-\eta^{2n})=\frac{h\eta^{1/2}}{1-\eta}\br{2(1-\eta^n)-(1-\eta^{2n})}\\
        &=\frac{h\eta^{1/2}}{1-\eta}(1-\eta^n)^2,
    \end{aligned}\\
    &\sigma^{vv}_{n}=\frac{(1-\eta)(1+\eta)}{1-\eta^2}(1-\eta^{2n})=1-\eta^{2n}.
\end{align*}

\subsection{The explicit optimized trajectory}\label{app: explicit optimized trajectory}
The trajectory $(y_0,\dots,y_n)$ defining the coalescence map $\Psi^n_{z,\tilde z}$ via \eqref{eq: coalescence map recursive relation} as proposed in \eqref{eq: y_k choices optimized trajectory} is given inductively by $y_0=\Delta z$ and
\begin{equation*}
    y_{k+1}=Ay_k-LE_{k+1},
\end{equation*}
where $E_k=L^T(A^T)^{n-k}\Sigma_{n}^{-1}A^n\Delta z$. By induction, it follows that
\begin{equation*}
    y_k=A^k\Delta z-\sum_{j=1}^kA^{k-j}LE_j=A^k\Delta z-\sum_{j=1}^kA^{k-j}LL^T(A^T)^{n-j}\Sigma_{n}^{-1}A^n\Delta z.
\end{equation*}
This can be rewritten into
\begin{align*}
    y_k&=A^k\Delta z-\underbrace{\br{{\sum_{j=1}^kA^{k-j}LL^T(A^T)^{k-j}}}}_{=\Sigma_{k}}(A^T)^{n-k}\Sigma_{n}^{-1}A^n\Delta z\\
    &=A^{k-n}\br{\Sigma_{n}-A^{n-k}\Sigma_{k}(A^T)^{n-k}}\Sigma_{n}^{-1}A^n\Delta z.
\end{align*}
The matrix inside the brackets reduces to
\begin{align*}
    \Sigma_{n}-A^{n-k}\Sigma_{k}(A^T)^{n-k}&=\sum_{j=1}^nA^{n-j}LL^T(A^T)^{n-j}-\sum_{j=1}^kA^{n-j}LL^T(A^T)^{n-j}\\
    &=\sum_{j=k+1}^nA^{n-j}LL^T(A^T)^{n-j}=\sum_{j=1}^{n-k}A^{n-k-j}LL^T(A^T)^{n-k-j}=\Sigma_{n-k}.
\end{align*}
In conclusion, the trajectory at step $n-k$ is
\begin{equation*}
    y_{n-k}=A^{-k}\Sigma_{k}\Sigma_{n}^{-1}A^n\Delta z.
\end{equation*}
We can evaluate this using the expressions \eqref{eq: OBABO n step mean} and \eqref{eq: OBABO n step covariance}. First, the inverses appearing in this expression are given by
\begin{equation*}
    A^{-k}=\begin{pmatrix} \id_d & -\frac{b_k}{c_k}\id_d \\ 0 & \frac{1}{c_k}\id_d \end{pmatrix},
\end{equation*}
and
\begin{equation*}
    \Sigma_{n}^{-1}=\frac{1}{\sigma^{xx}_{n}\sigma^{vv}_{n}-(\sigma^{xv}_{n})^2}\begin{pmatrix} \sigma^{vv}_{n}\id_d & -\sigma^{xv}_{n}\id_d \\ -\sigma^{xv}_{n}\id_d & \sigma^{xx}_{n}\id_d  \end{pmatrix}.
\end{equation*}
The part of $y_{n-k}$ that depends on $k$ is 
\begin{equation*}
    A^{-k}\Sigma_k=\begin{pmatrix}\br{\sigma^{xx}_k-\frac{b_k}{c_k}\sigma^{xv}_k}\id_d & \br{\sigma^{xv}_k-\frac{b_k}{c_k}\sigma^{vv}_k}\id_d \\ \frac{1}{c_k}\sigma^{xv}_k\id_d & \frac{1}{c_k}\sigma^{vv}_k\id_d \end{pmatrix},
\end{equation*}
where 
\begin{align*}
    \sigma^{xx}_{k}-\frac{b_k}{c_k}\sigma^{xv}_{k}&=\frac{h^2}{(1-\eta)^2}\sqbr{(1-\eta^2)k-4\eta(1-\eta^k)+\eta(1-\eta^{2k})-\eta\frac{(1-\eta^k)^3}{\eta^k}}\\
    &=\frac{h^2}{(1-\eta)^2}\sqbr{(1-\eta^2)k-\eta\br{3-4\eta^k+\eta^{2k}+\frac{1}{\eta^k}-3+3\eta^k-\eta^{2k}}}\\
    &=\frac{h^2}{(1-\eta)^2}\sqbr{(1-\eta^2)k-\frac{\eta}{\eta^k}\br{1-\eta^{2k}}},
\end{align*}
and
\begin{align*}
    \sigma^{xv}_{k}-\frac{b_k}{c_k}\sigma^{vv}_{k}=\frac{h\eta^{1/2}}{1-\eta}\sqbr{(1-\eta^k)^2-\frac{1}{\eta^{k}}(1-\eta^k)(1-\eta^{2k})}=-\frac{h\eta^{1/2}}{1-\eta}\frac{1}{\eta^k}(1-\eta^{k})^2.
\end{align*}
On the other hand,
\begin{align*}
    \Sigma_n^{-1}A_n\Delta z&=\frac{1}{\sigma^{xx}_{n}\sigma^{vv}_{n}-(\sigma^{xv}_{n})^2}\begin{pmatrix} \sigma^{vv}_{n}\id_d & -\sigma^{xv}_{n}\id_d \\ -\sigma^{xv}_{n}\id_d & \sigma^{xx}_{n}\id_d  \end{pmatrix}\begin{pmatrix}\Delta x+b_n\Delta v \\ c_n\Delta v \end{pmatrix}\\
    &=\frac{1}{\sigma^{xx}_{n}\sigma^{vv}_{n}-(\sigma^{xv}_{n})^2}\begin{pmatrix}\sigma^{vv}_n\br{\Delta x+b_n\Delta v}-\sigma^{xv}_nc_n\Delta v\\ -\sigma^{xv}_n\br{\Delta x+b_n\Delta v}+\sigma^{xx}_nc_n\Delta v \end{pmatrix},
\end{align*}

The position part of the trajectory at the $(n-k)th$ step is therefore given by
\begin{equation*}
    \begin{aligned}
        u_{n-k}=&\begin{pmatrix} \id_d & 0\end{pmatrix}y_{n-k}=\begin{pmatrix} \id_d & 0\end{pmatrix}A^{-k}\Sigma_{h,k}\Sigma_{h,n}^{-1}A^n\Delta z\\
        =&\frac{1}{\sigma^{xx}_{n}\sigma^{vv}_{n}-(\sigma^{xv}_{n})^2}\sqbr{\br{\sigma^{xx}_{k}-\frac{b_k}{c_k}\sigma^{xv}_{k}}\sigma^{vv}_{n}-\br{\sigma^{xv}_{k}-\frac{b_k}{c_k}\sigma^{vv}_{k}}\sigma^{xv}_{n}}(\Delta x+b_n\Delta v)\\
        &+\frac{1}{\sigma^{xx}_{n}\sigma^{vv}_{n}-(\sigma^{xv}_{n})^2}\sqbr{\br{\sigma^{xv}_{k}-\frac{b_k}{c_k}\sigma^{vv}_{k}}\sigma^{xx}_{n}-\br{\sigma^{xx}_{k}-\frac{b_k}{c_k}\sigma^{xv}_{k}}\sigma^{xv}_{n}}c_n\Delta v.
    \end{aligned}
\end{equation*}
The terms in the squared brackets are
\begin{equation}
\label{eq: alpha_k expression}
    \begin{aligned}
        \bigg(\sigma^{xx}_{k}&-\frac{b_k}{c_k}\sigma^{xv}_{k}\bigg)\sigma^{vv}_{n}-\br{\sigma^{xv}_{k}-\frac{b_k}{c_k}\sigma^{vv}_{k}}\sigma^{xv}_{n}\\
        &=\frac{h^2}{(1-\eta)^2}\sqbr{\br{(1-\eta^2)k-\frac{\eta}{\eta^k}\br{1-\eta^{2k}}}(1-\eta^{2n})+\frac{\eta}{\eta^k}(1-\eta^{k})^2(1-\eta^n)^2}\\
        &=\frac{h^2}{(1-\eta)^2}\sqbr{(1-\eta^2)(1-\eta^{2n})k-\frac{\eta}{\eta^k}\br{(1-\eta^{2k})(1-\eta^{2n})-(1-\eta^k)^2(1-\eta^n)^2}}\\
        &=\frac{h^2}{(1-\eta)^2}\sqbr{(1-\eta^2)(1-\eta^{2n})k-2\eta(1+\eta^{n-k})(1-\eta^k)(1-\eta^n)},
    \end{aligned}
\end{equation}
and 
\begin{align*}
    \bigg(\sigma^{xv}_{k}&-\frac{b_k}{c_k}\sigma^{vv}_{k}\bigg)\sigma^{xx}_{n}-\br{\sigma^{xx}_{k}-\frac{b_k}{c_k}\sigma^{xv}_{k}}\sigma^{xv}_{n}\\
    &\begin{aligned}
        =-\frac{h^3\eta^{3/2}}{(1-\eta)^3}\bigg[&\frac{1}{\eta^k}(1-\eta^{k})^2\br{(1-\eta^2)n-4\eta(1-\eta^n)+\eta(1-\eta^{2n})}\\
        &+\br{(1-\eta^2)k-\frac{\eta}{\eta^k}\br{1-\eta^{2k}}}(1-\eta^n)^2\bigg]\\
        =-\frac{h^3\eta^{3/2}}{(1-\eta)^3}\bigg[&(1-\eta^2)\br{\frac{1}{\eta^k}(1-\eta^k)^2n+(1-\eta^n)^2k}\\
        &-\frac{\eta}{\eta^k}(1-\eta^k)(1-\eta^n)\br{4(1-\eta^k)-(1-\eta^k)(1+\eta^{n})+(1+\eta^k)(1-\eta^n)}\bigg]\\
        =-\frac{h^3\eta^{3/2}}{(1-\eta)^3}\bigg[&(1-\eta^2)\br{\frac{1}{\eta^k}(1-\eta^k)^2n+(1-\eta^n)^2k}-2\frac{\eta}{\eta^k}(1-\eta^k)(1-\eta^n)(2-\eta^k-\eta^{n})\bigg].\\
    \end{aligned}
\end{align*}
As a shorthand, we introduce
\begin{equation}
\label{eq: alpha_k definition}
    \alpha_k\coloneqq (1-\eta^2)(1-\eta^{2n})k-2\eta(1+\eta^{n-k})(1-\eta^k)(1-\eta^n),
\end{equation}
\begin{equation}
\label{eq: beta_k definition}
    \beta_k\coloneqq(1-\eta^2)\br{\frac{1}{\eta^k}(1-\eta^k)^2n+(1-\eta^n)^2k}-2\frac{\eta}{\eta^k}(1-\eta^k)(1-\eta^n)(2-\eta^k-\eta^{n}).
\end{equation}
The relation \eqref{eq: alpha_k expression} reduces for $k=n$ to
\begin{equation*}
    \sigma^{xx}_{n}\sigma^{vv}_{n}-(\sigma^{xv}_{n})^2=\frac{h^2}{(1-\eta)^2}\alpha_n.
\end{equation*}
We therefore find the final expression
\begin{equation}\label{eq: expression for u_k}
    \begin{aligned}  
    u_{n-k}&=\frac{1}{\alpha_n}\br{\alpha_k\sqbr{\Delta x+\frac{h\eta^{1/2}}{1-\eta}(1-\eta^n)\Delta v}-\beta_k\frac{h\eta^{1/2}}{1-\eta}\eta^n\Delta v}\\
    &=\frac{\alpha_k}{\alpha_n}\Delta x-\frac{h\eta^{1/2}}{1-\eta}\frac{\eta^n\beta_k-(1-\eta^n)\alpha_k}{\alpha_n}\Delta v.
    \end{aligned}
\end{equation}

With this explicit expression, we can derive an uniform bound on $|u_{n-k}|$ as in \cref{lemma: potential part bound} by using the following result.
\medskip
\begin{lemma}
\label{lemma: alpha_k and beta_k are increasing}
    Both $(\alpha_k)_{0\leq k\leq n}$ and $(\beta_k)_{0\leq k\leq n}$ defined by \eqref{eq: alpha_k definition} and \eqref{eq: beta_k definition} are increasing.
\end{lemma}
\begin{proof}
     We first show that $(\alpha_k)_{0\leq k\leq n}$ is increasing. For $k\in\{0,\dots n-1\}$ we have 
    \begin{equation}
    \label{eq: a_k difference}
        \alpha_{k+1}-\alpha_k=(1-\eta)(1-\eta^n)\sqbr{(1+\eta)(1+\eta^{n})-2\eta(\eta^k+\eta^{n-k-1})},
    \end{equation}
    where we use that
    \begin{equation*}
        (1+\eta^{n-k-1})(1-\eta^{k+1})-(1+\eta^{n-k})(1-\eta^k)=(1-\eta)(\eta^{k}+\eta^{n-k-1}).
    \end{equation*}
    The term in the square brackets in \eqref{eq: a_k difference} can be expanded in rearranged as    
    \begin{align*}
        (1+\eta)(1+\eta^{n})&-2\eta(\eta^k+\eta^{n-k-1})=1+\eta+\eta^n+\eta^{n+1}-2\eta^{k+1}-2\eta^{n-k}\\
        &=(1-\eta^{k+1})+\eta(1-\eta^k)-\eta^{n-k}(1-\eta^k)-\eta^{n-k}(1-\eta^{k+1})\\
        &=(1-\eta^{k+1})(1-\eta^{n-k})+\eta(1-\eta^k)(1-\eta^{n-k-1})\geq0,
    \end{align*}
    since $\eta\in(0,1)$, showing that $(\alpha_k)_{0\leq k\leq n}$ is indeed increasing. For $(\beta_k)_{0\leq k\leq n}$, we again fix some $k\in\{0,\dots n-1\}$ and rewrite
    \begin{align*}
        \beta_k&=(1-\eta^2)\br{\frac{1}{\eta^k}(1-\eta^k)^2n+(1-\eta^n)^2k}-2\frac{\eta}{\eta^k}(1-\eta^k)^2(1-\eta^n)-2\frac{\eta}{\eta^k}(1-\eta^k)(1-\eta^n)^2\\
        &=\frac{1}{\eta^{k}}(1-\eta^k)^2\br{(1-\eta^2)n-2\eta(1-\eta^n)}+\br{(1-\eta^2)k-2\frac{\eta}{\eta^k}(1-\eta^k)}(1-\eta^n)^2,
    \end{align*}
    and note that
    \begin{equation*}
        \frac{1}{\eta^{k+1}}(1-\eta^{k+1})^2-\frac{1}{\eta^{k}}(1-\eta^{k})^2=\frac{1}{\eta^{k+1}}(1-\eta)(1-\eta^{2k+1}),
    \end{equation*}
    \begin{equation*}
        \frac{1}{\eta^{k+1}} (1-\eta^{k+1})-\frac{1}{\eta^{k}}(1-\eta^{k})=\frac{1}{\eta^{k+1}}(1-\eta).
    \end{equation*}
    We therefore have that
    \begin{equation*}
        \beta_{k+1}-\beta_k=\frac{1}{\eta^{k+1}}(1-\eta)\sqbr{(1-\eta^{2k+1})\br{(1-\eta^2)n-2\eta(1-\eta^n)}+\br{(1+\eta)\eta^{k+1}-2\eta}(1-\eta^n)}.
    \end{equation*}
    This expression can be rearranged using the fact that
    \begin{equation*}
        (1+\eta)\eta^{k+1}-2\eta=-\eta(2-\eta^k-\eta^{k+1})=-\eta\br{(1-\eta^k)(1-\eta^{k+1})+(1-\eta^{2k+1})},
    \end{equation*}
    to obtain
    \begin{equation}
    \label{eq: beta_k difference final equation}   
        \begin{aligned}
            \beta_{k+1}-\beta_k=\frac{1}{\eta^{k+1}}(1-\eta)\bigg[&(1-\eta^{2k+1})\br{(1-\eta^2)n-2\eta(1-\eta^n)-\eta(1-\eta^n)^2}\\
            &-\eta(1-\eta^k)(1-\eta^{k+1})(1-\eta^n)^2\bigg].
        \end{aligned}
    \end{equation}
    In order to bound this result from below, we use several of the inequalities of \cref{lemma: auxiliary exponential bounds}. First, recall that $\eta=e^{-\gamma h}$, so that $\frac{1-\eta^2}{2\gamma h}\geq\eta$ by \eqref{eq: exp bound lemma 3}, and so we can use \eqref{eq: exp bound lemma 6} with $a=\frac{1-\eta^2}{2\gamma h}$, $b=\eta$ and $x=\gamma hn$ to find
    \begin{equation}
        (1-\eta^2)n-2\eta(1-\eta^n)-\eta(1-\eta^n)^2\geq\br{\frac{1-\eta^2}{2\gamma h}}\frac{24(\gamma hn)^3+13(\gamma hn)^4+2(\gamma hn)^5}{\br{6+4\gamma hn+(\gamma hn)^2}^2}\geq0.
    \end{equation}
    In addition, 
    \begin{equation*}
        1-\eta^{2k+1}\geq1-\eta^{2k+1}=(1-\eta^k)(1+\eta^k),
    \end{equation*}
    and, by rewriting the first inequality of \eqref{eq: exp bound lemma 2},
    \begin{equation*}
        (1-\eta^n)^2\leq\frac{(6\gamma hn+(\gamma hn)^2)^2}{\br{6+4\gamma hn+(\gamma hn)^2}^2}=\frac{36(\gamma hn)^2+12(\gamma hn)^3+(\gamma hn)^4}{\br{6+4\gamma hn+(\gamma hn)^2}^2}.
    \end{equation*}
    Combining these observations, we see from \eqref{eq: beta_k difference final equation} that $\beta_{k+1}-\beta_k\geq0$ if
    \begin{equation*}
        \br{\frac{1-\eta^2}{2\gamma h}}(1+\eta^k)\frac{24(\gamma hn)^3+13(\gamma hn)^4+2(\gamma hn)^5}{\br{6+4\gamma hn+(\gamma hn)^2}^2}-\eta(1-\eta^{k+1})\frac{36(\gamma hn)^2+12(\gamma hn)^3+(\gamma hn)^4}{\br{6+4\gamma hn+(\gamma hn)^2}^2}\geq0.
    \end{equation*}
    In particular, since $\frac{1-\eta^2}{2\gamma h}\geq\eta$, and
    \begin{equation*}
        24(\gamma hn)^3+13(\gamma hn)^4+2(\gamma hn)^5\geq \frac{2}{3}\gamma hn\br{36(\gamma hn)^2+12(\gamma hn)^3+(\gamma hn)^4},
    \end{equation*}
    it suffices to show that
    \begin{equation*}
        \frac{2}{3}(1+\eta^k)\gamma hn\geq1-\eta^{k+1}.
    \end{equation*}
    In fact, we will show the stronger inequality
    \begin{equation*}
        \frac{1}{2}(1+\eta^k)\geq\frac{1-\eta^{k+1}}{\gamma h(k+1)},
    \end{equation*}
    which suffices since $n\geq k+1$. Using the first inequality of \eqref{eq: exp bound lemma 2} twice get both
    \begin{equation*}
        \frac{1}{2}(1+\eta^k)\geq\frac{6+\gamma hk}{6+4\gamma hk+(\gamma hk)^2},
    \end{equation*}
    and
    \begin{equation*}
        \frac{1-\eta^{k+1}}{\gamma h(k+1)}\leq\frac{6+\gamma h(k+1)}{6+4\gamma h(k+1)+(\gamma h(k+1))^2}.
    \end{equation*}
    By combining these two fractions we get
    \begin{equation*}
        \frac{1}{2}(1+\eta^k)-\frac{1-\eta^{k+1}}{\gamma h(k+1)}\geq\frac{18\gamma h+6(\gamma h)^2(2k+1)+(\gamma h)^3k(k+1)}{(6+4\gamma hk+(\gamma hk)^2)(6+4\gamma h(k+1)+(\gamma h(k+1))^2)}\geq0,
    \end{equation*}
    showing the claimed result. We conclude that $(\beta_k)_{0\leq k\leq n}$ is indeed increasing.
\end{proof}

\section{Proof of Lemma~\ref{lemma: potential free part bound}}
\label{appendix: potential free bound}
We begin by recalling that~\eqref{eq: identity for E_k^2} establishes the following identity:
\begin{equation*}
    \sum_{k=1}^n|E_k|^2=\l|\Sigma_{h,n}^{-1/2}A_h^n\Delta z\r|^2.
\end{equation*}
We will show that the right-hand side can be bounded appropriately in terms of the terminal time $hn$. For the sake of notational brevity, in what follows we only acknowledge the dependence on $h$ explicitly in our notation when needed, and use the notation $\eta=e^{-\gamma h}$ throughout.
\begin{proposition}
\label{prop: potential free operator norm bound}
For $h>0$, $n\in\N$, let $A_n$ and $\Sigma_{n}$ given by \eqref{eq: OBABO n step mean} and \eqref{eq: OBABO n step covariance}. Then we have the upper bound
    \begin{equation*}
        \l\|\Sigma_{n}^{-1/2}A^n\r\|\leq\br{\frac{44}{\gamma (hn)^3}+\frac{264+44\gamma^2}{\gamma hn}}^{1/2}.
    \end{equation*}
\end{proposition}
\begin{proof}
The norm of $\Sigma_{n}^{-1/2}A^n$ is given by its largest singular value:
    \begin{equation}
    \l\|\Sigma_{n}^{-1/2}A^n\r\|=\sqrt{\lambda_\textup{max}\br{(A^n)^T\Sigma_{n}^{-1}A^n}}.
\end{equation}
The matrices $A^n$ and $\Sigma_{n}$ are, up to the same reordering of the basis vectors, block matrices consisting of $d$ identical $2\times 2$ blocks of the form
\begin{equation*}
    \begin{pmatrix}
        1 & b_n\\ 0 & c_n
    \end{pmatrix}, \qquad \begin{pmatrix}
        \sigma_n^{xx} & \sigma_n^{xv} \\ \sigma_n^{xv} & \sigma_n^{vv}
    \end{pmatrix},
\end{equation*}
respectively. In particular, the eigenvalues of $(A^n)^T\Sigma_{n}^{-1}A^n$ are the same as the eigenvalues obtained from these blocks, so that
\begin{align}
\label{eq: norm upper bound via eigenvalue and trace}
    \lambda_\textup{max}\br{(A^n)^T\Sigma_{n}^{-1}A^n}&=\lambda_\textup{max}\br{\frac{1}{\sigma^{xx}_n\sigma^{vv}_n-(\sigma^{xv}_n)^2}\begin{pmatrix}
        \sigma^{vv}_n & b_n\sigma^{vv}_n-c_n\sigma^{xv}_n\\ b_n\sigma^{vv}_n-c_n\sigma^{xv}_n & b_n^2\sigma^{vv}_n-2b_nc_n\sigma^{xv}_n+c_n^2\sigma^{xx}_n
    \end{pmatrix}}\nonumber\\
    &\leq \frac{(1+b_n^2)\sigma^{vv}_n-2b_nc_n\sigma^{xv}_n+c_n^2\sigma^{xx}_n}{\alpha_n}\leq\frac{(1+b_n^2)\sigma^{vv}_n+c_n^2\sigma^{xx}_n}{\alpha_n},
\end{align}
using that the trace of a matrix gives an upper bound on its eigenvalues and the fact that $b_n,c_n,\sigma^{xv}_n\geq0$, and where $\alpha_n$ is given by \eqref{eq: alpha_k definition}. What remains is bounding the terms on the right hand side of this inequality. First, we rewrite the denominator as
\begin{align}
\label{eq: covariance matrix determinant}
    \alpha_n&=\frac{h^2}{(1-\eta)^2}\sqbr{(1-\eta^2)(1-\eta^{2n})n-4\eta(1-\eta^n)^2}\nonumber\\
    &=\frac{h^2}{(1-\eta)^2}(1-\eta^n)\sqbr{(1-\eta^2)(1+\eta^n)n-4\eta(1-\eta^n)}\nonumber\\
    &=\frac{h^2}{(1-\eta)^2}(1-\eta^n)\sqbr{2\frac{(1-\eta^2)}{2\gamma h}(1+\eta^n)\gamma hn-4\eta(1-\eta^n)}
\end{align}

Combining the first inequality of \eqref{eq: exp bound lemma 1} and \eqref{eq: exp bound lemma 4} from \cref{lemma: auxiliary exponential bounds} shows that for all $a\geq b\geq 0$ and $x\geq0$ we have 
\begin{equation*}
    (1-e^{-x})\sqbr{2ax(1+e^{-x})-4b(1-e^{-x})}\geq \frac{24(a-b)x^2+4(a-b)x^3+2ax^4}{(1+x)(6+4x+x^2)}.
\end{equation*}
where 
\begin{equation*}
    (1+x)(6+4x+x^2)=6+10x+5x^2+x^3\leq22(1+x^3).
\end{equation*}
The first inequality of \eqref{eq: exp bound lemma 3} shows that $\frac{1-\eta^2}{2\gamma h}\geq\eta$, so that we can apply the above inequalities to \eqref{eq: covariance matrix determinant} by setting $a=\frac{1-\eta^2}{2\gamma h}$, $b=\eta$ and $x=\gamma hn$ and recalling that $\eta^n=e^{-\gamma hn}$. We therefore obtain
\begin{equation*}
    \alpha_n\geq\frac{1}{11}\frac{h^2}{(1-\eta)^2}\frac{1-\eta^2}{2\gamma h}\frac{(\gamma hn)^4}{1+(\gamma hn)^3}.
\end{equation*}
This bound can then be extended in two ways. First, since $\frac{1-\eta^2}{2\gamma h}\geq\eta$, we have
\begin{equation}
\label{eq: alpha_n lower bound 1}
    \alpha_n\geq\frac{1}{11}\frac{h^2\eta}{(1-\eta)^2}\frac{(\gamma hn)^4}{1+(\gamma hn)^3}.
\end{equation}
Secondly, we have
\begin{equation*}
    \frac{h^2}{(1-\eta)^2}\frac{1-\eta^2}{2\gamma h}=\frac{h(1+\eta)}{2\gamma(1-\eta)}\geq\frac{h}{2\gamma(1-\eta)}\geq\frac{1}{2\gamma^2},
\end{equation*}
where the final inequality uses that $1-\eta=1-e^{-\gamma h}\leq\gamma h$. Hence,
\begin{equation}
\label{eq: alpha_n lower bound 2}
    \alpha_n\geq\frac{1}{22\gamma^2}\frac{(\gamma hn)^4}{1+(\gamma hn)^3}.
\end{equation}
For the next term, we use the second inequality of \eqref{eq: exp bound lemma 1} to see that
\begin{equation*}
    b_n^2=\frac{h^2\eta}{(1-\eta)^2}(1-\eta^n)^2\leq\frac{h^2\eta}{(1-\eta)^2}\frac{4(\gamma hn)^2}{(2+\gamma hn)^2},
\end{equation*}
while by the same inequality of \eqref{eq: exp bound lemma 1} we have
\begin{equation*}
    \sigma^{vv}_n=1-\eta^{2n}\leq\frac{2\gamma hn}{1+\gamma hn}.
\end{equation*}
Combining these, together with the fact that $(2+\gamma hn)^2(1+\gamma hn)\geq 1+(\gamma hn)^3$, results in
\begin{equation*}
    \frac{b_n^2\sigma^{vv}_n}{\alpha _n}\leq\frac{88}{\gamma hn},
\end{equation*}
where we use the lower bound \eqref{eq: alpha_n lower bound 1}. In addition, using the lower bound \eqref{eq: alpha_n lower bound 2}, we see
\begin{equation*}
    \frac{\sigma^{vv}_n}{\alpha_n}\leq\frac{44\gamma^2}{(\gamma hn)^3}\frac{1+(\gamma hn)^3}{1+\gamma hn}\leq\frac{44\gamma^2}{(\gamma hn)^3}(1+(\gamma hn)^2)=\frac{44}{\gamma (hn)^3}+\frac{44\gamma}{hn}.
\end{equation*}
For the final term, we have by the second inequality of \eqref{eq: exp bound lemma 2} that
\begin{equation*}
    c_n^2=\eta^{2n}=\eta^ne^{-\gamma h n}\leq \frac{6\eta^n}{6+(\gamma hn)^3}\leq\frac{6\eta}{1+(\gamma hn)^3}.
\end{equation*}
Also, inequality \eqref{eq: exp bound lemma 5} with $a=\frac{1-\eta^2}{2\gamma h}$, $b=\eta\leq1$ and $x=\gamma hn$, where $a-b\leq\frac{2}{3}(\gamma h)^2\leq\frac{2}{3}(\gamma hn)^2$ by the second inequality of \eqref{eq: exp bound lemma 3}, gives
\begin{equation*}
    \sigma^{xx}_n=\frac{h^2}{(1-\eta)^2}\sqbr{(1-\eta^2)n-4\eta(1-\eta^n)+\eta(1-\eta^{2n})}\leq\frac{8}{3}\frac{h^2}{(1-\eta)^2}(\gamma hn)^3.
\end{equation*}
Together, using \eqref{eq: alpha_n lower bound 1}, we see that
\begin{equation*}
    \frac{c_n^2\sigma^{xx}_n}{\alpha_n}\leq \frac{176}{\gamma hn}.
\end{equation*}
In conclusion, when we put the various bounds above into \eqref{eq: norm upper bound via eigenvalue and trace}, we obtain
\begin{equation*}
    \l\|\Sigma_{n}^{-1/2}A^n\r\|^2\leq \frac{44}{\gamma (hn)^3}+\frac{264+44\gamma^2}{\gamma hn},
\end{equation*}
completing the proof
\end{proof}

\section{Sufficient condition for optimality of the iterated one-shot coupling}
\label{appendix: sufficient condition optimality iterated one-shot coupling}

This section contains proofs for the identities in equation~\eqref{eq: coupling probability optimal multi-shot setting} of \cref{example: optimality condition of the multi-shot coupling}. Indeed, let $A,B\in\R^{d\times d}$ be non-singular matrices and consider the homogeneous linear Markov chains given by $Z_0=z$, $\tilde Z_0=\tilde z$ and
\begin{equation*}
    Z_{k+1}=AZ_k+B\xi_{k+1}, \quad \tilde Z_{k+1}=A\tilde Z_k+B\tilde\xi_{k+1},
\end{equation*}
where $(\xi_k)_{k\in \N_{>0}}$ and $(\tilde{\xi}_k)_{k\in \N_{>0}}$ are two sequences of i.i.d.\ $\mathcal{N}(0,\id_d)$-distributed random variables. By an induction argument it follows that
\begin{equation*}
    Z_n=A^nz+\sum_{k=1}^n A^{n-k}B\xi_k.
\end{equation*}
Consequently, $Z_n\sim\mathcal{N}(A^n z,\Sigma_n)$, where 
\begin{equation*}
    \Sigma_n=\sum_{k=0}^{n-1} A^{k}BB^T(A^{T})^{k}.
\end{equation*}
The TV distance between the two chains after $n$ steps thus follows from \eqref{eq: Gaussian TV distance} and is given by
\begin{equation}\label{eq: Gaussian TV distance spec}
    \dTV{\law(Z_n),\law(\tilde Z_n)}=2\Phi\br{\frac{|\Sigma_n^{-1/2}A^n\Delta z|}{2}}-1.
\end{equation}
Let us now assume that $ABB^TA^T=a^2BB^T$ for some nonzero $a\in\R$. Under this assumption, the covariance of the chain at each step reduces to a positive multiple of $BB^T$:
\begin{equation*}
    \Sigma_n=\sum_{k=0}^{n-1} A^{k}BB^T (A^T)^{k}=\br{\sum_{k=0}^{n-1} a^{2k}}BB^T.
\end{equation*}
Moreover, by the non-singularity of $A$, the assumption that $ABB^TA^T=a^2BB^T$ is equivalent to
\begin{equation}
\label{eq: multi-shot converse optimality assumption}
    A^T(BB^T)^{-1}A=a^2(BB^T)^{-1}.
\end{equation}
Applying both these observations, we deduce that
\begin{align*}
    (A^T)^n\Sigma_n^{-1}A^n&=\br{\sum_{k=0}^{n-1}a^{2k}}^{-1}(A^T)^n(BB^T)^{-1} A^n=\br{\sum_{k=0}^{n-1}a^{2k}}^{-1}a^{2n}(BB^T)^{-1}\\
    &=\br{\sum_{k=1}^{n}a^{-2k}}^{-1}(BB^T)^{-1}.
\end{align*}
The term $|\Sigma_n^{-1/2}A^n\Delta z|$, as it appears in~\eqref{eq: Gaussian TV distance spec}, can therefore be simplified to
\begin{equation*}
    |\Sigma_n^{-1/2}A^n\Delta z|^2=\br{\sum_{k=1}^na^{-2k}}^{-1}\Delta z^T(BB^T)^{-1}\Delta z=\br{\sum_{k=1}^na^{-2k}}^{-1}|B^{-1}\Delta z|^2,
\end{equation*}
so that the TV distance is given by
\begin{equation}
\label{eq: multi-shot optimal TV}
    \dTV{\law(Z_n),\law(\tilde Z_n)}=2\Phi\br{\frac{|B^{-1}\Delta z|}{2\br{\sum_{k=1}^na^{-2k}}^{1/2}}}-1.
\end{equation}
On the other hand, the probability that the two chains have met under the iterated one-shot coupling is, due to \cref{thm: iterated one-shot coupling probability}, given by
\begin{equation*}
    \P(Z_n\neq \tilde Z_n)=2\Phi\br{\frac{1}{2\Theta_{n}^{1/2}}}-1,
\end{equation*}
where 
\begin{equation*}
    \Theta_{n}=\sum_{k=1}^n\frac{1}{\l|B^{-1}A^k\Delta z\r|^{2}}.
\end{equation*}
Under \eqref{eq: multi-shot converse optimality assumption} this value can be simplified to
\begin{equation*}
    \Theta_{n}=\sum_{k=1}^n\frac{1}{\l|B^{-1}A^k\Delta z\r|^{2}}=\frac{1}{\Delta z^T(BB^T)^{-1}\Delta z}\sum_{k=1}^na^{-2k}=\frac{1}{|B^{-1}\Delta z|^2}\sum_{k=1}^na^{-2k}.
\end{equation*}
This shows that
\begin{equation}\label{eq: interated one shot coupling prob spec}
    \P(Z_n\neq \tilde Z_n)=2\Phi\br{\frac{|B^{-1}\Delta z|}{2\br{\sum_{k=1}^na^{-2k}}^{1/2}}}-1,
\end{equation}
which is equal to $\dTV{\law(Z_n),\law(\tilde Z_n)}$ given in \eqref{eq: multi-shot optimal TV}. 
\newpage

\printbibliography[heading=bibintoc,title={References}]

\end{document}